\documentclass{amsart}

\usepackage[english]{babel}
\usepackage{palatino}

\usepackage{amsmath,amsthm,amsfonts,amssymb}
\usepackage{bbm}
\usepackage{graphicx}
\usepackage[colorinlistoftodos]{todonotes}
\usepackage[colorlinks=true]{hyperref}
\usepackage{geometry}
\usepackage{spalign}
\usepackage{xfrac} 
\usepackage{float}
\usepackage{tikz-cd} 

\DeclareMathOperator{\sgn}{sgn}

\theoremstyle{plain}
\newtheorem{thm}{Theorem}[subsection]
\newtheorem{cor}[thm]{Corollary}

\newtheorem{prop}[thm]{Proposition}

\newtheorem{lem}[thm]{Lemma}
\newtheorem{claim}[thm]{Claim}
\newtheorem{fact}[thm]{Fact}

\theoremstyle{definition}
\newtheorem{defi}[thm]{Definition}

\theoremstyle{remark}
\newtheorem{remark}[thm]{Remark}
\newtheorem{example}[thm]{Example}

\newtheorem{question}[thm]{Question}

\title{ Lifting Branched covers to Braided Embeddings}
\author{Sudipta Kolay}
\address{School of Mathematics \\ Georgia Institute of Technology\\Atlanta, GA 30332, USA}
\email{skolay3@math.gatech.edu}
\date{}

\begin{document}
\maketitle

\begin{abstract} Braided embeddings are embeddings to a product disc bundle so the projection to the first factor is a branched cover. In this paper, we study which branched covers lift to braided embeddings, which is a generalization of the Borsuk-Ulam problem. We determine when a braided embedding in the complement of branch locus can be extended over the branch locus in smoothly (or locally flat piecewise linearly), and use it in conjunction with Hansen's criterion for lifting covers. We show that every branched cover over an orientable surface lifts to a codimension two braided embedding in the piecewise linear category, but there are non-liftable branched coverings in the smooth category. We  explore the liftability question for covers over the Klein bottle. In dimension three, we consider simple branched coverings over the three sphere, branched over two-bridge, torus and pretzel knots, obtaining infinite families of examples where the coverings do and do not lift. Finally, we also discuss some examples of non-liftable branched covers in higher dimensions. 

\end{abstract}

\section{Introduction}

A classical theorem of Alexander \cite{A1} states that every link in three space is isotopic to a closed braid. A theorem of Markov \cite{Mv} tells us that by stabilizations and conjugations we can go between any two braid closures which are isotopic.

The above results allows us to study knots and links (topological objects) from the point of view of braids (algebraic and combinatorial objects). This viewpoint has been helpful in the study of knot theory, especially in the construction of knot invariants.

One can hope that higher dimensional braids will play a similar role in higher dimensional knot theory. Braids have been generalized in higher dimensions in several different ways, our point of view is a natural one in studying embeddings of manifolds. In fact what we are calling braided embeddings have been studied by various mathematicians, using slightly different names: (branched) polynomial coverings by Hansen \cite{Ha,Ha4}, $d$-fat covers by Petersen~\cite{Pe}, folded embeddings by Carter and Kamada \cite{CK}, and sometimes without explicitly using a name such as in the work of Hilden-Lozano-Montesinos \cite{HLM}. Hansen~\cite{Ha,Ha1,Ha2,Ha3} extensively studied the lifting problem for honest coverings, giving a complete characterization of when a codimension two lift exists, and several results in higher codimension (see also the related work of Petersen~\cite{Pe}). Hansen~\cite{Ha4} also considered lifting branched coverings, although in a slightly different (less restrictive) setting than ours. Rudolph \cite{Ru}, Viro and Kamada \cite{Ka1,Ka2} studied braided surfaces in $\mathbb{R}^4$, obtaining analogues for Alexander's braiding theorem for surfaces. Carter and Kamada studied braided embeddings and immersions in low dimensions, and talked about existence and lifting problems.
Etnyre and Furukawa \cite{EF} studied braided embeddings in all dimensions, and focused on their interplay with contact embeddings. The author was able to show \cite{KAlex} the an analogue of Alexander's theorem in the piecewise linear category, about isotoping every closed oriented three manifold in $\mathbb{R}^5$ to be a braided embedding. The goal of this article will be to study the lifting problem, primarily in codimension two for branched cover over spheres.

We work in both the smooth and piecewise linear category linear categories. We will not mention the categories separately in case they behave similarly, however when necessary we will deal with them separately.

\subsection{Braided Embeddings}
 We say that an embedding $f:M\rightarrow N\times D^l$ is a \textit{co-dimension $l$ braided embedding} over $N$  if the embedding composed with the projection to $N$, $pr_1\circ f:M\rightarrow N$ is a (oriented in case both base and covering spaces are oriented) branched covering map.
\begin{remark} If we do not specify the co-dimension, a braided embedding would be a co-dimension two embedding of $M$ to $N\times D^2$.
\end{remark}

\begin{question}\label{QA}
Can every branched cover be lifted to a co-dimension $l$ braided embedding?
\end{question}

Since we are always working with compact manifolds, by Whitney embedding theorem \cite{W}, they always embed in a sufficiently high dimensional disc, and hence for sufficiently large $l$, any branched covering lifts to a codimension $l$ braided embedding (because it embeds in the disc factor). Hence it makes sense to ask the following:

\begin{question}\label{QB}

What is the smallest $l$ so that a given branched cover can be lifted to a co-dimension $l$ braided embedding?
\end{question}

Let us start begin by discussing the this question for some well known families of honest covering maps.
\begin{example}\label{egcircle}
    $p_n:S^1\rightarrow S^1$ defined by $z\mapsto z^n$, for $n\in \mathbb{N}$ (similar result holds for negative integers). For $n=1$, the map $p_1$ is the identity and hence an embedding, so $l=0$ in this case. For $n>1$, we claim that the smallest such $l$ is 2. It is clear that there are codimension two lifts. That there is no codimension one lift, is illustrated in Figure \ref{annbu} for the case when $n=3$. 
    \begin{figure}[!ht]
    \centering
    \includegraphics[width=7 cm]{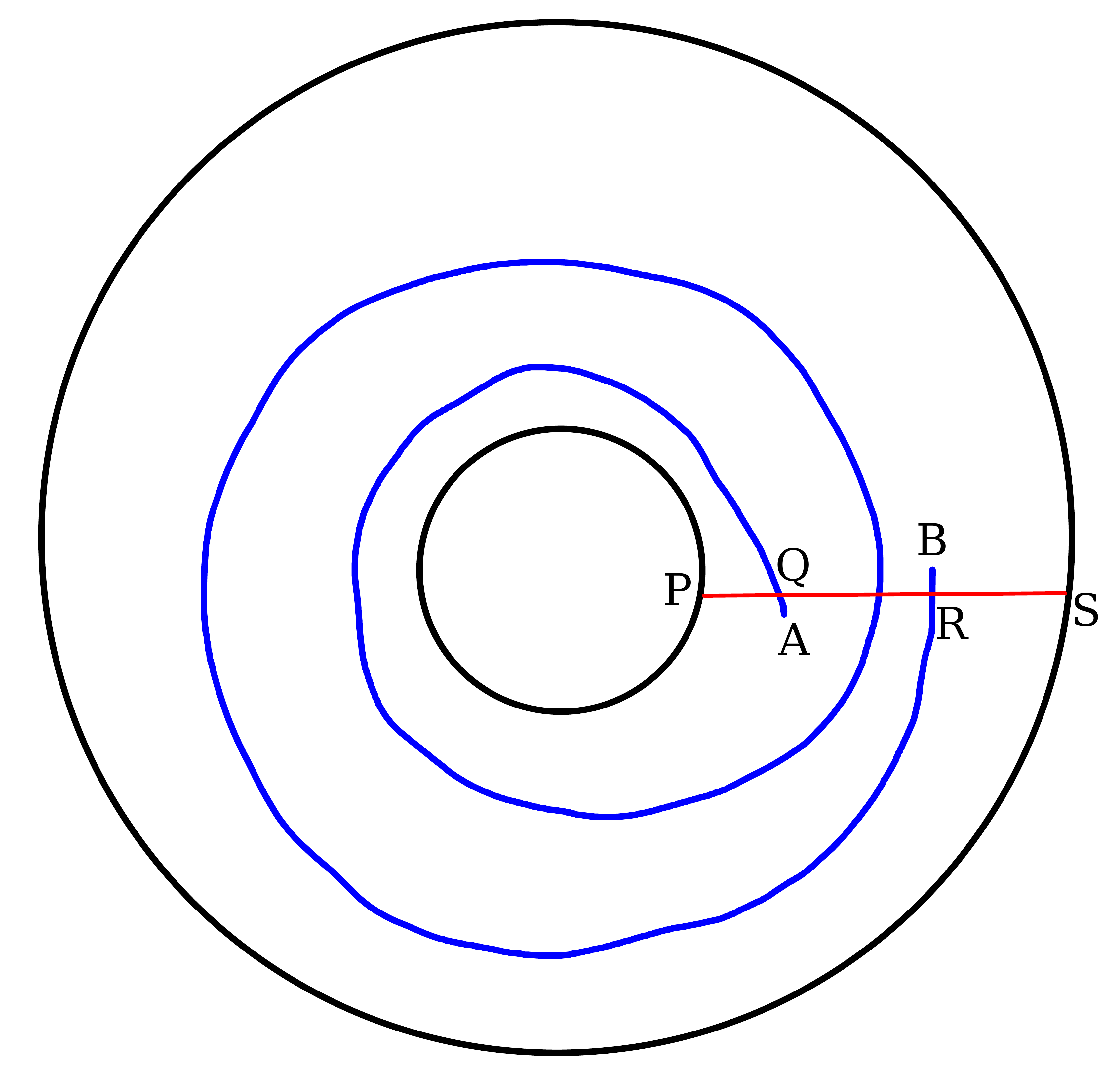}
    \caption{Suppose $n=3$, and we have a codimension 1 braided embedding of $S^1$ over $S^1$. Suppose we have drawn a part of the image starting at $A$ and ending at $B$, together with a transversal (constant slice of the annulus) meeting the image in three points. Since the leftmost (PQ) and rightmost (RS) segment of the transversal together with the segment along the curve from $Q$ to $R$ disconnects the annulus, there is no way to join $B$ to $A$ without causing an intersection.}
    \label{annbu}
\end{figure}
    
    Alternately, this statement also follows as a consequence of the intermediate value theorem. 
\end{example}

   \begin{example} 
  $a_n:S^n\rightarrow \mathbb{RP}^n$ induced by quotienting out by the antipodal action. In this case the smallest such $l$ turns out to be $n+1$.
    To see this note that since $S^n$ embeds in the disc $D^{n+1}$ (being the boundary of the closed disc), we have a braided embedding defined by $a_n\times i:S^n\rightarrow \mathbb{RP}^n\times D^{n+1}$ (where $i$ is some embedding of  $S^n$ in $D^{n+1}$). That $l$ cannot be any less than $n$ is precisely the content of the Borsuk Ulam theorem \cite{Bo}.
    
\end{example}

So the lifting problem is a generalization of the Borsuk Ulam theorem, where instead of allowing two-fold covers, we allow arbitrary branched covers.
The reason for allowing branched coverings instead of just coverings is that typically we will mostly be looking at (branched) coverings over the sphere, and since for $n>1$, the $n$-sphere is simply connected, in order to get interesting manifolds, we need to allow branching. For this paper we are going to focus on co-dimension two liftings of branched covers over spheres, because knots in codimension two turn out to be most interesting.

Let us discuss some fundamental results about branched coverings and embeddings to motivate our discussion of lifting branched covers to braided embeddings.
\subsection{Manifolds branched over spheres}
One can hope to understand all closed oriented manifolds of a given dimension $n$ coming from some sort of operation on a low-complexity $n$-manifold, like the sphere $S^n$. Another classical theorem of Alexander shows this is indeed the case:

\begin{thm}[Alexander \cite{A0}]
Every closed oriented piecewise linear $n$-manifold is a piecewise linear branched cover over the $n$-sphere.
\end{thm}

We remark here that in the above result there is no control over the number of sheets of the covering, and the branch locus can be an arbitrary codimension two subcomplex.
Moreover, Bernstein and Edmonds \cite{BE0} showed that in general the theorem cannot be improved. In particular, they showed that:
\begin{itemize}
\item Any branched covering of the $n$-torus over the $n$-sphere must have at least $n$-sheets (this result was shown first by Fox \cite{F3} in case $n=4$).
    \item The quaternionic projective plane $\mathbb{HP}^2$
    cannot be obtained as a branched cover of the 8-sphere where the branch locus is a submanifold.
\end{itemize}

However, one can hope that in low dimensions, we can realize all closed oriented $n$-manifolds as a branched over over the sphere, with the branch locus being a submanifold and the degree of the covering being $n$. It is easy to see this in the case $n=1$, as $S^1$ is the only closed connected 1-manifold. By the classification of surfaces, we know that any closed oriented surface is determined by its genus $g$. We know that quotienting the surface of genus $g$  by the hyperelleptic involution gives us the 2-sphere with $2g+2$ branch points, see Figure~\ref{qhu}.

   \begin{figure}[!ht]
    \centering
    \includegraphics[width=7 cm]{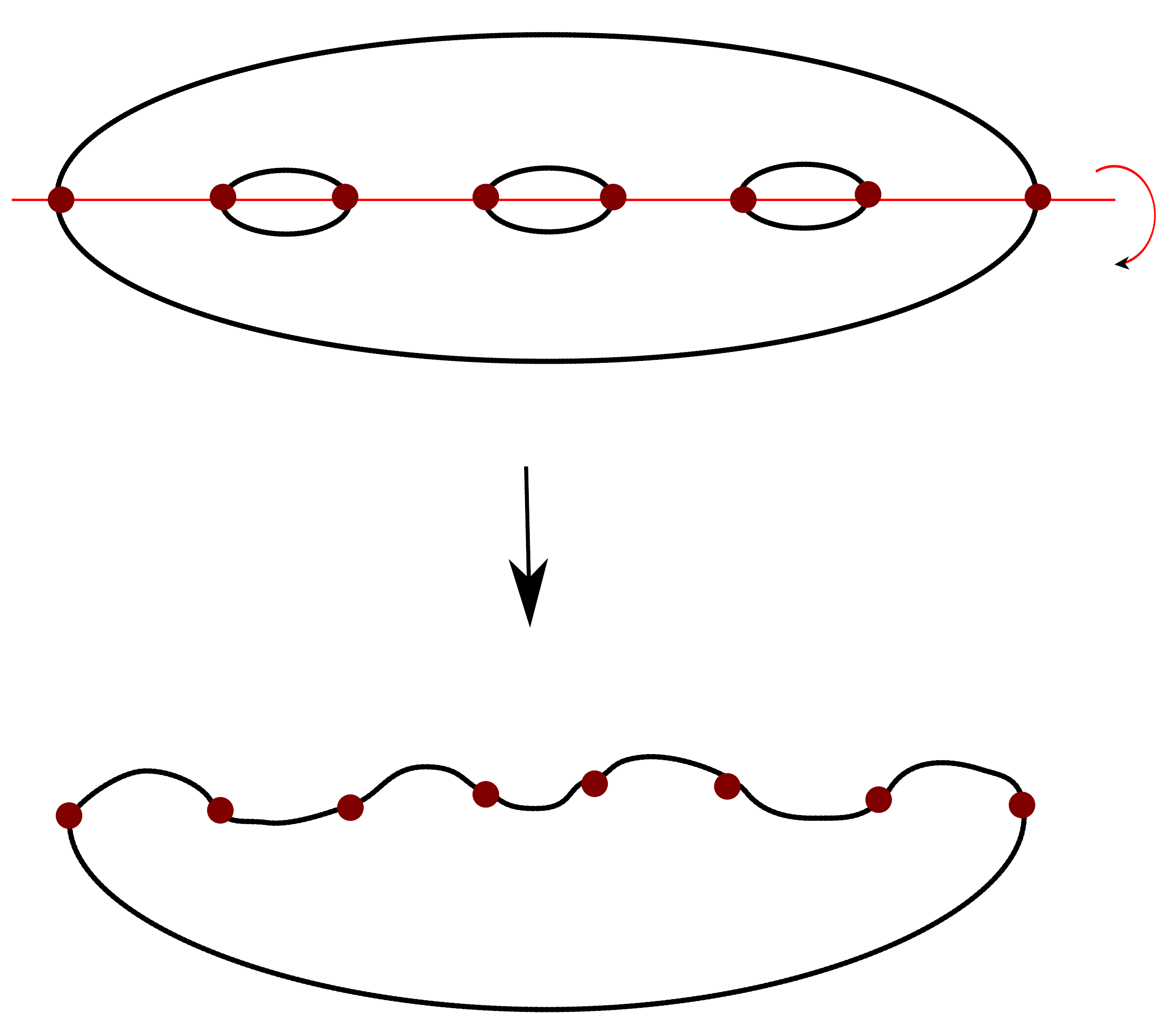}
    \caption{Consider the surface of genus $g$ sitting in $\mathbb{R}^3$ with an axis of rotational symmetry. If we quotient out we get a sphere with as many branch points as the number of times the axis intersects the surface.}
    \label{qhu}
\end{figure}

When $n=3$, we have the following theorem in both smooth and piecewise linear categories:
\begin{thm}[Hilden \cite{H}; Hirsch \cite{Hir}; Montesinos \cite{M1}]\label{HHMbc}
             Every closed oriented three manifold is a three sheeted simple branched cover over $S^3$, with the branch locus being a knot.
\end{thm}

Here by simple branched covering, we mean above any branch point only two sheets can come together.
There are similar results in dimension 4, but with some restrictions.
\begin{thm}[Piergallini \cite{P}]\label{P4}
             Every closed oriented piecewise linear four manifold is a four sheeted simple branched cover over $S^4$, with the branch locus being a transverse immersed piecewise linear surface.
\end{thm}


\subsection{Embeddings of manifolds}
Let us now discuss results about embedding manifolds in Euclidean space (or equivalently the sphere).

\begin{thm}[Whitney \cite{W} for smooth category, see \cite{RS} for piecewise linear category]
Every closed $n$-manifold embeds in $\mathbb{R}^{2n}$.
\end{thm}
There are characteristic class obstructions for embedding an arbitrary $n$-manifold in lower dimensional Euclidean space, for instance $\mathbb{RP}^n$ does not embed in $\mathbb{R}^{2n-1}$ when $n$ is a power of 2. However, just like the case of branched covers we have better bounds for low dimensional manifold. Every closed oriented surface embeds in $\mathbb{R}^3$, but as mentioned above, the real projective plane (or any of the non-orientable surfaces) do not embed in $\mathbb{R}^3$. It is a theorem of Hirsch \cite{Hi} that every closed oriented three manifold embeds in $\mathbb{R}^5$, and this was extended to the non-orientable case by Wall. 
  Hilden-Lozano-Montesinos \cite{HLM} gave an alternate proof of Hirsch's result, which in fact constructs a braided embedding\footnote{They used a slightly different terminology, but this is exactly what they proved.}.

\begin{thm}[Hilden-Lozano-Montesinos \cite{HLM}]
Every closed oriented three manifold has a three fold simple branched cover over $S^3$, which lifts to a braided embedding in $S^3\times D^2$.
\end{thm}
It is important to remark that their proof did not start with an arbitrary three fold simple branched cover and lift it to an embedding, they had to alter the branch locus (without changing the manifold upstairs) and bring it to a special form where the branched covering did lift. As we shall see, there are branched coverings over the three sphere which do not lift to a braided embedding.

Let us now go down in dimension and discuss if we can braided embed closed oriented one and two dimensional manifolds in trivial disc bundle of the sphere (and hence in Euclidean space).
Every closed one-manifold (that is disjoint union of circles) admits a codimension one\footnote{It admits a codimension zero embedding if and only if the manifold is connected} braided embedding in $S^1\times D^1$. In Example \ref{2fbc2s}, we see that the two fold branched cover of the genus $g$ orientable surface over the sphere obtained by quotienting by the hyperelleptic involution (see Figure \ref{qhu}) lifts to a co-dimension two braided embedding, and consequently every surface of genus $g$ braided embeds in $S^2\times D^2$.

Let us now turn to the main results of this paper, where we discuss:
\begin{question}\label{Q1}
Which branched covers over the sphere $S^n$ lift to co-dimension two braided embedding in $S^n\times D^2$?
\end{question}

The question is easy to answer for $n=1$, as any branched cover over $S^1$ must in fact be a covering map, and if we restrict to each component it must be equivalent to (in the notation of Example \ref{egcircle}) $p_n$ for some $n\in \mathbb{Z}\setminus\{0\}$, and thus any such covering lifts (these are the classical closed braids).\\

The situation becomes interesting for branched coverings of surfaces over the two-sphere, where we get different answers for the piecewise linear and smooth categories.

\begin{thm}\label{thm12} 
Every piecewise linear branched cover over $S^2$ lifts to a piecewise linear braided embedding in $S^2\times D^2$. In the smooth category, none of the odd fold cyclic branched covers over the spheres lift to 
a smooth co-dimension two braided embeddings, however the even fold cyclic branched covers, and all simple branched covers do lift.
\end{thm}

Note that we have only answered the lifting problem for simple branched covers and cyclic branched covers in the smooth category. We do have an algebraic obstruction (see Remark \ref{algobs}) to the existence of a lift in the smooth category but we are unable to answer if it is the only obstruction.

From the above theorem, we also obtain similar results for (branched) coverings of more general orientable surfaces.
\begin{thm}\label{thmliftor} 
Every coverings over a  genus $g$ orientable surface $\Sigma_g$ lifts to a branched cover in $\Sigma_g\times D^2$. Moreover, every piecewise linear branched cover over an orientable genus $g$ surface $\Sigma_g$ lifts to a piecewise linear branched cover in $\Sigma_g\times D^2$. 
\end{thm}
We should note that the liftability of coverings of degree less than 5 (and regular coverings of degree less than 60) over $\Sigma_g$ was already known from the work of Petersen \cite[Example 5.8]{Pe}.

We also explore liftability of coverings a non-orientable surface, the Klein bottle, in Subections~\ref{2fklein} and \ref{liftklein}. Along the way, we find an error in a thirty year old theorem regarding lifting composition of coverings, see Remark~\ref{countereg}.\\

Going one dimension higher, we will look at which simple branched coverings over the three sphere lift (and the answer will be the same for piecewise linear and smooth categories). As we saw in Theorem \ref{thm12}, not all branched covers over $S^2$ lift in the smooth category, so one would expect that not all branched covers over $S^3$ would lift. Indeed, if we are looking at non-simple branched covers over links in $S^3$, it is not too hard to find examples which do not lift. For instance, we have multiple non-liftable branched covers over the Hopf link, see Example \ref{hopf}.

We can show that three infinite families of branched coverings that lift to braided embeddings, for instance over torus knots.
\begin{thm}
  Every three fold simple branched covering of $S^3$, branched over
  a torus knot lifts to a simple braided embedding.
  \end{thm}

It turns out that not all simple branched covers over $S^3$ lift to braided embeddings, and we will try to answer the lifting problem of simple branched covers for some classes of knots and links. In fact, Etnyre and Furukawa \cite{EF} gave an infinite family of examples  of non-liftable simple branched covers over $S^3$ (of arbitrarily high degree),  and they used contact geometry to show there is no lift. They also showed connections between braided embeddings and contact embeddings in  three manifolds in $S^5$; in particular, if we can lift simple branched coverings where the branch locus is a closed braid we canonically obtain contact embeddings.

Carter and Kamada \cite{CK} first exhibited a three fold simple branched covering, branched over the $7_4$ knot, which does not lift to a braided embedding, using obstructions for existence of epimorphisms of knot groups. Here we will fill in a slight gap in their argument regarding
 the assumption of surjectivity of the group homomorphism. Moreover, we generalize this example of non-liftable tricoloring of $7_4$ knot to all two-bridge links in $S^3$. We give a complete characterization of which tricolorings of a two-bridge links lift to a braided embedding.
 
 \begin{thm}\label{2bridgeA} Suppose we are given a two-bridge link $L$ in $S^3$, and a Wirtinger presentation of fundamental group of the link complement $\langle a,b|r \rangle$, with $a,b$ meridians.  Then a three fold simple branched covering of $S^3$, branched over $L$ lifts to a braided embedding if and only if the relation $r$ holds when we set $a=\sigma_1$ and $b=\sigma_2$ (where $\sigma_1$ and $\sigma_2$ are the Artin generators of the three strand braid group $B_3$).
\end{thm}

We will see in Section~\ref{Secmono} that a braided embedding is equivalent information to braid monodromy, similar to how a branched covering is equivalent information as a permutation monodromy. Moreover, questions of lifting branched coverings to braided embeddings is equivalent to questions about lifting permutation monodromies to braid monodromies (and we will use the language of colorings, as mentioned in Subsection~\ref{Subscol}, to describe such monodromies).

We show in Theorem~\ref{thmepi} that if there is a braided embedding lifting a three fold simple branched covering over a knot in $S^3$ (i.e. a suitable homomorphism from the knot group to $B_3$), there is also an epimorphism from the knot group to $B_3$. Consequently (see the discussion following Theorem~\ref{thmepi}) there are obstructions to the existence of such homomorphisms, coming from (twisted) Alexander polynomials. Using the above two results we can completely answer which knots in Rolfsen's knot table admit a simple $B_3$-coloring.

  We address liftability of tricolorings of a large class of pretzel knots, including a complete understanding for three stranded pretzel knots.
  
  \begin{thm}\label{thm3pret3} Let $P(p,q,r)$ be a three stranded pretzel knot.
  \begin{enumerate}
      \item Suppose $p,q$ and $r$ are all odd. Then a non-trivial tricoloring on $P(p,q,r)$ lifts to a simple braid coloring
if and only if $\pm(p,q,r)\in\{(1,1,1),(3,3,-1),(3,-1,3),(-1,3,3)\}$.
\item Suppose exactly one of $p,q$ or $r$ even. Then, any tricoloring on $P(p,q,r)$ lifts to a simple $B_3$-coloring if and only if the tricoloring is constant on the twist region with an even number of half twists.\\
If we moreover assume exactly two of $p,q,r$ odd multiples of 3, and the third a multiple of 6; then exactly one of the four inequivalent tricolorings of $P(p,q,r)$ lift to a simple $B_3$-coloring.
  \end{enumerate}
 
  \end{thm}
  
  In particular, we obtain two different tricolorings of a pretzel knot (for instance $P(3,3,6)$), one of which lifts and the other not; which shows liftability is a property of the branched covering, and not just the branch locus. We also obtain some liftability results for higher stranded pretzel knots.
\begin{thm} \label{thm3pretm} 
Let $P(q_1, \ldots, q_m)$ be an $m$ stranded pretzel knot. 
\begin{enumerate}
\item If each of the $q_i$ and $m$ are odd and either 
\begin{enumerate}
\item $q_j=\pm q_{j+1}$ for some $j$ with $|q_j|\geq3$ and $|q_p|>3$ for some $p$, or
\item all the $q_j$ have the same sign,
\end{enumerate}
Then no non-constant simple tricoloring lifts to a simple $B_3$-coloring. 
\item If $m$ and $q_1$ are even and the other $q_j$ are odd, then any simple tricoloring lifts to a simple $B_3$-coloring.
\item If $q_1$ is even and the other $q_j$ and $m$ are odd, then a simple $S_4$-coloring list to a $B_4$-coloring if and only if the transpositions appearing in the left-most twist region (corresponding to $q_1$) are constant (i.e. the transpositions are identical if $q_1$ non-zero). 
\end{enumerate}
\end{thm}
  In Subsection \ref{fourfold} we discuss related results about lifting simple four fold branched covers over pretzel knots and show that:
  \begin{thm}\label{thm4pretB} 
Let $P(q_1, \ldots, q_m)$ be an $m$ stranded pretzel knot. 
\begin{enumerate}
\item For $m=3$, a non-constant simple $S_4$-coloring on $P(q_1,q_2,q_3)$ lifts to a simple $B_4$-coloring if and only if one of $q_1, q_2$ or $q_3$ is even, and moreover if said even number is a non-zero, the transpositions in appearing in the corresponding twist region are disjoint.
\item If each of the $q_i$ and $m$ are odd and either 
\begin{enumerate}
\item $q_j=\pm q_{j+1}$ for some $j$ with $|q_j|\geq3$ and $|q_p|>3$ for some $p$, or
\item all the $q_j$ have the same sign,
\end{enumerate}
Then no non-constant simple $S_4$-coloring lifts to a simple $B_4$-coloring. 
\item If $m$ and $q_1$ are even and the other $q_j$ are odd, then any simple $S_4$-coloring lifts to a simple $B_4$-coloring.
\item If $q_1$ is even and the other $q_j$ and $m$ are odd, then a simple $S_4$-coloring list to a $B_4$-coloring if and only if the transpositions appearing in the left-most twist region (corresponding to $q_1$) are disjoint (or identical). 
\end{enumerate}
\end{thm}
In the final section, we see how to construct non-liftable branched covers in higher dimensions.

Our main tool for the above results is to translate this problem about lifting group homomorphisms, with some constraints (for the embedding to be smooth or piecewise linear locally flat). Hansen has a complete characterization of which honest coverings lift to codimension two braided embeddings in terms of monodromy maps, see Subsection \ref{hc}. We formulate the following criterion to extend the braided embedding over the branch locus, see Subsection \ref{csplit} for the definition of 
completely split (standard) unlink and Section \ref{lm} for the definition of braid surrounding a branch point.

\begin{thm}\label{ext}

Suppose we have a  piecewise linear (respectively smooth) branched cover $p:M\rightarrow N$ with branch locus $B\subseteq N$ being a submanifold with trivial regular neighbourhood (respectively normal bundle) $\nu B\cong B\times D^2$, and $p|_{\tilde{B}}:\tilde{B}\rightarrow B$ is a covering map. Suppose we choose points $b_1,...,b_k$, one for each connected component of $B$. If we are given a locally flat piecewise linear (respectively smooth) braided proper embedding  $g_1:M\setminus \nu_0 \tilde{B}\hookrightarrow (N\setminus \nu_0 B)\times D^2$ lifting the honest covering $p_1:M\setminus \nu_0 \tilde{B}\hookrightarrow N\setminus \nu_0 B$ induced from $p$,  then $g_1$ extends to a locally flat piecewise linear (respectively smooth) braided embedding $g:M\hookrightarrow N\times D_2$ lifting $p$ if and only if each of the braids surrounding the branch points $b_i$ are completely split unlinks (respectively completely split standard unlinks).

\end{thm}

\noindent \textit{Acknowledgements}. The author would like to thank John Etnyre for introducing the problem, many useful discussions and making helpful comments on earlier drafts of this paper. This work is partially supported by NSF grants DMS-1608684 and DMS-1906414.

\section{Background}

\subsection{Covering maps} We recall that a continuous (we will typically be assuming maps are smooth or piecewise linear) map $p:Y\rightarrow X$ is a covering map if ever point in $X$ has a neighbourhood that is evenly covered by $p$. Alternately, they can be defined as a fiber bundle where the fibers are discrete. We will assume all our covering maps are finite sheeted.\\
We will now record a few useful results about covering spaces for future use.
\begin{claim}\label{pullback} Pull-backs of covering maps are covering maps.
 \end{claim}
 \begin{proof}
 Covering maps are fiber bundles where the fiber is a discrete collection of points. Pull-backs correspond to a base change, and the fibers remain unchanged. So the result follows.\\
 Alternately, it can be checked that the pull back of an evenly covered neighborhood is also evenly covered.
  \end{proof}
  We thank Carlo Petronio for explaining the following result to us.
 \begin{prop}\label{orifactor} Any branched covering of an orientable manifold over a non orientable manifold must factor through the orientation double cover.
  \end{prop}
  
  \begin{proof}
  Suppose $p:(Y,y_0)\rightarrow (X,x_0)$ with $Y$ orientable, and $X$ non-orientable. Let $q:(Z,z_0)\rightarrow (X,x_0)$ be the orientation double cover. Consider an arbitrary point $y\in Y$ and some path $\gamma$ joining $y_0$ to $y$. Then we can path lift $p(\gamma)$ with respect to the covering map $q$, and we set $r(y)$ to be the second endpoint of this lift (the first endpoint being $z_0$). To see $r:Y\rightarrow Z$ is well defined, note that if we had another path $\delta$ from $y_0$ to $y$, then the concatenation $\gamma* \bar{\delta}$ is loop on $Y$ which is orientation preserving (as $Y$ oriented). Since $p$ is a local homeomorphism, $p(\gamma* \bar{\delta}$ is orientation preserving in $X$, and so it lifts to a loop in $Z$. Thus, the second endpoint of the path lifts of $p(\gamma)$ and $p(\delta)$ coincide, so the map $r:Y\rightarrow Z$ is well-defined. Now one can check that $p=q\circ r$, and the result follows.
  \end{proof} 

\subsection{Branched coverings} By a branched covering in the piecewise linear category we will mean a map $p:M\rightarrow N $ so that there is a codimension two subcomplex $B$ in $M$ so that if we set $\tilde{B}=p^{-1}(B)$, then the restriction $p|_{M\setminus\tilde{B}}:M\setminus\tilde{B}\rightarrow N\setminus B$ is a honest covering map.

For a smooth branched cover we will put more restrictions, we want $B$ to be a smooth codimension two submanifold with trivial normal bundle, and for any point $\tilde{b}\in \tilde{B}$, there product neighbourhoods around $\tilde{b}$ and $b$ so that the map $p$ looks like $(c,z)\rightarrow (c,z^n)$ for some $n\in \mathbb{N}$. This restriction forces the map $p:\tilde{B}\rightarrow B$ to be a covering map.

\subsection{The Braid group} The braid groups were first defined (implicitly) by Hurwitz in \cite{Hu}, but was first explicitly studied by Artin \cite{Ar,Ar1}. The braid group on $n$ strands has  the following equivalent descriptions (see \cite{BB} for details):
\begin{itemize}
 \item the fundamental group of the unordered configuration space, $UConf_n(D^2)$ of $n$ distinct points in the open unit disc in $\mathbb{D}^2$ (or equivalently, the complex numbers $\mathbb{C}$). This group was first defined by Hurwitz, however this viewpoint was forgotten until it was rediscovered by Fox and Neuwirth \cite{FoN}.
 \item the group formed by isotopy classes of $n$-braids under concatenation. Here the identity braid is $n$-parallel strands with no crossings, and inverses are given by the reverse of the mirror.
 \item the mapping class group of the disc $D^2$ with $n$ marked points.
 
\end{itemize}

 We remark here that by our conventions, we write group elements from left to right in the first two formulations, however we write elements from right to left in mapping class groups, as this has become the standard convention for composing functions. Consequently, we have an anti-isomorphism (i.e. a bijective group antihomomorphism) between the groups in first (or second) and third viewpoint.
 
 The braid group $B_n$ on $n$-strands has a presentation (due to Artin \cite{Ar,Ar1})
$$B_n=\{\sigma_1,...,\sigma_{n-1}|\sigma_i\sigma_j=\sigma_j\sigma_i \text{ if } |i-j|>1, \sigma_i\sigma_{i+1}\sigma_i=\sigma_{i+1}\sigma_i\sigma_{i+1}\} $$

Looking at braids from the second viewpoint, if we forget all the crossing information and just look at where the endpoints go, we get a permutation
Forget$:B_n\rightarrow S_n$.


\subsection{Structure of centralizer of a braid}\label{centralizer}
For future use, let us record the a few facts about the structure of centralizers of braids.


\begin{prop} \label{centrper} (Ker\'{e}kj\'{a}rt\'{o} \cite{Ke, CK2}, Eilenberg \cite{Ei})
Every periodic $n$-braid is conjugate to either a power of 
$\delta=\sigma_n...\sigma_2\sigma_1$ or $\gamma=\sigma_n...\sigma_2\sigma_1^2$.
\end{prop}
Since centralizers of conjugate braids are conjugate, the structure of the centralizer of a periodic braid is understood once we understand the centralizers of powers of $\delta^k$ and $\gamma^k$. It is known \cite[Section 3]{GW} that these centralizers are isomorphic to  braid groups $B_d(D\setminus \{0\})$ on the punctured disc, where $d=\gcd(k,n)$ for $\delta^k$ and $d=\gcd(k,n-1)$ for $\gamma^k$.
In particular, we have:
\begin{cor}
For any $m$ coprime to $n$, the centralizer of the periodic braid $(\sigma_1...\sigma_{n-1})^m$ in $B_n$ is the infinite cyclic group generated by $\sigma_1...\sigma_{n-1}$.
\end{cor}

We now discuss the sructure of centralizers of non-periodic reducible braids.
We refer the reader to \cite[Section 5]{GW} for definitions of interior and tubular braids of a reducible braid and details of the proof of the following result.
\begin{thm}\cite[Theorem 1.1]{GW}
 The centralizer $Z(\beta)$ of a non periodic reducible braid $\beta$  in regular form fits in a split exact sequence:
 $$1\rightarrow  Z(\beta_{[1]})\times ...\times Z(\beta_{[t]})\rightarrow Z(\beta)\rightarrow Z_0(\widehat{\beta})\rightarrow 1,$$ 
 where $\beta_{[1]},...,\beta_{[t]}$ are the various interior braids and $\widehat{\beta}$ is the associated tubular braid; and $Z_0(\widehat{\beta})$ is the subgroup of the centralizer $Z(\widehat{\beta})$ of the tubular braid $\widehat{\beta}$ consisting of elements whose permutation is consistent with $\beta$.
\end{thm}


\subsection{Completely split links} \label{csplit}

We will say that a link $L$ in the solid torus $S_1\times D^2$ is a \emph{completely split link in regular form} there are disjoint sub-discs $D_1,...,D_k$ so that each component of $L$ lies in a different solid tori $S^1\times D^2_i$.
We will say that a link in the solid torus $S_1\times D^2$ is a \emph{completely split link} if it is isotopic in the solid torus to a completely split link in regular form.

We will say a link in the solid torus $S_1\times D^2$ is a \emph{completely split unlink} if it is a completely split link and each component is an unknot (i.e. bounds a disc in $S^3$, when $S_1\times D^2$ is included in $S^3$ as a neighbourhood of the unknot).

We define a closed braid $\hat{\beta}$ to be a standard unknot in the solid torus if $\beta\in B_n$ is conjugate to $\sigma_1 ...\sigma_{n-1}$ (all positive crossings) or 
$\sigma_1^{-1} ...\sigma_{n-1}^{-1}$ (all negative crossings).

We define a closed braid to be a \emph{completely split standard unlink} in the solid torus if it a completely split link so that each component is a standard unknot.

Let us consider the following subsets of the braid group, as introduced by Kamada
(see \cite[Section~16.5]{K2} for details)
$$A_n:=\{b\in B_n|\text{ the closure } \hat{b} \text{ of } b \text{ is completely split unlink} \}$$
$$SA_n:=\{b\in B_n|b \text{ is conjugate to }\sigma_1^{\pm 1} \}$$

We observe that $A_n$ consists of precisely all the elements of $B_n$ such that coning $\hat{b} \subseteq S^1\times D_2^2 $ over an appropriate number of points projecting to the origin, produces disjoint union of locally flat discs in $D_1^2 \times D^2_2$, and $SA_n\subset A_n$.

\section{Monodromy of Coverings and Braided embeddings }\label{Secmono}

\subsection{Monodromy of a covering}\label{pmh}
Suppose we have a covering map $p:(M,m_1)\rightarrow (N,n_1)$ with $n$ sheets. Then the covering map is determined (up to conjugation) by the monodromy map $\phi:\pi_1(N,n_1)\rightarrow S_n$. Let us suppose the points in the pre-image of $n_1$ are $m_1,...,m_n$ (in other words we are labelling the pre-image points with $1,...,n$, and different such choices give rise to conjugate representations), then for any loop $\gamma$ in $\pi_1(N,n_1)$, if we lift it we get a permutation of $\{1,...,n\}$ by seeing what the endpoint of the lift starting at $m_i$ is. This defines a group homomorphism $\phi:\pi_1(N,n_1)\rightarrow S_n$ which we will refer to as permutation monodromy.
Given a monodromy representation $\phi:\pi_1(N,n_1)\rightarrow S_n$, we obtain a covering space (which will be connected if and only if the monodromy action is transitive) by fixing some $j\in\{1,...,n\}$, and looking at the subgroup $H:=\{\gamma \in \pi_1(N,n_1)|\phi(\gamma)(j)=j\}$. By the correspondence between subgroups of the fundamental group and covering spaces, $H$ gives us a covering space with the required properties.

We now state a useful result concerning monodromy maps of covers.
 
 \begin{claim}\label {pbmono} If $f:(Y,y_0)\rightarrow (X,x_0)$ is a continuous map, and $p:X'\rightarrow X$ is a covering map with monodromy $\phi:\pi_1(X,x_0)\rightarrow S_n$, then the pullback covering $f^*p:f^*X'\rightarrow Y$ has monodromy $\phi\circ f_*:\pi_1(Y,y_0)\rightarrow S_n$, where it is assumed that the ordering on the pre-image points over the respective base-points is consistent with the pullback diagram\footnote{ otherwise the monodromy map will differ by a conjugation.}.
 \end{claim}
 \begin{proof} If $\gamma$ is any closed curve in $Y$ based at $y_0$, then $f\circ \gamma$ is a closed curve in $X$ based at $x_0$. If we pick any arc $\beta$ covering $\gamma$ (i.e. $f^*p\circ \beta=\gamma$), then the arc $p^*f\circ\beta$ covers $f\circ \gamma$. The result now follows from the consistent choice of numbering the pre-image points over $y_0$ and $x_0$.
    \end{proof}

\subsection{Monodromy of a branched covering}\label{pmb}
Given a branched covering space $p:M\rightarrow N$, then if we remove the branch locus $B$ from $N$, and its pre-image $\tilde{B}$ from $M$, then the restriction
$p|_{M\setminus\tilde{B}}:M\setminus\tilde{B} \rightarrow N\setminus B$ is a covering map, which completely determines the branched covering, by a result of Fox \cite{F1}.
So by monodromy map of a branched covering map, we will mean the monodromy map of the associated covering map.

\subsection{Monodromy of a braided embedding}\label{bm} Just like a covering $p:M\rightarrow N$ is determined by a monodromy representation $\phi$ of $N$ to the symmetric group, a braided embedding $p:M\rightarrow N\times D^2$ with $pr_1\circ f=p$ is determined by a monodromy representation $\psi$ of the $N$ to the braid group, which lifts $\phi$. To see this,
let us choose a base-point $x\in N$. For any simple closed curve\footnote{For self intersecting closed curves $\alpha:S^1\rightarrow N$, we can pullback the bundle $N\times D^2$ by $\alpha$ and get a solid torus over $S^1$, and the same statement holds.} $\gamma$ based at $x$, if we restrict the bundle $N\times D^2\rightarrow N$ to $\gamma$, we get a bundle $\gamma\times D^2\rightarrow \gamma$. For any point $c\in \gamma$ the braided embedding $f$ maps the $n$ pre-image points of $c$ under $pr_1\circ f$ to $n$ distinct points in the disc $\{c\}\times D^2$. Thus as we vary $c$ along $\gamma$, we get a closed braid in the solid torus $\gamma\times D^2$ (or equivalently a loop in the configuration space $UConf_n(D^2)$), which gives rise to a well defined braid by cutting the solid torus $\gamma\times D^2$ at $x\times D^2$ (or equivalently looking at the element of the fundamental group $UConf_n(D^2)$ given by that loop). This map gives rise to a group homomorphism $\psi:\pi_1(N)\rightarrow B_n$, which we will refer to as braid monodromy.

To construct a braided embedding from such a braid monodromy, we recall that Fadell, Fox and Neuwirth  \cite{FaN,FoN} showed that the configuration space $UConf_n(D^2)$ is aspherical, or in other words it is a $K(B_n,1)$. Thus a map of spaces $N\rightarrow UConf_n(D^2)$ is equivalent to a group homomorphism at the level of fundamental groups $\pi_1(N)\rightarrow B_n$. We note that a braided embedding $f:M\rightarrow N\times D^2$ (so that the associated covering map is $n$-sheeted) is equivalent to a choice of $n$-distinct points in $D^2$ for each point in $N$, i.e. a map $N\rightarrow UConf_n(D^2)$. This gives us the following criterion, due to Hansen, for lifting  a honest covering map to a codimension two braided embedding.
\begin{prop}[Hansen's criterion \cite{Ha}]\label{Hcri}
 A finite sheeted covering map $p:M\rightarrow N$ lifts to a braided embedding $f:M\rightarrow N\times D^2$ if and only if the associated permutation monodromy map $\phi:\pi_1(N)\rightarrow S_n$ lifts to a braid monodromy $\psi:\pi_1(N)\rightarrow B_n$, so that $Forget\circ\psi=\phi$.
\end{prop}

In case  $f:M\rightarrow N\times D^2$ is a braided embedding, with the projection to the first factor $pr_1\circ f:M\rightarrow N$ a branched covering map, it is similarly determined by the associated braided embedding $f|_{M\setminus\tilde{B}}:M\setminus\tilde{B}\rightarrow (N\setminus B)\times D^2$ (or equivalently its braid monodromy $\pi_1(N\setminus B)\rightarrow B_n$)  where we delete the branch locus and its pre-image. However, given a group homomorphism $\pi_1(N\setminus B)\rightarrow B_n$, one gets a braided embedding of the complement of the branch locus, but we need to be careful about extending over the branch locus (with appropriate restrictions, locally flat piecewise linear or smooth). We will discuss this issue in Section~\ref{Sextending}, where we prove Theorem~\ref{ext}.

\subsection{Permutation monodromy under composition} Now we will discuss how monodromy map of a composition of finite sheeted coverings looks like in terms of the individual monodromy maps. 
Suppose we have two covering maps $q:(Y,y_1)\rightarrow (X,x_1)$ and $p:(Z,z_1)\rightarrow (Y,y_1)$, with associated monodromy maps $\phi:\pi_1(X,x_1)\rightarrow S_n$ and  $\chi:\pi_1(Y,y_1)\rightarrow S_m$. Suppose the pre-images of the basepoint $x_0$ under $p$ are $y_1,...,y_{n}$ (i.e. we are fixing some ordering among them), and let us choose some paths $\alpha_1, \alpha_2,...\alpha_{n}$ joining $y_1$ to $y_1,...,y_{n}$ respectively (we may choose $\alpha_1$ to be the constant path).

Consider the composite covering $r=q\circ p:(Z,z_1)\rightarrow (X,x_1)$. Suppose the pre-image points of $y_i$ under $q$ are $z_{i,1},...,z_{i,m}$ , and thus we choose the lexicographic ordering for the pre-images of $x_1$ under $r$, i.e. we will enumerate these points as 
$$z_{(1,1)},...,z_{(1,m)}, \quad..., \quad z_{n,1},...,z_{(n,m)}.$$

For any based loop $\gamma$ in $(X,x_1)$, let us call the path lift of $\gamma$ starting at $y_i$ to be $\gamma_i$, for $1\leq i \leq n$. We know that that $\gamma_i$ ends at $y_{\phi(\gamma)(i)}$, by definition of the monodromy. Then the monodromy $\psi$ for $r$ is determined by:
$$\psi(\gamma)(i,j)=(\phi(\gamma)(i),\chi(\alpha_i*\gamma_i*\bar{\alpha}_{\phi(\gamma)(i)})(j)).$$

Let us make a few observations about the choices we made:
\begin{itemize}
    \item If we had chosen some other collection  of  paths $\beta_i$ joining $x_1$ to $x_i$, then the resulting monodromy using the $\beta$ curves will be conjugate to the one coming from the $\alpha$ curves, the conjugating permutation is given by the permutation built out of blocks $\chi(\alpha_i*\overline{\beta}_i)$.
    \item If we chose some other ordering of the pre-image points of the basepoint $x_1$ under $p$, or their pre-images under $q$, then the monodromy of $r$ the composition will also change by a conjugation, and the conjugating permutation is the one coming from the changing of the labeling.
\end{itemize}

\subsection{Liftings of composition of branched coverings}  

Suppose we have two branched covering maps $q:Y\rightarrow X$ and $p:Z\rightarrow Y$, so that their composition $r=q\circ p$ is also a branched covering.

\begin{prop} \label{comp0}
 If $p:Y\rightarrow X$ lifts to a codimension-$l$ braided embedding with trivial normal bundle, and $q:Z\rightarrow Y$ lifts to a codimension-$l$ braided embedding; then so does the composition $r=q\circ p$.
\end{prop}
\begin{proof}
We identify the normal bundle  of the image braided embedding of $Y$ in $X\times D^l$ with $Y\times D^l$, and then use the braided embedding of $Z$ in $Y\times D^l$ to obtain a braided embedding lifting $r$.
\end{proof}
\begin{remark}
We note that the trivial normal bundle condition is always satisfied in case of honest coverings, and in that case this result is proved in \cite[Theorem~4.1]{Pe}.
\end{remark}
We will see how the monodromy map of such a composite braided embedding looks like in the next subsection. Let us now discuss a partial converse to the above proposition.
\begin{prop} \label{comp1}
 If $r:Z\rightarrow X$ lifts to a codimension-$l$ braided embedding, then so does $q:Z\rightarrow Y$ (possibly to a non-locally flat braided embedding).
\end{prop}

\begin{proof}
Consider a braided embedding $R:Z\rightarrow X\times D^l$ lifting $r$. By composing with projection to the second factor, we obtain a separating map (i.e. which sends different pre-images of a point in the base space to distinct points) $s=pr_2\circ R:Z\rightarrow D^l$. We claim that $s$ is also a separating map for $p$ (i.e. $p\times s:Z\rightarrow Y\times D^l$ is a braided embedding lifting $p$). To see this, consider an arbitrary point $y\in Y$ and consider two pre-image points $z_1$ and $z_2$ under $p$. If it happens that $s(z_1)=s(z_2)$, then $s$ cannot be a separating map for $r$ since $r(z_1)=q(y)=r(z_2)$. The result follows.
\end{proof}

\begin{remark} In the setting of the above proposition, contrary to \cite[Theorem 4.2]{Pe}, $p$ need not lift to a braided embedding, see Remark \ref{countereg}.
\end{remark}
\begin{remark} The above proposition also generalizes to maps other than branched coverings, like the ones considered by Melikhov \cite{Me}. 

\end{remark}

\subsection{Braid monodromy under composition}

We can similarly determine the braid monodromy for composition of two braided embeddings. Let us suppose the braid monodromies of the braided embeddings are given by $\Phi:\pi_1(X,x_1)\rightarrow B_n$ and  $X:\pi_1(Y,y_1)\rightarrow S_m$ lifting $\phi:\pi_1(X,x_1)\rightarrow S_n$ and  $\chi:\pi_1(Y,y_1)\rightarrow S_m$, respectively.

We have $\Psi(\gamma)$ is a reducible braid determined by the tubular braid $\Phi(\gamma)$, and the interior braid corresponding to the $i$-th tubular braid is given by $X(\alpha_i*\gamma_i*\bar{\alpha})$.


\section{Lifting honest coverings}

\subsection{Applications of Hansen's criterion for lifting of covers}\label{hc}

It follows from Hansen's criterion~\cite{Ha} (Proposition \ref{Hcri}) that any covering over a space with free fundamental group lifts to a braided embedding. Moreover, one has a complete characterization when the fundamental group is finitely generated and abelian.
\begin{prop}\cite[Theorem~5.5]{Pe} If $\pi_1 (X)$ is finitely generated and abelian, then a covering $p:Y\rightarrow X$ lifts to a braided embedding iff $p_*(\pi_1(Y))$ contains all the torsion of $\pi_1(X)$.\end{prop}
 
 In particular, this implies:
 \begin{claim}\label{toruslifts} All finite sheeted coverings over an $n$-torus lifts to a braided embedding.
 \end{claim}
 This statement is easy to prove directly, using the fact that by a change of coordinates, any covering of the torus is equivalent to  a product of coverings over the circle; and Proposition \ref{comp0}.

Hansen's criterion also gives us examples of non-liftable covers. Since the braid groups $B_n$ are known to be torsion free, we get the following immediate obstruction to lifting:
\begin{claim}\label{torsionobs}
If $\alpha\in \pi_1(N)$ is torsion, and the permutation monodromy of $\alpha$ is non-trivial, then the permutation monodromy does not lift to a braid monodromy.
\end{claim}
In particular, if we apply this claim for the monodromy of the cover $a_2:S^2\rightarrow \mathbb{RP}^2$, then we recover Borsuk Ulam theorem in this dimension.

 \subsection{Liftings of two fold coverings over the Klein bottle}\label{2fklein}

Let us consider two 2-fold coverings over the Klein bottle $K$, which has fundamental group $\pi_1(K)=\langle a,b  |a^2b^2\rangle$. 
\begin{example} \label{egodckb}  
The first cover we will consider is the orientation double cover, that unwraps both the one-handles so the resulting cover is a torus. In other words this corresponds to the monodromy map: $$a\mapsto (12), b\mapsto (12).$$    

\end{example}   

\begin{example} \label{egkbckb}   
 The second cover we will consider is the also two sheeted cover, but one which unwraps only the first one-handle. In other words this corresponds to the monodromy map: $$a\mapsto (12), b\mapsto I.$$    
 The resulting cover is a connected non-orientable surface which has zero Euler characteristic, i.e. it is the Klein bottle. 
\end{example}  
    
\subsubsection{Liftability the above examples to braided embeddings}

Let us now consider the problem about lifting of the two coverings over the Klein bottle to braided embeddings, or equivalently if we can lift the permutation monodromy maps to a braid monodromy.

\begin{claim}\label{lifting2cov} The orientation double cover $f:\mathbb{T}\rightarrow K$ of the torus over the Klein bottle lifts to a braided embedding.
 \end{claim}
 \begin{proof} We have the braid monodromy $\pi_1(K)\rightarrow B_2$ sending $$ a\mapsto \sigma_1, b\mapsto \sigma_1^{-1},$$ which lifts the permutation monodromy in Example \ref{egodckb}.
 \end{proof}

\begin{claim}\label{nonlifting2cov} The two sheeted cover $g:{K}\rightarrow K$ of the Klein bottle over itself described in Example  \ref{egkbckb} does not lift to any braided embedding.
 \end{claim}
 \begin{proof} Suppose not, say there is a braid monodromy $\pi_1(K)\rightarrow B_2$ lifting it which sends $$ a\mapsto \sigma_1^m, b\mapsto \sigma_1^{n} \quad \text {  for some integers $m$ and $n$}.$$
 It must be the case that $m$ is odd and $n$ is even, as we obtain the symmetric group $S_2$ from the braid group $B_2$ by adding the relation $\sigma_1^2=1$. Then the image of $a^2b^2$ is $\sigma_1^{2(m+n)}$, which cannot be the identity as $2(m+n)$ is not divisible by 4.\\
 \end{proof}
 
 \begin{remark} Up to equivalence, there are only two connected two sheeted covering of the Klein bottle, the ones we discussed above. There is also a disconnected covering, with permutation monodromy $$a\mapsto I, b\mapsto I,$$ which lifts to the braid monodromy $$a\mapsto 1, b\mapsto 1.$$    
 \end{remark}
 
  \begin{remark}\label{countereg} We obtain the following commuting diagram of covering maps by using the two coverings and pull-backs.
  
  \[ \begin{tikzcd}
\mathbb{T} \arrow{r}{g^*f} \arrow[swap]{d}{f^*g} & K \arrow{d}{g} \\%
\mathbb{T} \arrow{r}{f}& K
\end{tikzcd}
\]

We see that by Claims \ref{lifting2cov} and \ref{toruslifts} both the coverings $f$ and $f^*g$ lift to braided embeddings, and hence by Proposition \ref{comp0}, so does their composition $f^*g\circ f=g^*f\circ g$. However, by Claim \ref{nonlifting2cov}, $g$ does not lift to any braided embeddings. This gives a counterexample to one half of one direction of \cite[Theorem 4.2]{Pe}. The gap in the proof of \cite[Theorem 4.2]{Pe} arises from subtleties in continuous variation of roots with respect to coefficients of polynomials.
 \end{remark}
 
 \subsection{Liftability of covers over the Klein bottle} \label{liftklein} 
 In this subsection we will explore some covers over Klein bottle which do (not) lift to braided embeddings. For notational convenience, in this section we will change our presentation of the fundamental group of the Klein bottle $K$  by replacing the generator with its inverse.
 Thus $\pi_1(K)=\langle x,y  |x^2=y^2\rangle$. So the question about liftability of covers can be rephrased as:\\

 \begin{question}\label{squareslift?}
  Given two permutations  $a$ and $b$ with $a^2=b^2$, can we  find braids $\alpha$ and $\beta$ lifting $a$ and $b$ respectively, and $\alpha^2=\beta^2$?
  \end{question}
 The answer, as we have seen is not always yes, and let us try to find what some of the the obstructions are. The first one comes from the sign of the permutations (and exponent sum of braids).
 
 \begin{claim} If the answer to Question \ref{squareslift?} is affirmative, then permutations $a$ and $b$ have to have the same sign, i.e. both are even or both are odd.
   \end{claim}
   
   \begin{proof} Suppose not, say $a$ is even and $b$ is odd. Then $\exp(\alpha)$ is even and $\exp(\beta)$ is odd, so $\exp(\alpha^2)$ is divisible by four while $\exp(\beta^2)$ is not, contradicting $\alpha^2=\beta^2$.
   \end{proof}
   The next question we may ask is if $a$ and $b$ satisfying $a^2=b^2$ have the same sign, then can we find lifts $\alpha$ and $\beta$ with the required properties? The answer is 'No' as the following claim and example shows.
 
  \begin{claim}If the answer to Question \ref{squareslift?} is affirmative, then permutations $a$ and $b$ have to be conjugate.
 \end{claim}
 
 \begin{proof}
  By \cite{GM}, roots in the braid group are unique up to conjugation, and thus if $\alpha^2=\beta^2$ in the braid group, the braids $\alpha$ and $\beta$ are conjugate, whence the associated permutations $a$ and $b$ have to be conjugate in the symmetric group.
 \end{proof}
 
 \begin{example} Consider $a=(123456)$ and $b=(153)(164)$, then $a^2=(135)(264)=b^2$, and $a$ and $b$ are not conjugate since they have different cycle types. Thus by the claim above, the corresponding cover does not lift.
  \end{example}
  
  One may now ask the question with $a$ and $b$ as above with the same cycle type, but the answer still remains 'No', as illustrated by the following claim.
  
  \begin{claim}  Consider $a=(12)$ and $b=(34)$, then $a^2=I=b^2$. There are no four braids  $\alpha$ and $\beta$ lifting $a$ and $b$ satisfying $\alpha^2=\beta^2$.
  \end{claim}

 \begin{proof}
 We note that if such a lift did exist, we would get a braided embedding of two disjoint Klein bottles. If we restricted our attention to one connected component, we would obtain a braided embedding whose associated covering is equivalent to Example~\ref{egkbckb}. However Claim~\ref{nonlifting2cov} shows this cannot lift to a braided embedding, which completes the proof of the claim.
 \end{proof}
  We can now ask the same liftability question by further adding the restriction that the covering is connected (i.e. a Klein bottle). While we are unable to answer the question completely, we give some partial results below.
  We will discuss possible lifts according to their Nielsen-Thurston type.
  We begin by noting that since for any lift $\alpha$ and $\beta$ of $a$ and $b$, we must have $\alpha^2=\beta^2$ and so $\alpha$ and $\beta$ must have the same Nielsen-Thurston type, i.e. either periodic, or non periodic reducibles or pseudo-Anosov.\\
  It turns out that the \textbf{pseudo-Anosov} case cannot happen unless $a=b$, as shown by the following proposition by specializing to $k=2$.
  \begin{prop} If $n$-braids $\alpha$ and $\beta$ satisfy $\alpha^k=\beta^k$ for $k>0$, and $\alpha$ is pseudo-Anosov, then $\alpha=\beta$.
   \end{prop}
   \begin{proof}
   This is because $k$-th roots of pseudo-Anosov braids are unique \cite[Section 4.2]{GM}, and we outline the argument below. By Nielsen-Thurston theory, $\alpha$ is  pseudo-Anosov iff $\alpha^k$ is; and they have the same stable and unstable foliations.
      Since $\alpha^k=\beta^k$ the stable and unstable foliations of $\alpha$ and $\beta$ coincide, whence they commute. So, $\alpha^k=\beta^k$ is equivalent to $(\alpha \beta^{-1})^k=1$, which forces $\alpha=\beta$, as braid groups are torsion free.
   \end{proof}

   Let us now discuss the case where $\alpha$ (and hence $\beta$) is \textbf{periodic}. Since $\alpha$ and $\beta$ have to be conjugate, we must have $\beta=\gamma^{-1}\alpha\gamma $ for some braid $\gamma$; and $\alpha^2=\beta^2$ is equivalent to stating $\gamma$ commutes with $\alpha^2$.
   By our discussion in subsection~\ref{centralizer} (i.e. from the results in \cite{GW}), we have a complete understanding of the center of periodic braids. Using this it follows that:
   \begin{prop}
    If $a$ is the permutation associated to a periodic braid $\alpha$, we can lift any $b$ conjugate to $a$ satisfying $a^2=b^2$, to a braid $\beta$ so that $\alpha^2=\beta^2$.
   \end{prop}
   Indeed, the given condition on $a$ and $b$ ensures there is a permutation $c$ conjugating $a$ to $b$ (i.e. $cb=ac$) which is a cyclic shift of a certain subset of $\{1,...,n\}$. It is then easy to find a symmetric (w.r.t. the periodic action of $\alpha^2$) braid $\gamma$ whose associated permutation is $c$, and we can take $\beta=\gamma^{-1}\alpha\gamma$.\\
   
   Finally, we move onto the case $\alpha$ (and hence $\beta$) is \textbf{non-periodic reducible}. In this case we know that the \emph{canonical reduction system} (CRS) of $\alpha$ and $\beta$ must coincide, by  \cite[Lemma 2.2]{GM} since $\alpha^2=\beta^2$. By utilizing this observation, we will give some examples of (non) liftings
   to such braids. Furthermore, we will outline a sort of algorithmic procedure to see if given $a,b$ satisfying necessary conditions, if one can find a lift to reducible braids. We begin with considering the case of liftability of a three fold cover over the Klein bottle.
   
   \begin{claim}\label{noredli}
   Consider $a=(13)$ and $b=(12)$, so $a^2=I=b^2$. There are no non-periodic reducible braids $\alpha$ and $\beta$ lifting $a$ and $b$ so that $\alpha^2=\beta^2$.
   \end{claim}
   
   \begin{proof}
   Suppose not, since the canonical reduction system of $\alpha$ and $\beta$ have to coincide, the canonical reduction system must induce a nontrivial (nested) partition. We list all possible partitions together with the permutations they can support:
   \begin{itemize}
       \item $\{1,\{2,3\}\}$ can support $I$ and $(23)$,
       \item $\{\{1,2\},3\}$ can support $I$ and $(12)$,
       \item $\{\{1,3\},2\}$ can support $I$ and $(13)$.
   \end{itemize}
  Since we cannot find a partition supporting both 
  $a$ and $b$ simultaneously, we arrive at a contradiction.
   \end{proof}
   \begin{remark}\label{perli} We should note that there is a periodic lift in the situation in the above claim. We can take $\alpha=\sigma_1\sigma_2\sigma_1$, and since $\alpha^2$ is a central element in $B_3$, we see that $\beta=\sigma_2^{-1}\alpha\sigma_2$ satisfies $\alpha^2=\beta^2$.
   \end{remark}
   
   Let us now mention a general approach  to see if we can find suitable reducible braids $\alpha$ and $\beta$ lifting $a$ and $b$.
   \begin{itemize}
       \item Given permutations $a,b\in S_n$, find a nested partition of $\{1,2,...,n\}$, compatible with $a$ and $b$, which is (setwise) invariant under a permutation $c$ satisfying $b=c^{-1}ac$.\\
       
       As mentioned earlier, if such lifts $\alpha$ and $\beta$ do exist, they have to be conjugate i.e. $\beta=\gamma^{-1}\alpha\gamma $ for some braid $\gamma$. If $c$ is the permutation associated with $\gamma$, this forces the relation $b=c^{-1}ac$. The canonical reduction system of $\alpha$ and $\beta$ has to coincide, and this induces a nested partition of $\{1,2,...,n\}$. As the CRS must be invariant under the conjugation by $\gamma$, we must have the nested partition to be invariant under $c$.\\
       
        In the next two steps, we try to use the CRS structure to solve similar problems in smaller stranded braid groups. \\
        
       \item For all possible choices of nested partitions and $c$ obtained in the previous step, look at the associated permutation $\widehat{a}$ formed by sending this nested partition to itself by $a$, and try to find the corresponding permutation $\widehat{a}$ for $c$, and try to find lifts $\widehat{\alpha}$ and $\widehat{\gamma}$ satisfying $\widehat{\alpha}^2=\widehat{\gamma}^{-1}\widehat{\alpha}^2\widehat{\gamma}$.\\
       
       Here we try to find possible tubular braids $\widehat{\alpha}$ for $\alpha$, and $\widehat{\gamma}$ for $\gamma$. We see that if $\alpha^2$ commutes with $\gamma$, then 
       we must have $\widehat{\alpha}^2$ commutes with $\widehat{\gamma}$.\\
      
      \item For all possible compatible choices $\widehat{\alpha}$ and $\widehat{\gamma}$ of tubular braids from the previous step, we try to find interior braids so that all the conditions are satisfied.
       
   \end{itemize}
   We illustrate this approach with an example to see in some cases this will give us a suitable lift.
      \begin{claim}
   Let $a=(172839)(465)$ and $b=(132)(475869)$, and so $a^2=(123)(456)(789)=b^2$. Then there are braids $\alpha$ and $\beta$ lifting $a$ and $b$ satisfying $\alpha^2=\beta^2$.
   \end{claim}

   We note that $a$ is cannot be a permutation of a periodic braid, so if a lift exists, it must necessarily be non periodic reducible.
  \begin{proof} From $b^2=(123)(456)(789)$, a natural partition to consider is $\{\{1,2,3\},\{4,5,6\},\{7,8,9\}\}$. Also, we can take $c$ to equal $(14)(25)(36)$ (we see that $b=c^{-1}ac$). Since $a$ swaps $\{1,2,3\}$ with $\{7,8,9\}$; and $b$ swaps $\{4,5,6\}$ with $\{7,8,9\}$, the tubular braids (of a possible lift) have the same configuration as in the statement of Claim~\ref{noredli}. Thus we can take the tubular braids coming from Remark~\ref{perli}. After that we can add periodic 3-braids as interior braids and obtain a reducible braid $\alpha$ lifting $a$, as shown in Figure~\ref{redA}. We observe that with this choice, all the interior braids of $\alpha^2$ is the same periodic three braid.
  \begin{figure}[!ht]
    \centering
    \includegraphics[width=9 cm]{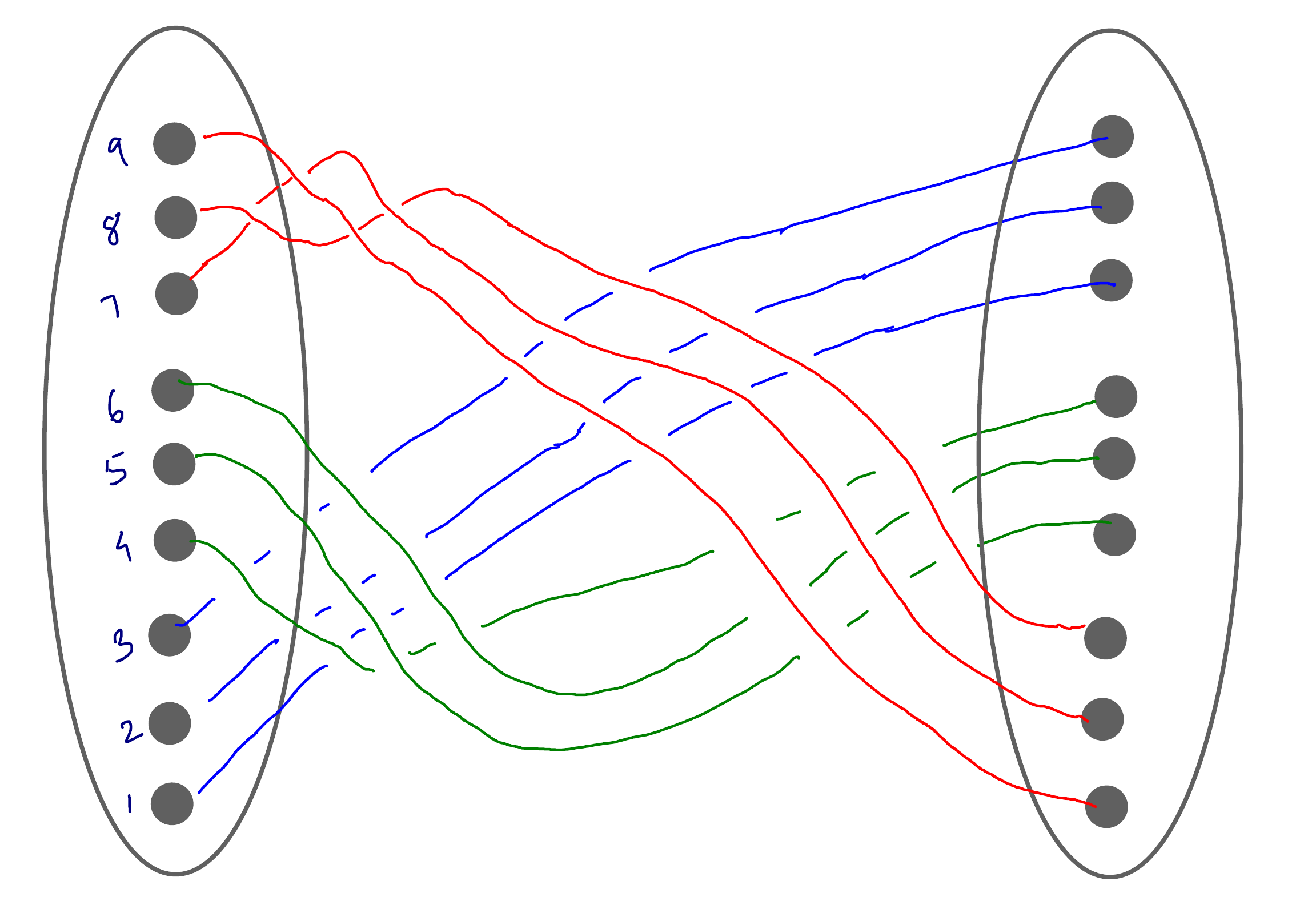}
    \caption{A reducible braid $\alpha$ lifting $a=(172839)(465)$.}
    \label{redA}
\end{figure}
Finally we note that to find $\gamma$ lifting $c$, we can take all interior braids to be the identity, as shown in Figure~\ref{redC}.
  \begin{figure}[!ht]
    \centering
    \includegraphics[width=9 cm]{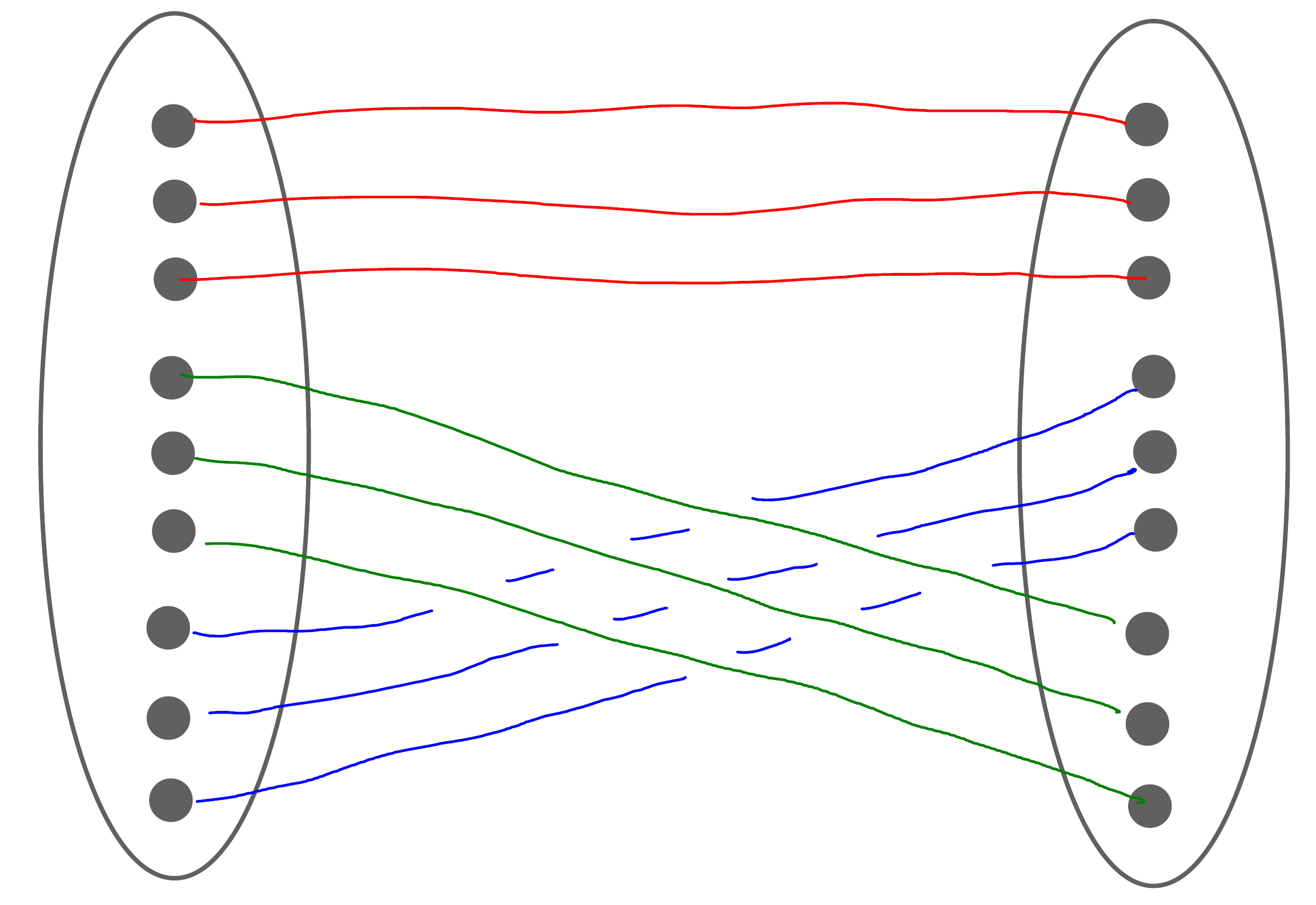}
    \caption{A reducible braid $\gamma$ lifting $c=(14)(25)(36)$.}
    \label{redC}
\end{figure}
Now it is easy to verify $\alpha^2$ commutes with $\gamma$, and so $\beta=\gamma^{-1}\alpha\gamma $ satisfies $\alpha^2=\beta^2$, as desired.
  \end{proof}
   
   Let us conclude this subsection by showing any covering of the torus over the Klein bottle always lifts.

  \begin{prop} Any covering of a torus over the Klein bottle lifts to a braided embedding.
   \end{prop}
 
 \begin{proof}
  If we have any covering of torus over the Klein bottle, then it must factors through the orientation double cover by Proposition~\ref{orifactor}. We know that  by claims \ref{toruslifts} and \ref{lifting2cov}, both these coverings lift to braided embedding, and hence so does the composition (as discussed in Proposition~\ref{comp0}, or in \cite[Theorem 4.1]{Pe}).
  
 \end{proof}
 
 \subsection{Lifting of coverings over general orientable surfaces}
 
 Now that we have seen several non-liftable covers over a non-orientable surface (and it is easy to generalize this to higher genus non-orientable surfaces), we can restrict our attention to orientable surfaces and ask if every covering lifts. The answer turns out to be 'Yes', but the only proof we have will use results about lifting branched coverings, discussed in later sections.

\begin{thm}\label{orcovsurf}
 Every covering of over an orientable surface lifts to a braided embedding.
\end{thm}

\begin{proof} Given any covering $p:\Sigma_g\rightarrow \Sigma_h$, we can compose with the two sheeted branched covering $q$ of $\Sigma_h$ over $S^2$ obtained by quotienting by the hyperelleptic involution (see Figure~\ref{qhu}). The resulting composition $r=q\circ p:\Sigma_g\rightarrow S^2$ is a branched covering. By Theorem \ref{bcl2s}, we know that 
any branched covering over the sphere lifts to a braided embedding.  The result now follows from Proposition \ref{comp1}.
 
\end{proof}

\section{Local models near branch points in dimension two}\label{lm}

In this section we will find necessary and sufficient conditions for extend a braided embedding of a punctured two dimensional disc (the puncture being one branch point removed) over the puncture, thereby proving Theorem \ref{ext} in the simplest case, where the branch locus is an isolated collection of points. In the next section we will extend it to a more general setup. 

Suppose we have a braided embedding (in either piecewise linear or smooth categories) over a disc $D_b^2$ with a single branch point $O$,  $f:\sqcup_{i=1}^j D_i^2\hookrightarrow D_b^2\times D^2$ , with $p_i:=pr_1\circ f|_{D_i}:D_i^2\rightarrow D^2_b$ a branched covering map with at most one branch point at $O$. Note that $j$ equals the number of cycles (including the fixed elements) in the permutation corresponding to a loop surrounding $O$ in the monodromy representation. Let us suppose the point $O_i$ is the unique point in $D_i^2$ mapping to $O$ under $p_i$. We note that $pr_2\circ f$ maps the $O_i$ to distinct points in $D^2$, as all of them project to the same point $O$ under $pr_1\circ f$ and $f$ is an embedding. Hence we can choose a small closed disc $C_\epsilon$ around $O$ so that when we set $C_i=p_i^{-1}(C_\epsilon)$, the images  $pr_2\circ f(C_i)$ are disjoint for different $i$.

Let us look at the braid monodromy of braided embedding induced from $f$, once we remove $O$ and $O_i$'s. Since the fundamental group of $D^2_b\setminus\{O\}$ is infinite cyclic freely generated by any  circle $\gamma_r$ of radius $r$, where $0<r<1$. Thus the braid monodromy $\psi$ is completely determined by $\psi(\gamma_r)$,
and we will call this the \textit{braid surrounding the branch point} $O$. Note that if we pick $C_\epsilon$ to be the disc of radius $\epsilon$ centered at the origin, then $\psi(\gamma_\epsilon)=f(\sqcup_{i=1}^j \partial C_i)$ is the closure of the braid surrounding the branch point in the solid torus $\gamma_\epsilon\times D^2$. Let us now see what constraints we get on the braids surrounding the branch point $O$.

\subsection{Piecewise linear category}\label{SSextpl} In case we are working in the piecewise linear category, we may assume that we chose the disc $C_\epsilon$ small enough so that $f$ is defined on
$\sqcup_{i=1}^j C_i$ by coning $\sqcup_{i=1}^j f(\partial C_i)$ over the points $f(O_i)$, independently for each $i$. Since the images $f(C_i)$ are disjoint, we see that the braid surrounding the branch point must be reducible if $j>1$.

Moreover since the embedding must be locally flat it must be locally flat near the points $O_i$ which means the image $f(\partial C_i)$ must be an unknot in the solid torus $\gamma_\epsilon\times D^2$. Thus we see that $\psi(\gamma_\epsilon)$, the braid surrounding the branch point must be a completely split unlink.

Conversely given any closed braid which completely split unlink (without loss of generality, let us assume it is in regular form) and satisfies the appropriate permutation restrictions (i.e. agrees with the permutation monodromy at $O$), we can extend it to a (locally flat) piecewise linear braided embedding  over the entire disc. To see this let us assume the various components $L_1,..,L_j$ of the closed braid lie in disjoint solid tori $S_1\times N_1,...,S_1\times N_j$. For each $i\in\{1,...,j\}$, let us pick a point $n_i\in N_i$, and we can then set $f(O_i)=(O,n_i)$, and we can cone $L_i$ (in $\partial C_\epsilon\times D^2$) at $f(O_i)$, and get a well defined braided embedding over $C_\epsilon$. Note that since the projections under $p_2$ of various coning operations are concentrated within the disjoint discs $N_i$, we see that the above map is indeed injective. Therefore we obtain,
\begin{lem}
A braided embedding of a disjoint union of circles (i.e. a closed braid), it extends piecewise linearly to a braided embedding of a disjoint union of discs with one branch point if and only if the closed braid is a completely split unlink.
\end{lem}

\subsection{Smooth Category}
Similarly to the piecewise linear category, we see that the braid surrounding the branch point must be completely split, however there are more conditions to smoothly extend it over a branch point. Let us restrict $f$ (and call this restriction $f_i$) to one of the nontrivial (i.e. we have actual branching) components $D_i^2$ and see what braid we get. The Jacobian matrix at the $O_i$ will look like $Df_i=[0|A]$, as $pr_1\circ f$ has a local model $z\mapsto z^n$ with $n>1$. Since $f_i$ is an embedding $Df_i$ must have rank 2, and hence $A$ is an invertible $2\times 2$ matrix. By the inverse function theorem $pr_2\circ f$ is a local diffeomorphism. So the circle $\{|z|=a\}$ in $D^2$  embeds in $D_2^2$ via $pr_2\circ f_i$ for small $a$. By the Schoenflies theorem in the plane and isotopy extension theorems we see that $p_2\circ f_i$ is isotopic to either the identity or complex conjugation near $O_i$\footnote{An alternate way to think about this is the contractibility of the space of embeddings of disc to a disc, see https://mathoverflow.net/questions/181424/contractibility-of-space-of-embeddings-of-a-disc .}. Thus, $f_i$ is locally equivalent (that is isotopic in a sufficiently small neighborhood of $O_i$) to $f_+$ or $f_-$,
where the maps $f_\pm:D_a^2\rightarrow D_b ^2 \times D ^2$ defined by 
$z\mapsto (z^n,z)$ and $z\mapsto (z^n,\bar{z})$. It therefore suffices to understand what the braids surrounding the branch points in the local models $f_\pm$ are. By choosing the $n$-th roots of unity as collection of $n$ distinct points, it is clear geometrically that we get a positive (respectively negative) partial twist for $f_+$ (respectively $f_-$); i.e. the braid surrounding the branch point is (upto conjugation) $\sigma_1...\sigma_{n-1}$ (respectively $\sigma_1^{-1}...\sigma_{n-1}^{-1}$). We can also argue this more analytically as follows.

Let us say our convention is that for the disc $D^2$ we project out the second factor to get a knot diagram. To be precise, for any $0<a<1$ the circle $\gamma_a=\{|z|=a\}$ in $D^2_a$ maps to the circle $\gamma_b$ in $D^2_b$, where $b=a^n$, and $f_\pm(\gamma_a)$ gives rise to a closed braid in the solid torus $\gamma_b\times D^2$, and ignoring the second coordinate of $D^2$, we get the knot diagram, where crossings happen at the point where the first three coordinates of $f_\pm$ agree. If $z_1^n=z_2^n$, it necessarily must be the case that the complex arguments of $z_1$ and $z_2$ differ by $\frac{2k\pi}{n}$, where $0<k<n$ is an integer. For the third coordinates of $f_\pm$ to agree, we need to have the cosines of these arguments to agree. Now we observe that for each integer $k$ satisfying $0<k<n$,
the equation $$\cos (\theta)=\cos (\theta+\frac{2k\pi}{n}) \text { or equivalently } 
-2\sin (\theta+\frac{k\pi}{n})\sin(\frac{k\pi}{n})=0 $$ has exactly one solution in $[0,\pi)$, namely $\theta=\frac{(n-k)\pi}{n}.$ Thus we see there are exactly $n-1$ crossing points, and by looking at the sines at these points we see that they are positive crossings for $f_+$ and negative crossings for $f_-$, i.e. the braids surrounding these branch points are standard unknots. The actual braids we get can be chosen to be $\sigma_1...\sigma_{n-1}$ and $\sigma_1^{-1}...\sigma_{n-1}^{-1}$ by choosing the base-point with complex argument 0.

\begin{remark} We remark here that if $\tau$ is any permutation of $\{1,...,n\}$ then $\sigma_1^{\eta_1}...\sigma_{n-1}^{\eta_{n-1}}$ is conjugate to $\sigma_{\tau(1)}^{\eta_{\tau(1)}}...\sigma_{\tau(n-1)}^{\eta_{\tau(n-1)}}$. To see this note that we can go between the two words by applying a sequence of far commutation relations and conjugations. In particular, there is nothing special about the braid $\sigma_1...\sigma_{n-1}$ we got for $f_+$, we can apply any permutation $\tau$ and we would get the same braid closures.
\end{remark}

Thus, returning to our original setup, we see that the braid surrounding a smooth branch point has to be a completely split standard unlink, and conversely given such a closed braid, we can extend it smoothly as a braided embeddings over the disc with exactly one branch points by using the models $f_\pm$ (with different values of $n$) locally to extend the map on each of the components separately. In other words, we have:
\begin{lem}
A braided embedding of a disjoint union of circles (i.e. a closed braid), it extends smoothly to a braided embedding of a disjoint union of discs with one branch point if and only if the closed braid is a completely split standard unlink.
\end{lem}


\section{Extending braided embedding over the branch locus}\label{Sextending}
In this section we address the question:
\begin{question}
Given a co-dimension $2$ braided embedding on the complement of the branch locus, when can we extend the braided embedding over the branch locus (smoothly or locally flat p.l.)?

\end{question}

We will answer this question completely when the branch locus is a submanifold with trivial normal bundle. We will mostly use this result for branched covers over the sphere, and whenever the branch locus is embedded as a codimension two submanifold, thus has a Seifert hypersurface \cite{Kir}, which trivializes the normal bundle.

As we saw in the last section, when we are looking at branched covers of surfaces (i.e. the branch locus is just a discrete set of points), we had conditions on the braid surrounding the branch points for the braided embedding to extend over the branch point. Of course, the same constraint is there for each braid corresponding to meridian around the branch point, and as we see below, these conditions are enough in the general case too.

In case the branch locus is a manifold, the various meridians are conjugate as long as the branch points belong to the same connected components. Consequently, in the above case, one needs to verify that for each connected component, the braid surrounding the branch points satisfy the appropriate conditions for us to extend the braided embedding over the branch locus. 


We will analyze the situation in the smooth and the piecewise linear categories seriously below. We will use the local models discussed in the last section, to first deal with the special case when the branch locus $B$ and its pre-image $\Tilde{B}$ are both connected, and then show the general case reduces to the above special case by using the structure of centralizers of reducible braids, see \cite{GW}.



\subsection {Smooth Category}

Here we will prove Theorem \ref{ext} in the smooth category.
\begin{prop}
Suppose we have a smooth branched cover $p:M\rightarrow N$ with branch locus $B\subseteq N$ being a submanifold with trivial normal bundle $\nu B\cong B\times D^2$. Suppose we choose points $b_1,...,b_k$, one for each connected component of $B$. If we are given a smooth braided proper embedding  $g_1:M\setminus \nu_0 \tilde{B}\hookrightarrow (N\setminus \nu_0 B)\times D^2$ lifting the honest covering $p_1:M\setminus \nu_0 \tilde{B}\hookrightarrow
N\setminus \nu_0 B$ induced from $p$, then $g_1$ extends to a smooth braided embedding $g:M\hookrightarrow N\times D_2$ lifting $p$ if and only if each of the braids surrounding the branch points $b_i$ are completely split standard unlinks.

\end{prop}

We observe that if $\hat{b}_i$ is another point is the same connected component as $b_i$, then the meridians surrounding $b_i$ and $\hat{b}_i$ are conjugate in the fundamental group of $N\setminus \nu_0 B$, and consequently so are the braids surrounding $b_i$ and $\hat{b}_i$. Thus we see that the condition stated in the the above proposition is independent of which particular points from each connected components of $B$ is chosen.

\begin{proof}
Suppose $g_1$ does extend to a smooth braided embedding $g$. Then for any $i$, if we restrict $g$ to the slice $\{b_i\}\times D^2$ in $\nu B$, then we see that we get a smooth braided embedding over a disc $D^2$ with exactly one branch point. By the local model we studied in the previous section we see that the braid surrounding the branch point $b_i$ has to be a completely split standard unlink.

It remains to show the converse, so let us now suppose we have a braided embedding $g_1$ so that the braid surrounding the branch points are completely split standard unlinks. It suffices to construct a braided embedding of $h:\nu\tilde{B}\hookrightarrow \nu B\times D^2$, so that the braided embedding on the boundary $h_1:\tilde{B}\times S^1\hookrightarrow  (B\times S^1)\times D^2$ coming from $h$ agrees with that coming from $g_1$. Observe that in the above case we can isotope both the braided embeddings $g_1$ and $h$ in the above case, so the braided embedding is invariant in a collar neighbourhoods of the boundary, and then identify the collar neighbourhoods, and thereby obtaining a smooth braided embedding $g$ lifting $p$.

Recall that since we are considering a smooth branched covering $p:M\rightarrow N$, then the restriction of $p|_{\tilde{B}} :\tilde{B}\rightarrow B$ is a covering map. To define this map, we can define the braided embedding on each connected component of $B$ individually. Without loss of generality, let us now assume $B$ is connected. Suppose the fundamental group of $B$ has the presentation $\langle x_1,...,x_s|r_1,...,r_t\rangle$ (recall we are assuming our manifolds are compact, and thus the branch locus being a closed submanifold also has the same property. Consequently, its fundamental group is finitely generated. However, the reader can observe that the same argument given here also works if $\pi_1B$ is not finitely presented).

Since $B$ has trivial normal bundle, the boundary of $\nu B$ is diffeomorphic to $B\times S^1$, and so has fundamental group $$\pi_1(\partial \nu B)\cong\pi_1(B)\times \mathbb{Z}\cong \langle x_1,...,x_p,\mu|r_1,...,r_q, [x_i,\mu] \text { for all } 1\leq i\leq s\rangle.$$ Here $\mu$ denotes the loop (meridian) corresponding to the $S^1$ factor, and $[x_i,\mu]$ denotes the commutator of $x_i$ and $t$.
The braided embedding $g_1$ induces a braided embedding $g_2:\tilde{B}\times S^1\hookrightarrow (B\times S^1)\times D^2$, which in turn gives rise to the braid monodromy map $\psi_2:\pi_1(B\times S^1)\cong \pi_1(B)\times \mathbb{Z}\rightarrow B_n$, where the covering $p_2$ (i.e. restriction of $p$ to $\tilde{B}\times S^1$) is $n$-sheeted.
To define the braided embedding $h$ of $\nu \tilde{B}$, we will first construct a braided embedding lifting $p|_{\tilde{B}} :\tilde{B}\rightarrow B$ induced from the braided embedding $g_2$.

Recall from Subsection \ref{centralizer}, we have a map $\Theta:Z(\psi_2(\mu))\rightarrow Z_0(\overline{\psi_2(\mu)})$ sending a braid commuting with $\psi_2(\mu)$ to a braid commuting with the associated tubular braid $\overline{\psi_2(\mu)}$.
To this end, note that for each $1\leq i \leq t$ of the braids $\psi_2(x_i)$ commute with $\psi_2(\mu)$, and consequently the image $\psi_2(\pi_1(B)\times \{0\})$ is contained in $Z(\psi_2(\mu))$. Thus we get a well defined group homomorphism $\psi_3:\pi_1(B)\rightarrow B_m$ defined on the generators by sending $x_i\mapsto \Theta \circ \psi_2(x_i)$. This braid monodromy induces a braided embedding:
$g_3:\tilde{B}\rightarrow B\times D^2\cong \nu B$. Taking products with a disc, we obtain a proper braided embedding $g_4:\nu \tilde{B}\rightarrow \nu B\times D^2$ lifting $p|_{\nu \tilde{B}}:\nu \tilde{B}\rightarrow \nu B$, and restricting $g_4$ to the  boundary of the normal bundles we obtain the braided embedding $g_5: \tilde{B}\times S^1\rightarrow (B\times S^1)\times D^2$, with braid monodromy defined by $\psi_5:\pi_1(\partial \nu B)\rightarrow B_n$, so that for any $\gamma\in\pi_1(\partial \nu B)$ with the tubular braids under $\psi_2$ and $\psi_5$ agreeing, i.e. $\overline{\psi_2(\gamma)}=\overline{\psi_4(\gamma)}$;
although the interior braids (and hence $\psi_2(\gamma)$ and $\psi_5(\gamma)$) may differ in the following way. The interior braids of $\psi_5$ are all identity, since we uniformly took products with a disc.  By hypothesis, we know that the interior braids $\{\alpha_k\}$ of $\alpha:=\psi_2(\mu)$ is a standard positive or negative braid, i.e. conjugate to $(\sigma_1...\sigma_k)^{\pm 1}$ for some $k$. We know that centralizer of the above braid is the cyclic subgroup generated by itself. Thus for any $i$, the $k$-th interior braids of $\psi_2(x_i)$ has to be powers of $\alpha_k$.


The idea now is to twist the braided embedding $g_4$ so that the induced braided embedding $g_5$ agrees with that of $g_2$. Let us first consider the special case $B$ is a circle $S^1$, with fundamental group $\langle x \rangle\cong \mathbb{Z}$. $\tilde{B}$ will be a disjoint union of circles, and each component of the normal bundle $\nu \tilde{B}$ will be a solid torus.
In order to construct a braided embedding $h:\nu\tilde{B}\hookrightarrow \nu B\times D^2$ so that the braid monodromy of which agrees with that of $\psi_2$, we will modify $g_4$ by applying Dehn twists on those components of $\nu\tilde{B}$ so that the interior braids match with that of $\psi_2(x)$.

\begin{figure}[H]
    \centering
    \includegraphics[width=11 cm]{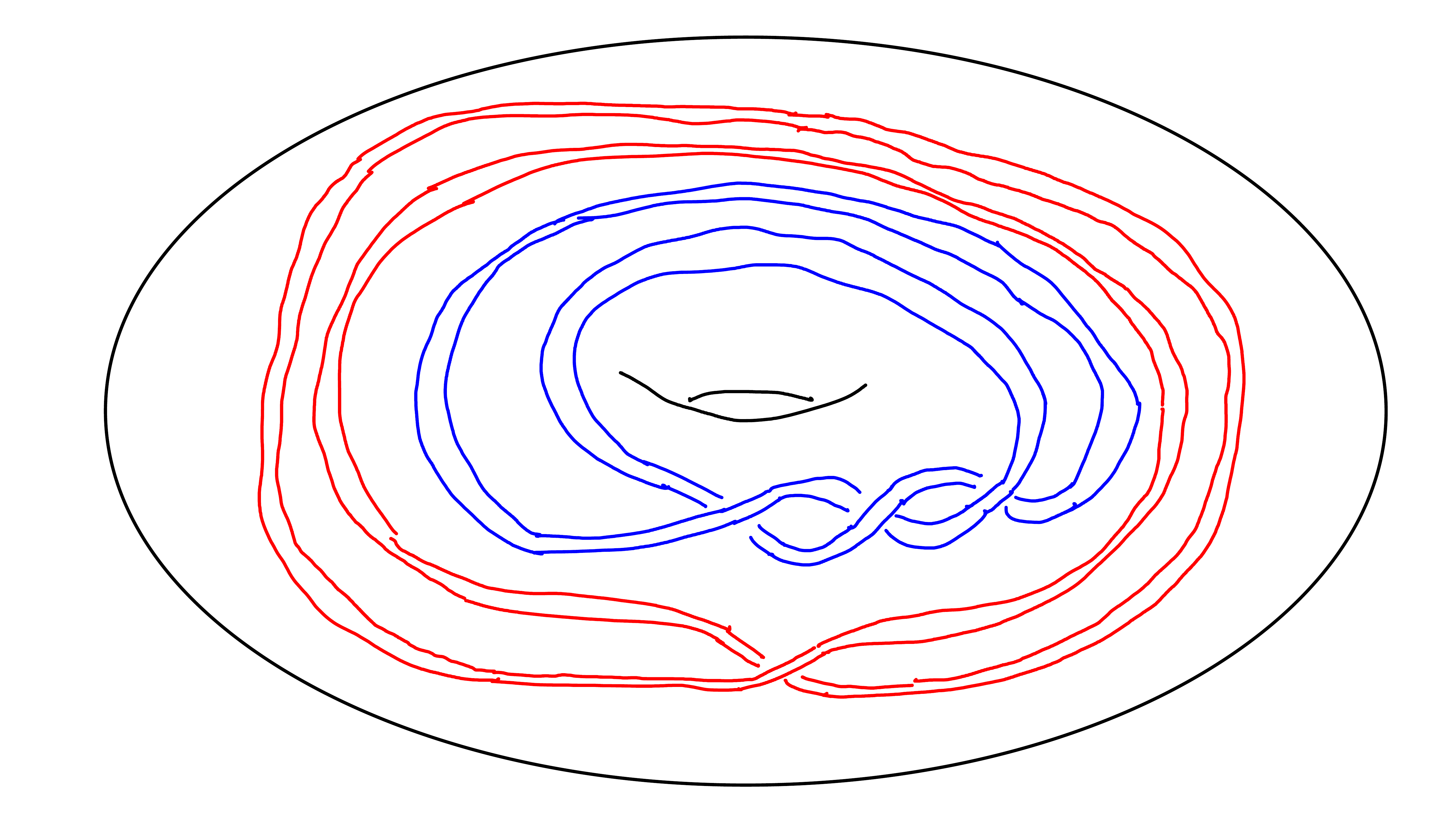}
    \caption{When $B$ is $S^1$, the figure illustrates a braided embedding of $\nu\tilde{B}$ in $B\times D^2$}
    \label{exten}
\end{figure}
More formally, we will pre-compose 
the braided embedding $g_4$ with a diffeomorphism of $\nu \tilde{B}$ which preserves each solid tori setwise, and induces a number (this number is equal to the exponent of the corresponding interior braid of $\psi_2(x)$) of Dehn twist\footnote{Recall for any diffeomorphism of the boundary of a solid torus which preserves the meridian can be extended to a diffeomorphism of the entire solid torus.} on the boundary of the solid torus. The braided embedding $h$ so constructed has the property that the induced braided embeddings on the boundary $h_1$ is isotopic to that of $g_5$ (since the braid monodromies agree), and thus we can patch up $h$ and $g_1$ to extend the braided embedding over the branch locus.

To make the above idea work in general we need to come up with an analogue of Dehn twists in higher dimension, and ensure we can carry out a similar construction as above even when the fundamental group of $B$ is complicated. To elaborate on the latter point, suppose we did some twisting of $g_4$ so the interior braid corresponding to $x_1$ agrees with that of $\psi_2(x_1)$, but now if we do some twisting so that the interior braids of $x_2$ agree, we need to make sure this does not alter the interior braids of $x_1$.

Let us begin with the analogues of Dehn twists, which we will call D-twists.

\begin{defi}
Suppose $X$ is any smooth manifold, and $Y$ is any 
 hypersurface of $X$ with a tubular neighbourhood $Y\times [0,2\pi]$. Then we can define a diffeomorphism of $X\times S^1$ as follows: we cut $X\times S^1$ along $Y\times S^1$, and define $$Y\times [0,2 \pi]\times S^1\rightarrow Y\times [0,2 \pi]\times S^1,$$ by sending:
 $$(y,\theta,z)\mapsto (y,\theta,e^{i\theta}z).$$
 we see that this map agrees with the identity at the boundary $Y\times \{0,2\pi\}\times S^1$, and as such we can extend this to the rest of $X\times S^1$ by identity. We will call this \textit{D-twist} of $X\times S^1$ along $Y\times S^1$ and denote it by $T_Y$.
 
 Also, we observe that this map $T_Y$ extends to a diffeomorphism $S_Y$ of $X\times \overline{D^2}$ by essentially the same formula:
 let us define $$Y\times [0,2 \pi]\times \overline{D^2}\rightarrow Y\times [0,2 \pi]\times \overline{D^2}$$ 
by  $$(y,\theta,z)\mapsto (y,\theta,e^{i\theta}z),$$ and the identity map elsewhere.

\end{defi}
Notice that when $X$ is the unit circle and $Y$ is a single point in $X$, the D-twist $T_Y$ is exactly the Dehn twist $T_Y\times S^1$ along a meridian, and $S_Y$ is the extension of this diffeomorphism to the entire solid torus.

Returning to the case $X$ being a manifold, if $\gamma$ is a simple closed curve in $X$ which intersects $Y$ in some finite set of points $y_1,...,y_l$, then we claim that the effect of $T_Y$ on the torus $\gamma\times S^1$ is the Dehn twist along the meridian $\{y_1\}\times S^1$ to the power $\langle \gamma, Y \rangle$ times, where $\langle \gamma, Y \rangle$ denotes the algebraic intersection number of $\gamma$ with $Y$. For, if there is one intersection it is easy to see that the only change happens near that point and it is a positive or negative Dehn twist about the meridian, depending on the sign of the intersection. If there are multiple such intersections, we will get several meridional Dehn twists, with signs corresponding to that of the intersection.
 
 \textbf{Special Case}: Let us assume the map $p:\tilde{B}\rightarrow B$ is a diffeomorphism. We would like to realize the braided embedding $g_2:B\times S^1\rightarrow (B\times S^1)\times D^2$ as the boundary of braided embedding $h:\nu B \rightarrow \nu B\times D^2$.
 In this context, the closure of $\beta=\psi_2(\mu)$ is a standard unknot, and we know that the centralizer of $\psi_2(\mu)$ is an infinite cyclic group generated by itself. Thus if we restrict the braid monodromy map $\psi_2:\pi_1(\partial\nu B)$ to the subgroup $\pi_1(B)$
 (obtained by fixing a particular point in $S^1$), we see that the restriction maps from $\pi_1(B)$ to $Z(\beta)\cong\mathbb{Z}$. Thus it must
  factor through the first homology $H_1(B)$, or in other words is an element of $Hom(H_1(B),\mathbb{Z})$. By the universal coefficient theorem for cohomology, this element is the image of some cohomology class in $H^1(B)$ (there may be multiple pre-images depending on if the corresponding Ext term is non-trivial). By Poincare duality there is a homology class in $H_{\dim B-1}(B)$ dual to it, which is represented by an embedded \cite{Kir} hypersurface $Y$. It follows that if we pre-compose the braided embedding $g_4:\nu\tilde{B}\hookrightarrow \nu B\times D^2$ with the diffeomorphism $S_Y:\nu \tilde{B}\rightarrow \nu \tilde{B}$ (recall $B$ and $\tilde{B}$ are diffeomorphic), we obtain a braided embedding $h$ with the required properties.

\textbf{General Case}:
Given the braided embedding $g_2:\tilde{B}\times S^1\hookrightarrow (B\times S^1)\times D^2$, we had its braid monodromy map
 $\psi_2:\pi_1(B\times S^1)\rightarrow B_n$. Let $C$ be one component of $\tilde{B}$, by restricting $g_2$ to the component $C\times S^1$, we get a braided embedding $g_2|_{C\times S^1}:C\times S^1\hookrightarrow (B\times S^1)\times D^2$, and we claim that this braided embedding is the composition of two braided embeddings:\\
 $s: C\times S^1\hookrightarrow (C\times S^1)\times D^2$ (the associated cover $id\times p$ unwraps only the $S^1$-direction), and the braided embedding $C\times S^1\hookrightarrow (B\times S^1)\times D^2$ (the associated  cover preserves the $S^1$ direction, and we get a covering $p|_{C}:C\rightarrow B$ in the orthogonal direction). To see this note that if we restrict the braid monodromy $\psi_2:\pi_1(B\times S^1)\rightarrow B_n$ to the subgroup $\pi_1(C\times S^1)$ obtained from (corresponding to the covering map $p|_{C}\times id:C\times S^1\rightarrow B\times S^1)$, if we ignore the tubular braids not corresponding to $C$, we see that the tubular braids are all identity (and the number of such strands of the tubular braid is the number of sheets of the cover $p|_{C}:C\rightarrow B$), and the various interior braids are conjugate (the number of strands in an interior braid equals the number of times the $S^1$ factor is unwrapped, say $l$). If we choose a particular base point in $C$ over the base point in $B$, the interior braids coming from that factor gives us a braid monodromy $\psi_6:\pi_1(C\times S^1)\rightarrow B_l$. It suffices to show that we can realize this braid monodromy by twisting, i.e. we will precompose each $\nu C$ with an appropriate $D$-twist realizing appropriate interior braids, and doing this for each component $C$ of $\tilde{B}$ will give us a braided embedding $h$ with the required properties.
 Thus, we reduce to the special case considered above when $p:\tilde{B}\rightarrow B$ is a diffeomorphism, and the result follows. \end{proof}
 
 



\subsection{Piecewise Linear Category}

We have a similar result in the piecewise linear category, when the branch locus is a submanifold. The reader will note that this results requires braids surrounding branch points to be completely split unlinks (as opposed to standard unlinks), which is to be expected given our local model near a branch point in the last section.

\begin{prop}
Suppose we have a piecewise linear branched cover $p:M\rightarrow N$ with branch locus $B\subseteq N$ being a submanifold with trivial regular neighbourhood $\nu B\cong B\times D^2$, and $p|_{\tilde{B}}:\tilde{B}\rightarrow B$ is a covering map. Suppose we choose points $b_1,...,b_k$, one for each connected component of $B$. If we are given a locally flat piecewise linear braided proper embedding  $g_1:M\setminus \nu_0 \tilde{B}\hookrightarrow (N\setminus \nu_0 B)\times D^2$ lifting the honest covering $p_1:M\setminus \nu_0 \tilde{B}\hookrightarrow N\setminus \nu_0 B$ induced from $p$, then $g_1$ extends to a locally flat smooth braided embedding $g:M\hookrightarrow N\times D_2$ lifting $p$ if and only if each of the braids surrounding the branch points $b_i$ are completely split unlinks.

\end{prop}

The proof will be similar to that of the smooth category, once again there will be a special case, and a reduction from the general case to the special case. The latter step is essentially the same as the proof in the smooth category, however we will need to approach the special case differently, as the author does not know any sort of classification which braids closures are unknots, let alone a precise understanding of their centralizers.

\begin{proof}
By our study of local models in the last section, it is clear that the hypothesis (of braids surrounding branch points be completely split unlinks) is a necessary condition. It suffices to show it is a sufficient condition as well.

 \textbf{Special Case}: let us assume the map $p:\tilde{B}\rightarrow B$ is the identity map. We would like to realize the braided embedding $g_2:B\times S^1\rightarrow (B\times S^1)\times D^2$ as the boundary of braided embedding. We will obtain $h$ simply by coning $g_2$.  Thinking of a slice $D^2_\nu$ of the normal bundle as the cone $p*S^1$, we see that $B\times D^2_\nu\times D^2$ is nothing but the parametric join $B\times S^1\times D^2$ along $B\times p$.
 
 Let us choose the origin $O$ in the disc $D^2$, we see that for each $b\in B$,the image of $g_2$ in $\{b\}\times S^1\times D^2$ is an unknot (because we are in the special case), we can cone this at the point $\{b\}\times \{p\}\times\{O\}$, and obtain a locally flat braided disc in $\{b\}\times D^2_\nu\times D^2$. By doing this coning operation for each $b\in B$, we obtain a locally flat piecewise linear proper braided embedding $h$ which induces the braided embedding $g_2$ in the boundary, and can be used to extend $g_1$.
 
 \textbf{General Case}: Without loss of generality we may assume $B$ is connected (we can run the same argument for each component of $B$), however $\tilde{B}$ may have multiple components. In this case we need to choose an appropriate number of points $O_i$, and cone over them. However, we need to make sure that this procedure gives us an embedding, we would like to vary the points $O_i$ in $D^2$ (continuous) parametrically in $B$ so the result of the coning is injective. To do this formally, we will make use of the tubular braids.
 Given braided embedding $g_2:\tilde{B}\times S^1\hookrightarrow (B\times S^1)\times D^2$, just like in the smooth category, we obtain a braided embedding $g_3:\tilde{B}\hookrightarrow B\times D^2$ by looking at the tubular braids of the braid monodromy $\psi_2$ of $g_2$ (we have less control over the interior braids in this case, but the tubular braids behave similarly). By looking at a small neighbourhood of this braided embedding, we again obtain a braided embedding $g_4:\nu \tilde{B}\rightarrow \nu B\times D^2$ lifting $p|_{\nu \tilde{B}}:\nu \tilde{B}\rightarrow \nu B$. We can replace the braided embeddings on the boundary of this untwisted braided embedding $g_4$ with the one coming from $g_2$, and use the small neighbourhoods in the disc $D^2$ (i.e, the points $O_i$ is determined by the braided embedding $g_3$) to do the coning operation.
 The result follows. \end{proof}

\section{Lifting branched coverings over the two-sphere}

\subsection{Braid Systems and Permutation Systems}
Since the fundamental group of a sphere with $m$ punctures has the presentation $\langle x_1,...,x_m|x_1....x_m\rangle$ where the $x_i$ is the loop surrounding the $i$-th puncture, we can store permutation and braid monodromies of the punctures sphere as tuples.
Let $G$ be any group with any subset $H$, let us define
$$P^m_g(G,H):=\{(h_1,...,h_m)|h_i\in H, h_1...h_m=g\}$$
We omit $H$ from the notation if $H=G$, and we omit $g$ from the notation if $g=e$.\\
An element of $P^m(S_n)$ will be called a \emph{permutation system}.
An element of $P^m(B_n,A_n)$ will be called a \emph{braid system}.
An element of $P^m(B_n,SA_n)$ will be called a \emph{simple braid system}. Here we recall that the notations $A_n$ and $SA_n$ were introduced in Subsection~\ref{csplit}.

\subsection {Hurwitz sliding moves} Let us consider the sliding moves $s_k:P^n_g(G)\rightarrow P^n_g(G)$ defined by:
$$(a_1,...,a_k,a_{k+1},...,a_m)\mapsto
(a_1,...,a_ka_{k+1}a_k^{-1},a_k,...,a_m)$$
and its inverse $s_k^{-1}:P^n_g(G)\rightarrow P^n_g(G)$

$$(a_1,...,a_k,a_{k+1},...,a_m)\mapsto
(a_1,...,a_{k+1},a^{-1}_{k+1}a_k a_{k+1},...,a_m)$$

\textit{Notation}: We will use the following notation to describe operations on  braid/permutation systems
\begin{itemize}
\item $\rightarrow$ sliding move (including inverse), \item $\downarrow$ deleting a subsystem,\item $\uparrow$ inserting a subsystem.
\end{itemize}


\subsection{Lifting simple branched covers}
Let us begin by discussing liftings of two fold branched covers over the two sphere.

\begin{example}\label{2fbc2s}
The monodromy of the branched covering described in Figure \ref{qhu} can be described by the permutation system $((12),(12),(12),(12),(12),(12),(12),(12)).$
One possible lift is given by the braid system
$(\sigma_1,\sigma_1,\sigma_1,\sigma_1,\sigma_1^{-1},\sigma_1^{-1},\sigma_1^{-1},\sigma_1^{-1}).$ We observe that this example generalizes to arbitrary genus, and any two fold branched cover over the sphere lifts to a braided embedding. 
\end{example}

For simple branched covers over the sphere, Carter and Kamada \cite[Theorem~3.8]{CK}  answered Question~\ref{Q1} concerning liftability of branched covers to braided embeddings affirmatively,
 using the following classification of simple branched covers due to L\"{u}roth \cite{L} and Clebsch \cite{C}  (see \cite[Section~4]{BE} for a proof in English). 
\begin{prop}[L\"{u}roth \cite{L} and Clebsch \cite{C}]\label{p1} Any transposition (that is simple permutation) system can be brought to the form 
$$((12),...,(12),(13),(13),(14),(14),...,(1n),(1n))$$ (with an even number of $(12)$'s) using some sliding moves and conjugation.
\end{prop}

Once a permutation system is brought to this standard form, it is easy to find a braid system lifting it. If we recursively define $\alpha_1:=\sigma_1$, and $\alpha_k:=\sigma_k \alpha_{k-1}\sigma_{k}^{-1}$, we see that $\alpha_{k-1}$ is a braid lifting $(1k)$ which is conjugate to the standard generator $\sigma_1$. Then we see that the braid system 
$(\alpha_1,\alpha_1^{-1},...,\alpha_1,\alpha_1^{-1},\alpha_2,\alpha_2^{-1},\alpha_3,\alpha_3^{-1},...,\alpha_{n-1},\alpha_{n-1}^{-1})$ lifts the permutation system \\ $((12),(12),...,(12),(12),(13),(13),(14),(14),...,(1n),(1n))$.
After this one can apply the sliding and conjugation moves in reverse to the braid system and finally get a braid system lifting the original permutation system.

 We will build on L\"{u}roth and Clebsch's method of proof to show that the answer Question~\ref{Q1} is yes for any branched cover over the 2-sphere.
Before doing that, let us illustrate the idea of proof of Proposition~\ref{p1} with an example, and then show that we can then get a braid system lifting the permutation system by applying the reverse process to the braid system.
\begin{example}
Let us consider a 4-fold simple branched cover with Hurwitz permutation system  $((12),(34),(13),(24),(14),(23))$. We will perform sliding moves
to get the permutation system to standard form
$$((12),(34),(13),(24),(14),(23))\xrightarrow{s_5}
((12),(34),(13),(24),(23),(14))$$
$$\xrightarrow{s_4^{-1}}((12),(34),(13),(23),(34),(14))\xrightarrow{s_2^{-1}}
((12),(13),(14),(23),(34),(14))$$
$$\xrightarrow{s_3}((12),(13),(23),(14),(34),(14))\xrightarrow{s_4}
((12),(13),(23),(13),(14),(14))$$
$$\xrightarrow{s_2}((12),(12),(13),(13),(14),(14))$$

Now it is easy to find a braid system that lifts the given permutation system once it is in standard form.

The braids $\alpha_1=\sigma_1$, $\alpha_2=\sigma_2\sigma_1\sigma_2^{-1}$
and $\alpha_3=\sigma_3\sigma_2\sigma_1\sigma_2^{-1}\sigma_3^{-1}$
lift $(12),(13)$ and $(14)$, respectively.
Clearly each $\alpha_i\in SA_4$ being a conjugate of $\sigma_1$. Now we see that the braid system 
$(\alpha_1,\alpha_1^{-1},\alpha_2,\alpha_2^{-1},\alpha_3,\alpha_3^{-1})$ lifts $((12),(12),(13),(13),(14),(14))$. We can apply the inverses of the sliding moves we applied earlier to $(\alpha_1,\alpha_1^{-1},\alpha_2,\alpha_2^{-1},\alpha_3,\alpha_3^{-1})$ and get
$$(\alpha_1,\alpha_1^{-1},\alpha_2,\alpha_2^{-1},\alpha_3,\alpha_3^{-1})\xrightarrow{s_2^{-1}}(\alpha_1,\alpha_2,\alpha_2^{-1}\alpha_1^{-1}\alpha_2,\alpha_2^{-1},\alpha_3,\alpha_3^{-1})=(\alpha_1,\alpha_2,\sigma_2^{-1},\alpha_2^{-1},\alpha_3,\alpha_3^{-1})$$
$$\xrightarrow{s_4^{-1}}(\alpha_1,\alpha_2,\sigma_2^{-1},\alpha_3,\alpha_3^{-1}\alpha_2^{-1}\alpha_3,\alpha_3^{-1})
=(\alpha_1,\alpha_2,\sigma_2^{-1},\alpha_3,\sigma_3^{-1},\alpha_3^{-1})$$
$$\xrightarrow{s_3}(\alpha_1,\alpha_2,\sigma_2^{-1}\alpha_3\sigma_2,\sigma_2^{-1},\sigma_3^{-1},\alpha_3^{-1})=(\alpha_1,\alpha_2,\alpha_3,\sigma_2^{-1},\sigma_3^{-1},\alpha_3^{-1})$$
$$\xrightarrow{s_2}(\alpha_1,\alpha_2\alpha_3\alpha_2^{-1},\alpha_2,\sigma_2^{-1},\sigma_3^{-1},\alpha_3^{-1})=(\alpha_1,\sigma_3,\alpha_2,\sigma_2^{-1},\sigma_3^{-1},\alpha_3^{-1})$$
$$\xrightarrow{s_4}(\alpha_1,\sigma_3,\alpha_2,\sigma_2^{-1}\sigma_3^{-1}\sigma_2,\sigma_2^{-1},\alpha_3^{-1})$$
$$\xrightarrow{s_5^{-1}}(\alpha_1,\sigma_3,\alpha_2,\sigma_2^{-1}\sigma_3^{-1}\sigma_2,\alpha_3^{-1},\alpha_3\sigma_2^{-1}\alpha_3^{-1})=(\alpha_1,\sigma_3,\alpha_2,\sigma_2^{-1}\sigma_3^{-1}\sigma_2,\alpha_3^{-1},\sigma_2^{-1})$$
Thus we see that $(\alpha_1,\sigma_3,\alpha_2,\sigma_2^{-1}\sigma_3^{-1}\sigma_2,\alpha_3^{-1},\sigma_2^{-1})$ lifts the given permutation system\\ $((12),(34),(13),(24),(14),(23))$.

\end{example}

For closed two braids, being a completely split unlink is equivalent to being a completely split standard unlink, whence the lifting question for simple branched covers are equivalent for the piecewise linear and smooth categories. However, this does not hold more generally, so let us consider the two categories separately.

\subsection{Lifting in the piecewise linear category}

\begin{thm}\label{bcl2s} Every branched cover of a surface over $S^2$ can be lifted to a piecewise linear braided embedding. 
\end{thm}
 The proof will follow from the following slightly stronger Proposition~\ref{p2} below, which in particular shows that one can lift to a braided embedding which can be perturbed to be simple.\\

\textit{Notation}: For $n<m$, there are canonical inclusions $\iota_{n,m}:S_n\hookrightarrow S_{m}$
and $i_{n,m}:B_n\hookrightarrow B_{m}$, and for notational convenience, we will be implicitly using these maps to make identifications. For example if $\rho\in \iota_{n,m}(S_n)$ then we can think of $\rho\in S_{n}$, and conversely any element of $S_n$ can be thought to be an element of $S_{m}$ (and similarly for the braid groups).

\begin{prop}\label{p2}  Every  permutation system $(\rho_1,...,\rho_m)$ in $S_n$  lifts to a braid system $(\alpha_1,...,\alpha_m)$ in $B_n$ so that
\begin{enumerate}
\item If $\rho_j\in S_k$ with smallest such $k$, $\alpha_j$ has a braid word of the form $\beta_j \sigma_{k-1}^{\pm 1}\gamma_j$, where $\beta_j$ and $\gamma_j$ are in $B_{k-1}$. 
\item Moreover if  $\rho_j\in S_k$ with smallest such $k$ is a transposition, then 
$\alpha_j$ has a braid word of the form $\beta_j \sigma_{k-1}^{\pm 1}\beta_j^{-1}$, where  $\beta_j$ is in $B_{k-1}$. 
\end{enumerate}
\end{prop}
\begin{proof} The proof will be by induction on $n$.\\

\emph{Base case}: $n=2$. Every permutation system in $S_2$ looks like $((12),...,(12))$ with a even number of $(12)$'s. The braid system 
$(\sigma_1,\sigma_1^{-1},...,\sigma_1,\sigma_1^{-1})$ lifts $((12),(12),...,(12),(12))$ satisfying the above conditions.\\

\emph{Inductive step}: Let us assume the statement holds for $q=n-1$. Suppose we have the permutation system $(\rho_1,...,\rho_m)$ in $S_n$.
For each $\rho_i$, let us factorize $\rho_i=\varrho_i \tau_i$ so that $\varrho_i\in S_{n-1}$, and $\tau_i$ is a transposition of the form $(a,n)$, where $a\in\{1,2,...,n-1\}$ (such a factorization is not unique, and certain $\varrho_i$ or $\tau_i$ can be the identity, in which case we can drop it from the permutation system). We will do fission (that is split certain permutations, for example break up a three cycle into two transpositions) to the original permutation system to obtain $(\varrho_1,\tau_1...,\varrho_m,\tau_m)$. We will use inverse sliding moves to bring all the $\varrho_i$'s to the left of all the $\tau_i$, as follows:
$$(\varrho_1,\tau_1,...,\varrho_{m-1},\tau_{m-1},\varrho_m,\tau_m)\xrightarrow{s_{2m-2}^{-1}}
(\varrho_1,\tau_1...,\varrho_{m-1},\varrho_m,\tilde{\tau}_{m-1},\tau_m)\xrightarrow{s_{2m-4}^{-1}s_{2m-3}^{-1}}...$$ 
$$\xrightarrow{s_2^{-1}...s_{m}^{-1}}(\varrho_1,...,\varrho_{m-1},\varrho_m,\tilde{\tau}_{1},...,\tilde{\tau}_{m-1},\tilde{\tau}_m)$$
where each $\tilde{\tau_i}$ is a conjugate of $\tau_i$ and is of the form $(a,n)$.
Now we will break this permutation system into  permutation tuples $\Theta=(\varrho_1,...,\varrho_{m-1},\varrho_m)$, $\Phi=(\tilde{\tau}_{1},...,\tilde{\tau}_{m-1},\tilde{\tau}_m)$, and apply sliding moves and its inverses to modify $\Phi$, as described below.

$\dagger$ Let $r$ be the largest number smaller than $n$ so that there is a transposition in $\Phi$ so that $(r,n)=\tilde{\tau}_{i}$ for some $i$.
Let $i_1<i_2<...<i_j$ be the indices so that $\tilde{\tau}_{i}=(r,n)$. We use sliding moves to bring $\tilde{\tau}_{i_{2k-1}}$ to 
the ${i_{2k}-1}$' th spot, i.e. bring it to the very left of $\tilde{\tau}_{i_{2k}}$. If $j$ is even, we bring $\tilde{\tau}_{i_j}$
to the extreme right of the permutation tuple.
After applying this sequence of sliding moves to $(\tilde{\tau}_{1},...,\tilde{\tau}_{m-1},\tilde{\tau}_m)$, the permutation system will now look like (we illustrate the case when $j$ is odd, where there will be a single $(r,n)$ at the very right):

$$(\mu_1,\nu_1,(r,n),(r,n),\mu_2,\nu_2,(r,n),(r,n),\mu_3,\nu_3,...,\mu_k,\nu_k,(r,n))$$
where $ \mu_i$'s  are permutation tuples with each element of the form $(a,n)$ for some $a<n$, and $\nu_i$'s are permutation tuples not containing any $n$, so each element is of the form $(b,r)$ for some $b<r$. Now we consider the new permutation tuple which is formed by deleting the sub-permutation systems $((r,n),(r,n))$
$$\downarrow(\mu_1,\nu_1,\mu_2,\nu_2,\mu_3,\nu_3,...,\mu_k,\nu_k,(r,n)) $$

We can then use inverse sliding moves on this permutation tuple to bring the $\nu_j$'s to the left 
$$\rightarrow (\nu_1,\nu_2,\nu_3,...,\nu_k,\tilde{\mu}_1,\tilde{\mu}_2,\tilde{\mu}_3,...,\tilde{\mu}_k,(r,n))$$

Let us now append $(\nu_1,\nu_2,\nu_3,...,\nu_k)$ to the right of $\Theta$ and set $\Phi=(\tilde{\mu}_1,\tilde{\mu}_2,\tilde{\mu}_3,...,\tilde{\mu}_k,(r,n))$, and apply the same procedure (beginning in $\dagger$) to it. At each step the length of this permutation tuple reduces by at least $2$, and in a finite number of steps $\Phi$ will be empty. At that stage $\Theta$ will be a permutation system in $S_{n-1}$, and by the induction hypothesis we can find a braid system lifting it with the stated properties.\\

Now we can apply to this braid system the reverse of the entire process we applied to the permutation system, i.e.
\begin{itemize} 
\item we will introduce the braid system $(\eta_r,\eta_r^{-1})$ 
corresponding to the places we deleted the permutation system $((r,n),(r,n))$, where $$\eta_r:=(\sigma_r...\sigma_{n-2})\sigma_{n-1}(\sigma_r...\sigma_{n-2})^{-1}=\sigma_r...\sigma_{n-2}\sigma_{n-1}\sigma_{n-2}^{-1}...\sigma_r^{-1}$$
\item we will apply $s_k^{\mp 1}$ to the braid system if we applied $s_k^{\pm 1}$ to the permutation system.\\
\end{itemize}

After we apply all these moves we will be left with a braid system $$(\alpha_1,{\delta}_{1},...,\alpha_{m-1},{\delta}_{m-1},\alpha_m,\delta_m)$$
lifting the permutation system $$(\varrho_1,\tau_1,...,\varrho_{m-1},\tau_{m-1},\varrho_m,\tau_m),$$ 
\begin{claim}
This braid system has Properties 1 and 2 as in the statement.
\end{claim}
Properties 1 and 2 for the $\alpha_i$'s will follow from the induction hypothesis.
The following observations will show that the $\delta_i$'s have the Properties 1 and 2 as in the statement.
\begin{itemize}
\item 

The first time we introduce the braid $\eta_r$, it is of the form $\beta \sigma_{n-1}^{\pm 1}\beta^{-1}$ where $\beta\in B_{n-1}$, and when we apply sliding moves or their inverses to it, we conjugate it by an element of $B_{n-1}$, and so it remains of that form.

\item Note that the only times we applied the sliding moves of the form $((r,n),(s,n))\rightarrow ((s,r),(r,n))$, it was the case that $s<r$. Moreover, since $(s,r)$ is a permutation in $S_{n-1}$, we ensured that we applied inverse sliding moves to it to bring it to $\Theta$, and then used the induction hypothesis to find a braid $\beta \sigma_{r-1}^{\pm 1}\beta^{-1}$ lifting it, where $\beta\in B_{r-1}$. While applying the reverse procedure this braid does not change until it becomes adjacent to the braid  $\eta_r^{\pm 1}$ lifting $(r,n)$, and then we apply the inverse  sliding move $$(\beta \sigma_{r-1}^{\pm 1}\beta^{-1},\eta_r^{\pm 1})\rightarrow 
(\eta_r^{\pm 1}, \eta_r^{\mp 1}\beta \sigma_{r-1}^{\pm 1}\beta^{-1}\eta_r^{\pm 1})$$
Now we will show that $\eta_r^{\mp 1}\beta \sigma_{r-1}^{\pm 1}\beta^{-1}\eta_r^{\pm 1}$ is of the form $\gamma \sigma_{n-1}^{\pm 1}\gamma^{-1}$,
where $\gamma\in B_{n-1}$, by repeatedly applying (equivalent form of) the braid relation
$$\sigma_{i+1}^{-1}\sigma_{i}^{\pm 1} \sigma_{i+1}=\sigma_i\sigma_{i+1}^{\pm 1} \sigma_i^{-1}. $$
We will consider the case that the exponent of $\eta_r$ above is $1$, the other case is similar.

$$\eta_r^{-1}\beta \sigma_{r-1}^{\pm 1}\beta ^{-1}\eta_r^{}=\beta \eta_r^{-1}  \sigma_{r-1}^{\pm 1}\eta_r^{}\beta ^{-1}$$
$$=\beta(\sigma_{r}\sigma_{r+1}...\sigma_{n-2}\sigma_{n-1}^{-1}\sigma_{n-2}^{-1}...\sigma_{r+1}^{-1}\sigma_r^{-1})\sigma_{r-1}^{\pm 1}(\sigma_r\sigma_{r+1}...\sigma_{n-2}\sigma_{n-1}\sigma_{n-2}^{-1}...\sigma_{r+1}^{-1}\sigma_r^{-1} )\beta ^{-1}$$
$$=\beta\sigma_{r}\sigma_{r+1}...\sigma_{n-2}\sigma_{n-1}^{-1}\sigma_{n-2}^{-1}...\sigma_{r+1}^{-1}(\sigma_r^{-1}\sigma_{r-1}^{\pm 1}\sigma_r)\sigma_{r+1}...\sigma_{n-2}\sigma_{n-1}\sigma_{n-2}^{-1}...\sigma_r^{-1} \beta ^{-1}$$
$$=\beta(\sigma_{r}\sigma_{r+1}...\sigma_{n-2})\sigma_{n-1}^{-1}\sigma_{n-2}^{-1}...\sigma_{r+1}^{-1}\sigma_{r-1}\sigma_{r}^{\pm 1}\sigma_{r-1}^{-1}\sigma_{r+1}...\sigma_{n-2}\sigma_{n-1}(\sigma_{n-2}^{-1}...\sigma_{r+1}^{-1}\sigma_r^{-1}) \beta ^{-1}$$
$$=\beta(\sigma_{r}\sigma_{r+1}...\sigma_{n-2})\sigma_{r-1}\sigma_{n-1}^{-1}\sigma_{n-2}^{-1}...(\sigma_{r+1}^{-1}\sigma_{r}^{\pm 1}\sigma_{r+1})...\sigma_{n-2}\sigma_{n-1}\sigma_{r-1}^{-1}(\sigma_{n-2}^{-1}...\sigma_{r+1}^{-1}\sigma_r^{-1}) \beta ^{-1}$$
$$=...=\beta(\sigma_{r}\sigma_{r+1}...\sigma_{n-2})(\sigma_{r-1}\sigma_{r}...\sigma_{n-2})\sigma_{n-1}^{\pm 1}(\sigma_{n-2}^{-1}...\sigma_{r}^{-1}\sigma_{r-1}^{-1})(\sigma_{n-2}^{-1}...\sigma_{r+1}^{-1}\sigma_r^{-1}) \beta ^{-1}$$
\end{itemize}

Assuming the claim, we see that the braids $\alpha_i\in B_{n-1}$ and $\alpha_i\delta_i\in B_n$ have the same closure since they are related by conjugation and stabilization,
$$\alpha_i\doteqdot\gamma_i^{-1}\alpha_i\gamma_i\nearrow \gamma_i^{-1}\alpha_i\gamma_i\sigma_{n-1}^{\pm 1}\doteqdot \alpha_i\gamma_i\sigma_{n-1}^{\pm 1}\gamma_i^{-1}=\alpha_i\delta_i $$
Here, following Morton, we are denoting conjugation by $\doteqdot$ and stabilization by $\nearrow$.
Thus the braid system $$(\alpha_1{\delta}_{1},...,\alpha_{m-1}{\delta}_{m-1},\alpha_m\delta_m)$$ lifts the permutation system $(\rho_1,...,\rho_{m-1},\rho_m)$ with all the required properties.  \end{proof}

The above proof is notationally inconvenient, so let us work out some examples explicitly.

\begin{example} Let us consider the permutation system $\rho=((123),(24),(14)(23),(34))$.
We use fission on each permutation containing 4 to construct the new permutation system $$((123),(24),(23),(14),(34)).$$
Now we will use sliding moves to move each transposition containing 4 to the right.

$$((123),\textcolor{purple}{(24),(23)},(14),(34))\xrightarrow{s_2^{-1}}((123),(23),(34),(14),(34))$$

Now we will just focus on the transpositions on the right containing $4$ and use  sliding moves there.
$$((123),(23),\textcolor{blue}{(34),(14)},(34))\xrightarrow{s_3}((123),(23),(13),(34),(34))$$

We now observe that $$(\sigma_2\sigma_1, \sigma_2^{-1},\sigma_1^{-1} \sigma_2^{-1}\sigma_1, \sigma_3, \sigma_3^{-1})$$ is a braid system lifting $$((123),(23),(13),(34),(34))$$ with the required properties as in statement.
Now we apply $s_3^{-1}$ to this braid system.
$$(\sigma_2\sigma_1, \sigma_2^{-1},\textcolor{purple}{\sigma_1^{-1} \sigma_2^{-1}\sigma_1, \sigma_3}, \sigma_3^{-1})\xrightarrow{s_3^{-1}}
(\sigma_2\sigma_1, \sigma_2^{-1}, \sigma_3,  \sigma_3^{-1}\sigma_1^{-1} \sigma_2^{-1}\sigma_1 \sigma_3, \sigma_3^{-1})$$
Now observe that $\sigma_3^{-1}\sigma_1^{-1} \sigma_2^{-1}\sigma_1 \sigma_3=\sigma_1^{-1}\sigma_3^{-1} \sigma_2^{-1} \sigma_3\sigma_1=\sigma_1^{-1}\sigma_2 \sigma_3^{-1} \sigma_2^{-1}\sigma_1$. Now we apply $s_2$ to this braid system.
$$(\sigma_2\sigma_1, \textcolor{blue}{\sigma_2^{-1}, \sigma_3},\sigma_1^{-1}\sigma_2 \sigma_3^{-1} \sigma_2^{-1}\sigma_1 , \sigma_3^{-1})\xrightarrow{s_2} (\sigma_2\sigma_1,\sigma_2^{-1} \sigma_3 \sigma_2, \sigma_2^{-1}, \sigma_1^{-1}\sigma_2 \sigma_3^{-1} \sigma_2^{-1}\sigma_1 , \sigma_3^{-1}) $$

We obtain a braid system $(\sigma_2\sigma_1,\sigma_2^{-1} \sigma_3 \sigma_2, \sigma_2^{-1}, \sigma_1^{-1}\sigma_2 \sigma_3^{-1} \sigma_2^{-1}\sigma_1 , \sigma_3^{-1}) $ lifting the permutation system $((123),(24),(14)(23),(34))$ with the required properties. 

\end{example}

\begin{example}  Let us consider the permutation system $\rho=((143),(24),(34),(23),(13))$.
We use fission on each permutation to construct the new permutation system
$$\varrho=((13),(34),(24),(34),(23),(13)).$$
We will now be using sliding moves and its inverse to this system, (almost) as described before.
$$((13),\textcolor{blue}{(34),(24)},(34),(23),(13))$$
$$\xrightarrow{s_2}((13),(23),\textcolor{red}{(34),(34)},(23),(13))$$
$$\downarrow((13),\textcolor{red}{(23),(23)},(13))$$
$$\downarrow((13),(13))$$
We see that the braid system $(\sigma_1 \sigma_2 \sigma_1^{-1},\sigma_1 \sigma_2^{-1} \sigma_1^{-1})$
Using the sliding moves in reverse to this braid system and introducing appropriate braid subsystems at places we deleted permutation subsystems, we get a braid system lifting $\varrho$.
$$(\sigma_1 \sigma_2 \sigma_1^{-1},\sigma_1 \sigma_2^{-1} \sigma_1^{-1})$$
$$\uparrow(\sigma_1 \sigma_2 \sigma_1^{-1},\textcolor{magenta}{\sigma_2, \sigma_2^{-1}},\sigma_1 \sigma_2^{-1} \sigma_1^{-1})$$
$$\uparrow(\sigma_1 \sigma_2 \sigma_1^{-1},\sigma_2,\textcolor{magenta}{\sigma_3, \sigma_3^{-1}}, \sigma_2^{-1},\sigma_1 \sigma_2^{-1} \sigma_1^{-1})$$
$$=(\sigma_1 \sigma_2 \sigma_1^{-1},\textcolor{blue}{\sigma_2,\sigma_3}, \sigma_3^{-1}, \sigma_2^{-1},\sigma_1 \sigma_2^{-1} \sigma_1^{-1})$$
$$\xrightarrow{s_2^{-1}}(\sigma_1 \sigma_2 \sigma_1^{-1},\sigma_3,\sigma_3^{-1}\sigma_2\sigma_3, \sigma_3^{-1}, \sigma_2^{-1},\sigma_1 \sigma_2^{-1} \sigma_1^{-1})$$
$$=(\sigma_1 \sigma_2 \sigma_1^{-1},\sigma_3,\sigma_2\sigma_3\sigma_2^{-1}, \sigma_3^{-1}, \sigma_2^{-1},\sigma_1 \sigma_2^{-1} \sigma_1^{-1})$$
By fusion (that is merging) we see that $$(\sigma_1 \sigma_2 \sigma_1^{-1}\sigma_3,\sigma_2\sigma_3\sigma_2^{-1}, \sigma_3^{-1}, \sigma_2^{-1},\sigma_1 \sigma_2^{-1} \sigma_1^{-1})$$ is a braid system lifting $\rho$.

\end{example}

\begin{example}  Let us consider the permutation system $\rho=((143),(15),(25),(45),(25),(35),(45),(25),(15))$.

$$((143),(15),(25),(45),(25),(35),(45),(25),(15))$$
$$\xrightarrow{s_4}((143),(15),(25),(24),(45),(35),(45),(25),(15))$$
$$\xrightarrow{s_5}((143),(15),(25),(24),(34),(45),(45),(25),(15))$$
$$\searrow((143),(15),(25),(24),(34),(25),(15))$$
$$\xrightarrow{s_4^{-1} s_3^{-1}}((143),(15),(24),(34),(35),(25),(15))$$
$$\xrightarrow{s_4^{-1} s_3^{-1}}((143),(15),(24),(34),(35),(25),(15))$$
$$\xrightarrow{s_3^{-1} s_2^{-1}}((143),(24),(34),(15),(35),(25),(15))$$
$$\xrightarrow{s_6 s_5}((143),(24),(34),(15),(23),(13),(35))$$
$$\xrightarrow{s_5^{-1} s_4^{-1}}((143),(24),(34),(23),(13),(35),(35))$$
$$\searrow ((143),(24),(34),(23),(13)) $$

From the previous example we know that $(\sigma_1 \sigma_2 \sigma_1^{-1}\sigma_3,\sigma_2\sigma_3\sigma_2^{-1}, \sigma_3^{-1}, \sigma_2^{-1},\sigma_1 \sigma_2^{-1} \sigma_1^{-1})$ is a braid system lifting $((143),(24),(34),(23),(13)) $

We will get required braid system by following the above mentioned process. Using the sliding moves in reverse to this braid system and introducing appropriate braid subsystems at places we deleted permutation subsystems, we will get a braid system lifting $\rho$.
\end{example}

\subsection{Lifting branched covers in the smooth category}

For any natural number $n\geq 2$, consider the branched cover over the sphere with permutation system
$(\rho,...,\rho)$, where $\rho=(12...n)$ is an $n$-cycle, which is repeated $n$ times.

\begin{claim} For even $n$, the permutation system $(\rho,...,\rho)$ lifts to a smooth braided embedding.
\end{claim}
\begin{proof} Note that if we take $\alpha=\sigma_{n-1}...\sigma_{1}$ and $\beta=\sigma_{n-1}^{-1}...\sigma_{1}^{-1}$, then both $\alpha$ and $\beta$ lift $\rho$, and 
both these braids give rise to valid smooth local models near branch points (see Section \ref{lm}).
Note that $\alpha^{\frac{n}{2}}$ is the Garside element $\Delta$, see \cite{Ga}. Recall, the Garside element has the property that $\Delta \sigma_i^\pm=\sigma_{n-i}^\pm\Delta$ for all $i$. Thus it follows that $$\Delta \beta=\Delta \sigma_{n-1}^{-1}...\sigma_{1}^{-1}=\sigma_{1}^{-1}...\sigma_{n-1}^{-1}\Delta=\alpha^{-1}\Delta;$$
and hence $\Delta \beta^{\frac{n}{2}}=\alpha^{-\frac{n}{2}}\Delta=\Delta^{-1}\Delta=1$. Hence
the braid system $(\alpha,...,\alpha,\beta,...,\beta)$ smoothly lifts the given permutation system, where both $\alpha$ and $\beta$ appears $\frac{n}{2}$ times in the braid system.
\end{proof}

\begin{claim} For odd $n$, the permutation system $(\rho,...,\rho)$ does not lift to a smooth braided embedding.

\end{claim}
\begin{proof} Recall from Section \ref{lm}, that for any braid surrounding a branch point, it has to be conjugate to either $\sigma_{n-1}...\sigma_{1}$ or $\sigma_{n-1}^{-1}...\sigma_{1}^{-1}$, in particular the exponent sum has to be $\pm (n-1)$. However, there is no way of adding up an odd number of $\pm (n-1)$ and getting $0$. Consequently, there cannot be any lift, as required.

\end{proof}

\begin {remark} The above claim shows that there are some differences between braided embeddings in the piecewise linear and smooth categories. While all branched covers over $S^2$ lift in the former, there are infinitely many cyclic branched covers in the smooth category which do not lift.
\end{remark}

\begin {remark}\label{algobs}
We see that we obtain the following algebraic obstruction to lifting a branched cover over $S^2$ smoothly, given the permutation system, we need to be able to assign positive or negative signs to the various disjoint cycles appearing in the permutation system, so that the total sum is zero. We also can refine this by looking at connected components, for instance the permutation system
$$((123)(456),(123)(456),(123)(456))$$ cannot be lifted to a braid system smoothly, although if we assign positive signs to $(123)$ and negative signs to $(456)$ the total sum is 0.
\end{remark}

\subsection{ Branched covers over higher genus orientable surfaces}
We will use the same strategy as that of Theorem~\ref{orcovsurf} to prove a more general result for branched coverings over all orientable surfaces.
\begin{thm}
 Every branched covering of over an orientable surface lifts to a piecewise linear locally flat braided embedding.
\end{thm}

\begin{proof}
 Given any branched covering $p:\Sigma_g\rightarrow \Sigma_h$, we can compose with the two sheeted branched covering $q$ of $\Sigma_h$ over $S^2$ (see Figure~\ref{qhu}), in such a way that the branch locus of $p$ is disjoint from the pre-image of the branch locus of $q$. By Theorem \ref{bcl2s} we know that the composite branched covering $r=q\circ p$ lifts to a braided embedding. By Proposition \ref{comp1}, we see that $p$ lifts to a braided embedding. To check that this embedding is actually a piecewise linear locally flat, we need to check the braid monodromy around branch points. For any based loop $\gamma$ in $(\Sigma_h,*)$, we can push it down by $q$, and look at the braid monodromy of $q\circ \gamma$ corresponding to the braided embedding $R$ lifting $r$. Since $\gamma$ is a loop, the pre-image points of $*$ must go to itself under the permutation monodromy, and consequently the braid monodromy of $\gamma$ with respect to the braided embedding $p$ is the interior braid of the braid monodromy of $q\circ \gamma$ with respect to the braided embedding $r$, corresponding to the fiber $p^{-1}\{*\}$. The result now follows from our local characterization of braids surrounding locally flat branch points in Subsection \ref{SSextpl}.
\end{proof}





\section{Lifting branched coverings over the three-sphere}\label{lS3}

In this section, we will consider the case of branched covers over the three sphere $S^3$. We will see that there are algebraic obstructions (torsion) to lifting branched covers over $S^n$, when $n\geq 4$. There is no such easy obstructions in three dimension, because knot groups are torsion free. 

We would mostly restrict to the case of simple branched covers (actually only simple three and four colorings) in this subsection, and contrast with the case one dimension lower, where we saw every simple branch cover lifted  to a braided embedding, in both piecewise linear and smooth categories. In fact by our analysis of local models earlier, we saw that in any dimension and over any manifold, a simple branched cover lifts in the piecewise linear category if and only if it lifts in the smooth category.

\subsection{Colorings}\label{Subscol}
\begin{defi}
For any group $G$, we define a $G$-coloring to be a homomorphism (or antihomomorphism) from the fundamental group of the link complement to $G$. By the Wirtinger presentation of the link group\footnote{By this we will mean the fundamental group of link complement.}, it is equivalent to color (or label) each strand of any link diagram by elements of $G$ such that at the crossings the Wirtinger relations are satisfied.

\end{defi}

For instance, a Fox-$n$ coloring of a link is a homomorphism from link group to the dihedral group $D_n$ (which canonically is a subgroup of the symmetric group $S_n$), so that meridians go to reflections. In particular, a tricoloring of a link is a homomorphism from link group to $D_3\cong S_3$. We will use the following colors to indicate a simple $S_3$ or $B_3$ colorings, see Figure \ref{braidcol}.
\begin{figure}[H]
    \centering
    \includegraphics[width=8 cm]{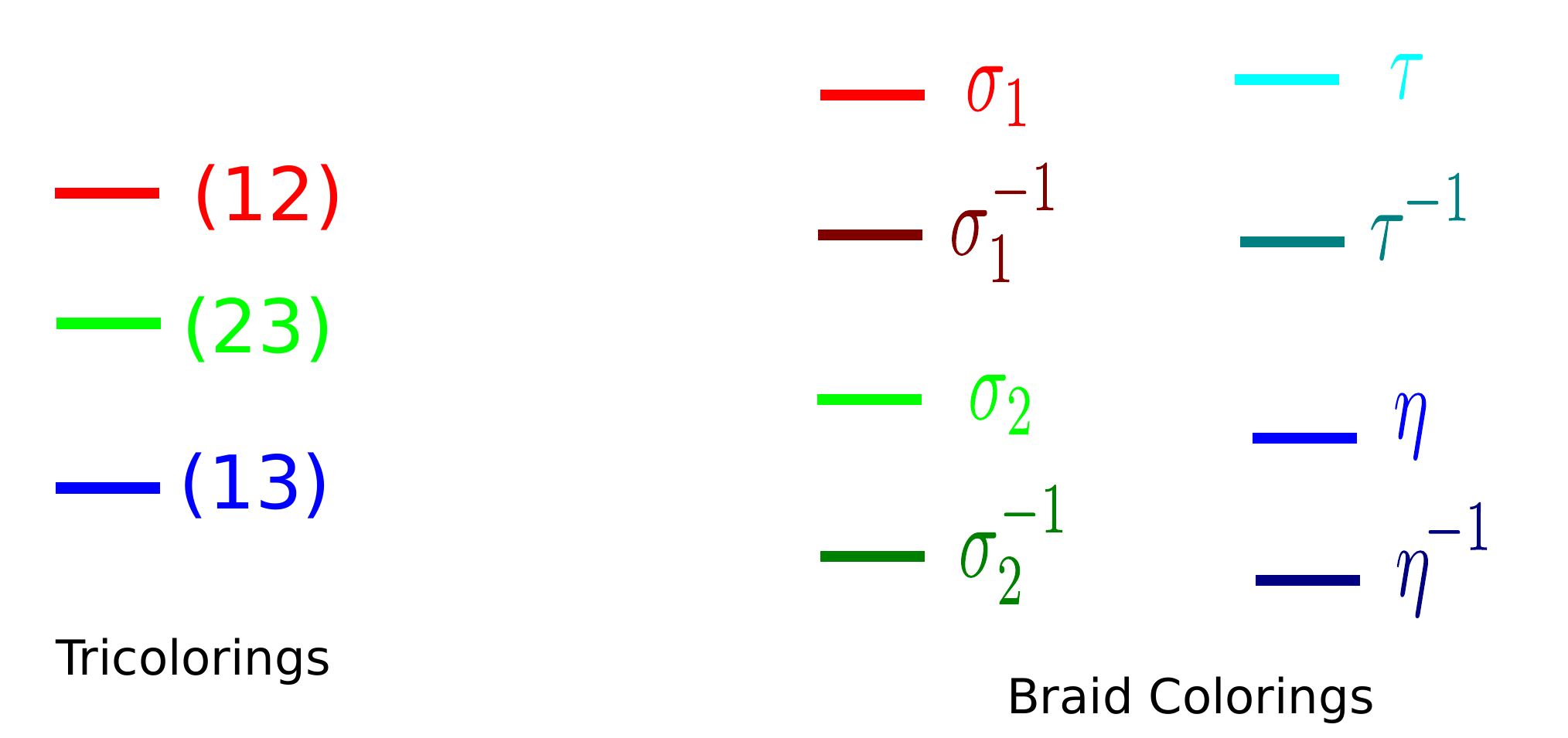}
    \caption{We will use these colors to indicate the colorings on the strands. In the right, $\sigma_1$ and $\sigma_2$ are the standard generators of $B_3$, and $\tau=\sigma_2^{-1}\sigma_1\sigma_2=\sigma_1\sigma_2\sigma_1^{-1}$, and $\eta=\sigma_2\sigma_1\sigma_2^{-1}=\sigma_1^{-1}\sigma_2\sigma_1$.   }
    \label{braidcol}
\end{figure}

\begin{remark}
Our convention is that we read group elements from left to right in fundamental groups (which includes braid groups) and symmetric group (the elements of which we think about as a product of cycles), and from right to left in mapping class groups (this is the standard convention for composing functions). Consequently, the some of the colorings we will consider will be group anti-homomorphisms (depending on the conventions of multiplication in the target group), which is why we included it in the definition of $G$-colorings. For any given coloring, it will be clear from the context if we are talking about a homomorphism or anti-homomorphism. An alternate notational convention would be to take the opposite group when necessary, so that a coloring is always a homomorphism. 
\end{remark}


It turns out that a branched covering is completely determined by its monodromy data of the associated covering space, and following Fox, we will call such homomorphisms colorings.


 


\subsection{Torus Knots}
It is known that \cite{BOT} a torus knot $T_{p,q}$ (by symmetry, let us assume that $p$ is odd) tricolorable if and only if $p$ is a multiple of $3$, and $q$ is even, and moreover in those cases the tricoloring is conjugate to the ``main tricoloring'' as illustrated by Figure \ref{trs} (the color pattern repeats both horizontally and vertically as $p$ and $q$ vary):
\begin{figure}[H]
    \centering
    \includegraphics[width=10 cm]{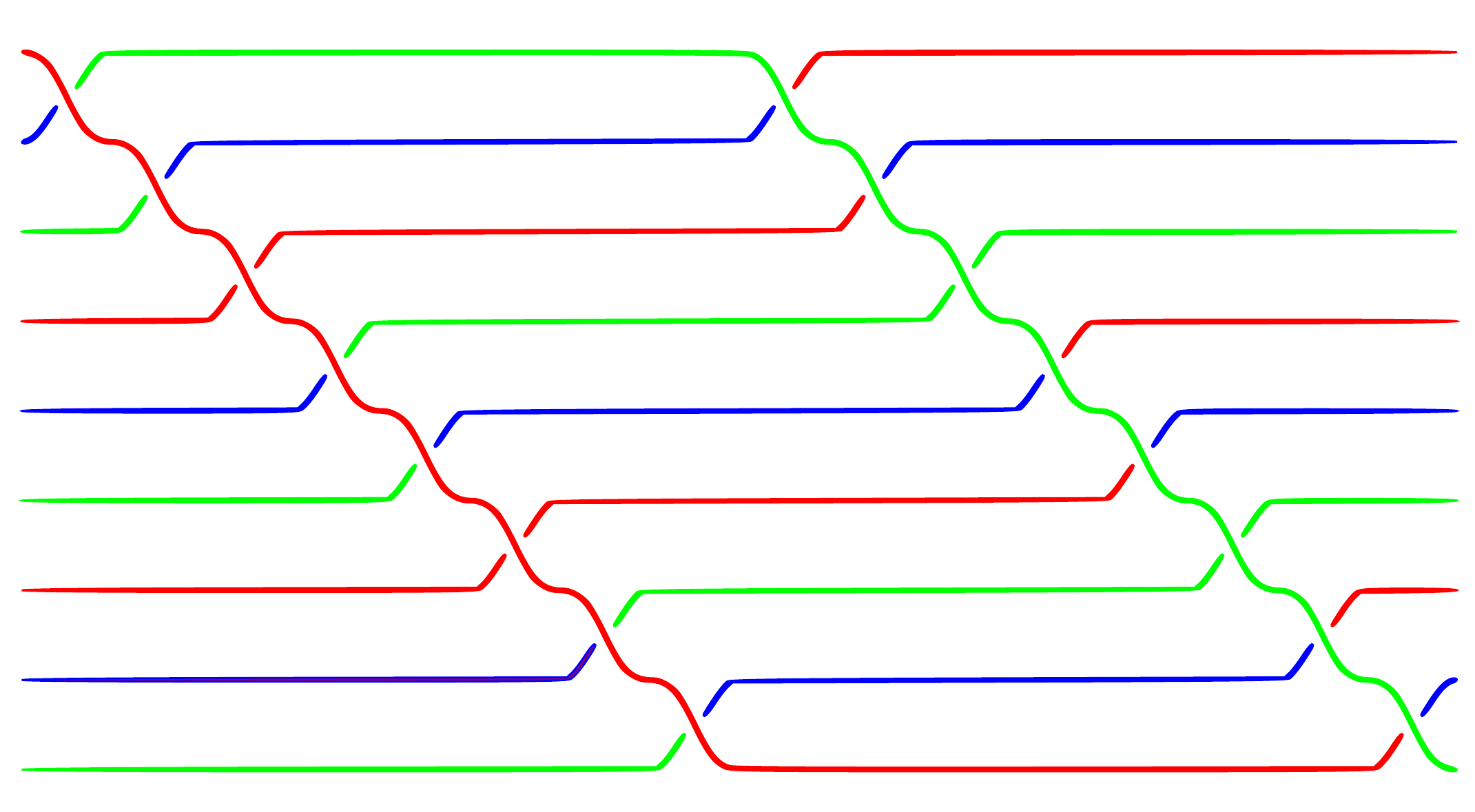}
    \caption{Main Tricoloring on the (9,2) torus knot}
    \label{trs}
\end{figure}

It follows that if we can show that this tricoloring lifts to a simple $B_3$-coloring, then all tricolorings on torus knots lift to simple braid coloring. Indeed, the following braid coloring pattern shows that there indeed is a lift (observe that for the lift of (1,3), we need to use both $\tau$ and $\eta$ to obtain a valid coloring), see Figure \ref{trs2}.

\begin{figure}[H]
    \centering
    \includegraphics[width=10 cm]{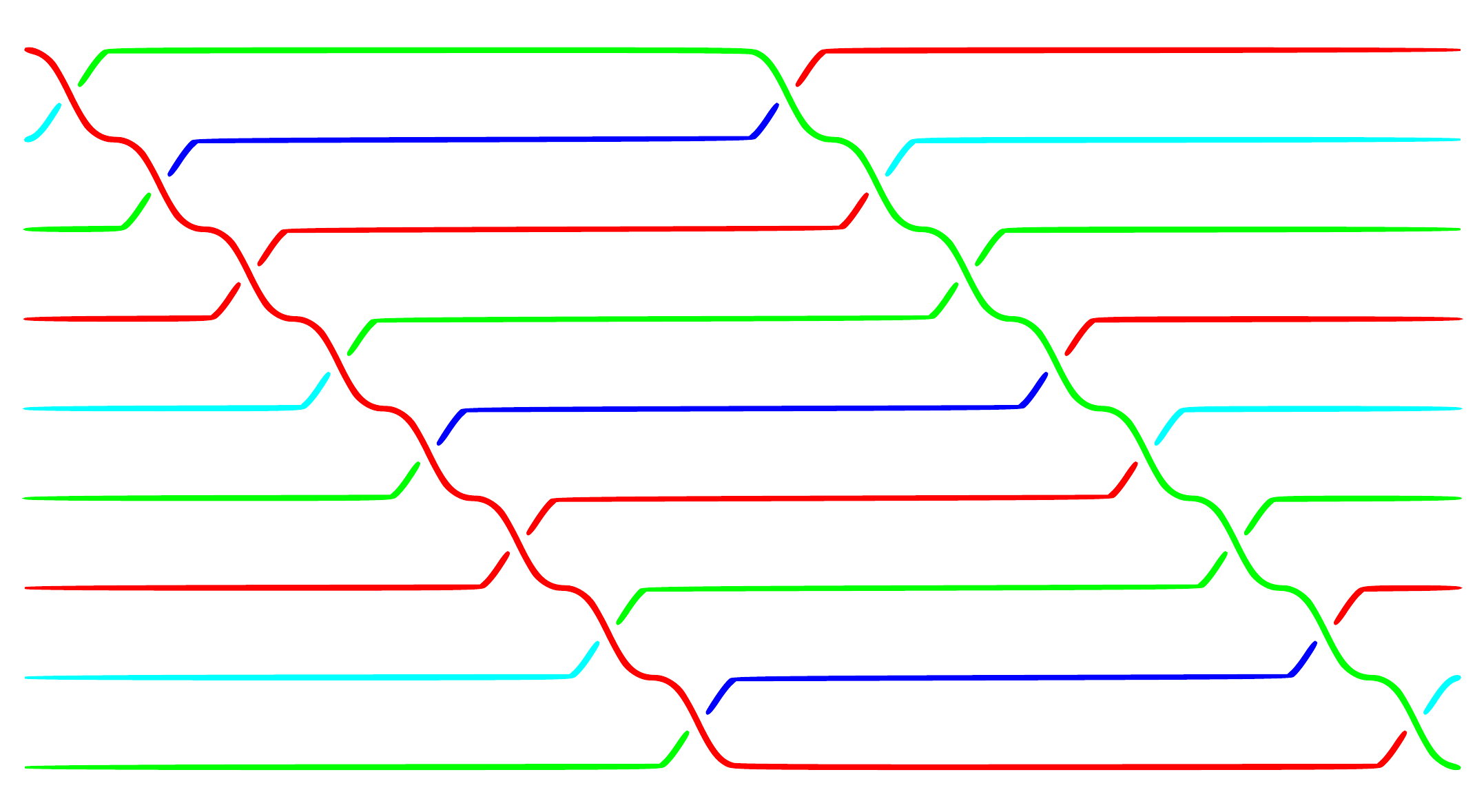}
    \caption{Simple $B_3$-coloring on the (9,2) torus knot}
    \label{trs2}
\end{figure}

Let us summarize the above discussion:
\begin{prop}
 A torus knot $T_{p,q}$ is tricolorable if and only if one of $p$ and $q$ is an odd multiple of 3, and the other is even. In this case, we moreover have that there is only one tricoloring (up to conjugation), and the tricoloring lifts to a simple $B_3$-coloring.
\end{prop}

 It turns out other sorts of colorings are possible for torus links, and there is no known classification (to the best of the author's knowledge) of tricolorings of torus links. For example, consider the following tricoloring on $T_{4,4}$ in Figure \ref{44t}:
 \begin{figure}[H]
    \centering
    \includegraphics[width=14 cm]{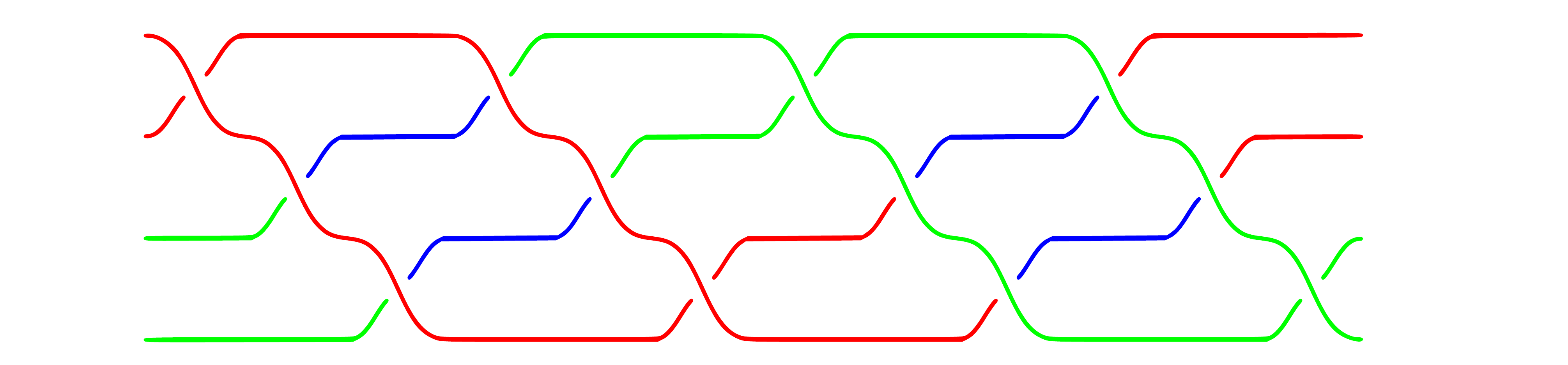}
    \caption{A tricoloring on the (4,4) torus link}
    \label{44t}
\end{figure}

\noindent While this tricoloring lifts to a simple braid coloring, as illustrated by Figure \ref{44b};

\begin{figure}[H]
    \centering
    \includegraphics[width=14 cm]{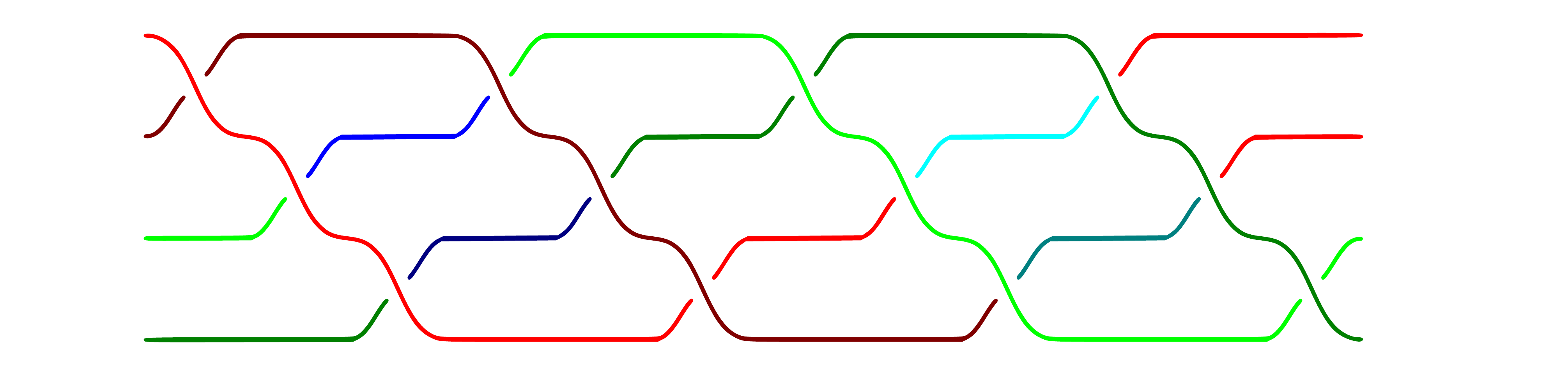}
    \caption{Simple $B_3$-coloring on the (4,4) torus link}
    \label{44b}
\end{figure}

\noindent Etnyre and Furukawa \cite{EF} show 
 that it is possible to modify the branch locus by Montesinos 3- moves to obtain a branched cover over a knot, see Figure \ref{Efeg}:
 \begin{figure}[H]
    \centering
    \includegraphics[width=12 cm]{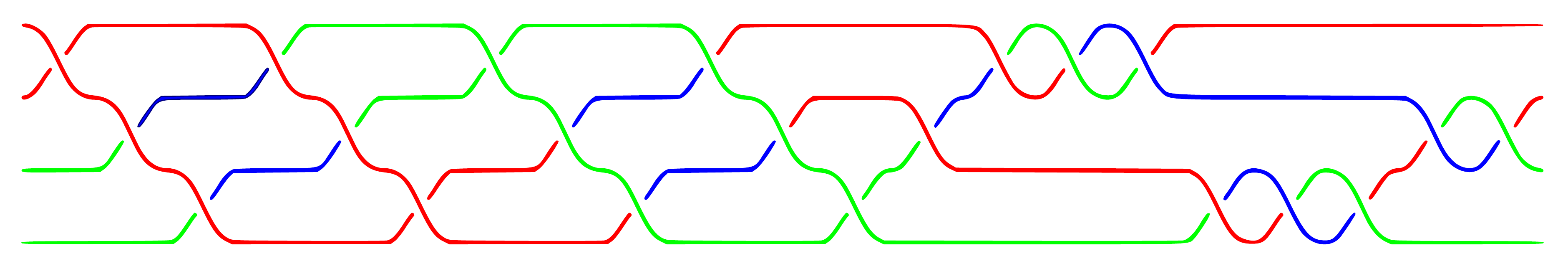}
    \caption{A non-liftable tricoloring.}
    \label{Efeg}
\end{figure}
\noindent which does not lift to a braided embedding. In fact the above example is one of an infinite family of Etnyre and Furukawa, which starts of with a liftable simple $S_n$ coloring on a torus link, but after some Montesinos 3- moves one obtains a $S_n$ colored knot which do not lift. These examples suggest something subtle is going on with changing the branch locus with 3-moves, while it does not change the branched manifold upstairs, one frequently can go from a liftable branched covering to a non-liftable one (and vice versa).

We now discuss lifts of non-simple branched covers of the Hopf link (the simplest torus link after the unlink).
\begin{example}\label{hopf}
The link group of the Hopf link is $\langle x, y| xy=yx \rangle$, where $x$ and $y$ are meridians. It follows that we need two commuting elements to define a coloring. However, the lifting problem for branched coverings of the three sphere, branched over the Hopf link is not the same as the problem of lifting a covering over the torus, since we now have constraints on which braids the meridians can go to, as we want the braided embedding to be smooth (or piecewise linear locally flat). Recall, that the only braids lifting the $n$-cycle $(12...n)$ to a smooth braided embedding, must be conjugate (by a pure braid) to either $\alpha_n=\sigma_{n-1}...\sigma_{2}\sigma_{1}$, or to $\beta_n=\sigma_{n-1}^{-1}...\sigma_{2}^{-1}\sigma_{1}^{-1}$. The center of $\alpha_n$ (respectively $\beta_n$) is known to be generated by $\alpha_n$ (respectively $\beta_n$). However, the only power of $\alpha_n$ (respectively $\beta_n$) which have exponent sum $\pm n$ is $\alpha_n^{\pm 1}$ (respectively $\beta_n^{\pm 1}$). Hence it follows that if we take $a$ to be the $n$-cycle
$(12...n)$, and $b=a^p$ for some $p$ coprime to $n$ so that $b$ is different from $a$ and $a^{-1}$, then the permutation coloring on the Hopf link determined by $a,b$ does not lift to a smooth braided embedding. 
\end{example}

\subsection{Two-bridge knots and links}
Recall that two bridge links (see \cite{G} for details) in $S^3$ are parameterized by a rational number $\frac{p}{q}$, and such a link always has a Wirtinger presentation of the form
$\langle a,b|aw=wb \rangle$ if it is a knot, and $\langle a,b|aw=wa \rangle $ if it is a link. Here $a,b$ are meridians and $w$ is a word in $a,b$. 
If there is a homomorphism from the link group to braid group $B_3$ sending the meridians to half twists, then those half twists have to satisfy the relation $\star$  of the link group (where $\star$ denotes either $aw=wb$ or $aw=wa$). We can take double branched cover of the  disc with three points (recall $B_3$ is the mapping class group of the thrice punctured disc), and those half twists lift to Dehn twist in the mapping class group of $S_1^1$, the once punctured torus. In general, the study of mapping class group of a surface and some cover is called Birman-Hilden theory \cite{BH}, but in this particular cases it is fairly straightforward since the mapping class groups are isomorphic.
Thus the original problem translates to: can we find two Dehn twists which satisfies the relation $\star$. This problem  of relations between two Dehn twists in an orientable surface has been studied by Thurston \cite{T}, who showed that two Dehn twists either satisfy the braid relation (if and only if the geometric intersection number of the corresponding curves is 1), or they commute (if and only if the geometric intersection number of the corresponding curves is 0, however note that this cannot happen for two distinct simple closed curves in once punctured torus); or they do not have any relation. For a non-trivial 2-bridge link, the above relation $\star$ is always non trivial, and thus the two Dehn twists have to satisfy the braid relation.
Consequently, we have:
\begin{thm}\label{2bridge} Given any two-bridge link, and any Wirtinger presentation of link group  $\langle a,b| r \rangle$ .   Suppose the link has  a non-trivial tricoloring (i.e. the relation $r$ holds when we set $a=(12)$ and $b=(23)$),  then the tricoloring lifts to a group homomorphism to $B_3$ if and only if the relation $r$ holds when we set $a=\sigma_1$ and $b=\sigma_2$.
\end{thm}

\begin{remark}\label{2braid} (Closures of two strand braids): The closure of the 2-braid $\sigma_1^n$ is tricolorable if and only if $n$ is a multiple of 3. If $n$ is a multiple of 3, any non-constant tricoloring on closure $\widehat{\sigma_1^n}$, lifts to a unique braid coloring (up to conjugation in the braid group $B_3$), by repeating the following braid coloring, see Figure \ref{2br}.
\begin{figure}[H]
    \centering
    \includegraphics[width=10 cm]{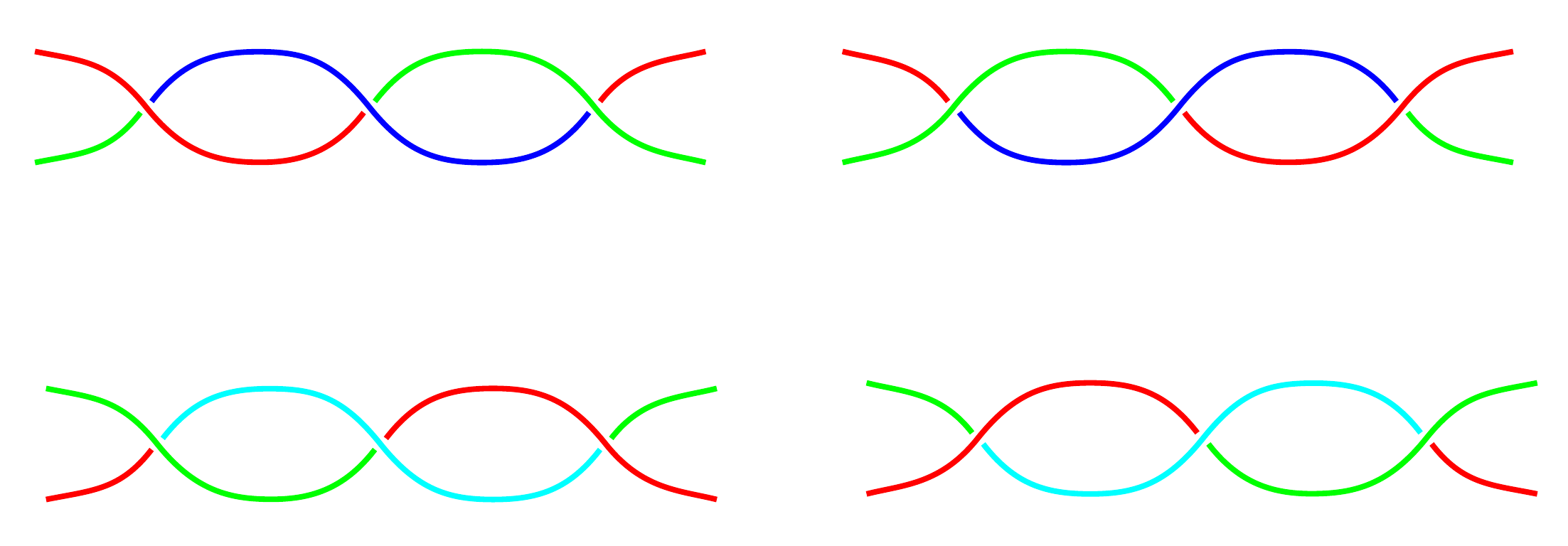}
    \caption{We indicate the various possibilities of colorings with three half twists, if the initial colorings are $\sigma_1$ and $\sigma_2$.}
    \label{2br}
\end{figure}
When $n$ is even, the closure $\widehat{\sigma_1^n}$ is a link, and the above theorem implies that the tricoloring does not lift to a braid coloring if we want the induced orientation on each component going the opposite way.

\end{remark}

Since the word problem in braid groups is solvable \cite{Ar}, the above results give a complete characterization of which tricolorings of two-bridge knots or links lift. It would, however be interesting to find a characterization more directly in terms of the rational number  $\frac{p}{q}$ (maybe something involving continued fraction expansion of  $\frac{p}{q}$, or the up-down graph \cite{G}). 
Recall that a two-bridge knot can have at most one (non-trivial) tricoloring, up to conjugation; and in fact it follows from the above discussion that when a two-bridge knot admits a simple simple $B_3$ coloring, it is unique, up to conjugation.

Among the tricolorable two-bridge knots in Rolfsen's knot table \cite{Ro}, the following admits a simple simple $B_3$ coloring:
$$3_1,9_1, 9_6, 9_{23}, 10_5, 10_9,  10_{32}, 10_{40};$$
and the following do not:
$$ 6_1, 7_4, 7_7, 8_{11}, 9_2, 9_4, 9_{10}, 9_{11}, 9_{15}, 9_{17}, 10_4, 10_{10}, 10_{19}, 10_{21}, 10_{29}, 10_{31} , 10_{36},10_{42}. $$


\subsubsection{Homomorphism of link groups}
In this subsection, let us discuss some generalities about having homomorphisms from a link group $\pi_1(S^3\setminus L)$ to some group $G$. 

Let $H$ be a subgroup of $G$ contained in the center $\mathcal{Z}(G)$ of $G$. So given any homomorphism $\phi:\pi_1(S^3\setminus L)\rightarrow G$, we get a group homomorphism from $\psi:\pi_1(S^3\setminus L)\rightarrow G/H$ by composing with the natural projection $G\rightarrow G/H$. We will show that the converse is also true. 

We know that $\pi_1(S^3\setminus L)$ has a Wirtinger presentation $\langle x_1,...,x_k|r_1,...,r_k \rangle$, so given any group homomorphism $\psi:\pi_1(S^3\setminus L)\rightarrow G/H$, let us pick any element $g_1$ in $G$ which projects to $\psi(x_1)$. The Wirtinger relation $x_j=x_ix_1x_i^{-1}$
will determine where $x_j$ has to go, as follows. Pick any $g_i$ lifting $\psi(x_i)$, and we are forced to send $x_j$ to $g_ig_1g_i^{-1}$. The reader should note that if we picked another lift $g'_i$ then $g'_i=g_iz$ for some central element $z$, and consequently $g_ig_1g_i^{-1}=g_i'g_1{g'}_i^{-1}$. Hence, if we choose meridians, one for each of the component of the link, and lift for the images under $\psi$ each of those meridians, the Wirtinger relations give us a lift $\phi$ of $\psi$, as required. 

Let us focus on the case of $G=B_3$ and $H=Z(G)$, which is known to be generated by the square of the Garside element $\Delta^2=(\sigma_1\sigma_2\sigma_1)^2$ . It is well known that the quotient $G/H$ is the modular group $PSL(2,\mathbb{Z})\cong \mathbb{Z}_2*\mathbb{Z}_3$.
The above result means given any homomorphism  $\psi:\pi_1(S^3\setminus L)\rightarrow PSL(2,\mathbb{Z})$ then it lifts to a homomorphism  $\phi:\pi_1(S^3\setminus L)\rightarrow B_3$ (and of course we get similar statement if we replace $PSL(2,\mathbb{Z})$  with $SL(2,\mathbb{Z})$). Consequently, it follows that lifting an $S_3$-coloring of a link to a $B_3$-coloring is equivalent to whether it lifts to an $PSL(2,\mathbb{Z})$-coloring  (or $SL(2,\mathbb{Z})$-coloring). The reader should note that the natural projection $B_3\rightarrow S_3$ factors through $SL(2,\mathbb{Z})$ and $PSL(2,\mathbb{Z})$, and so it makes sense to talk about the lifting problem in that context. Also, we see that admissible $B_3$-colorings correspond to admissible $PSL(2,\mathbb{Z})$-coloring  (or admissible $SL(2,\mathbb{Z})$-coloring).

Now if we were considering an epimorphism (= surjective homomorphism) $\psi:\pi_1(S^3\setminus L)\rightarrow PSL(2,\mathbb{Z})$ sending meridians to standard generators (standard transvections), then we see that it has to lift to an epimorphism $\phi:\pi_1(S^3\setminus L)\rightarrow B_3$, as follows. Observe that, up to conjugation, we can choose to send via $\phi$ a meridian $\mu\in\pi_1(S^3\setminus L)$ to $\sigma_1$ (which is the lift of a standard transvection), and since we know $\psi$ is surjective we know there is an element $\alpha\in \pi_1(S^3\setminus L)$ so that $\psi(\alpha)=p(\sigma_2)$. Then for the lift $\phi$ we have $\phi(\alpha \mu\alpha^{-1})=\sigma_2\sigma_1 \sigma_2^{-1}$.  Since by the braid relation $$\sigma_2=\sigma_1\sigma_2\sigma_1 \sigma_2^{-1}\sigma_1^{-1}=\phi(\mu\alpha \mu\alpha^{-1}\mu^{-1}),$$ the image of $\phi$ contains both $\sigma_1$ and $\sigma_2$,
an so $\phi$ is surjective.

If we have a homomorphism $\psi:\pi_1(S^3\setminus L)\rightarrow PSL(2,\mathbb{Z})$, then the image of $\psi$  is a subgroup of $PSL(2,\mathbb{Z})$; and by Kurosh subgroup theorem, will be abstractly isomorphic to a free product of $\mathbb{Z}$, $\mathbb{Z}_2$ and $\mathbb{Z}_3$'s.\\

\subsection{Tricolorings of knots}
If we are considering the image of a knot group, then we know that the abelianization has to be cyclic, so the only possible images in  $PSL(2,\mathbb{Z})$ are isomorphic to $\{1\}$, $\mathbb{Z}$, $\mathbb{Z}_2$, $\mathbb{Z}_3$ and $\mathbb{Z}_2*\mathbb{Z}_3$. If we are considering such a $\psi$ coming from a non-trivial simple $B_3$-coloring, then the only possibility is $\mathbb{Z}_2*\mathbb{Z}_3$ because:
\begin{itemize}
 \item The subgroup generated by the coset of half twist is isomorphic to the integers (rules out $\{1\}$, $\mathbb{Z}_2$, $\mathbb{Z}_3$)
 \item Two distinct half twists in $B_3$ cannot commute (as one of the endpoints has to be the same); and this continues to be true in the quotient, which rules out $\mathbb{Z}$. It is known that two half twists in $B_3$ either satisfy no relations, or satisfy the braid relation. 
\end{itemize}
 
 Suppose we have any epimorphism\footnote{we do not assume here it sends meridians to standard transvections, unlike the previous subsection} $\psi:\pi_1(S^3\setminus K)\rightarrow PSL(2,\mathbb{Z})$, and we will show that it has to lift to an epimorphism $\phi:\pi_1(S^3\setminus K)\rightarrow B_3$, as follows. For any meridian $\mu\in\pi_1(S^3\setminus K)$, if $\psi(\mu)$ has exponent sum in $\mathbb{Z}_6$ (exponent sum from $B_3$ is well defined to the integers, and since we get $PSL(2,\mathbb{Z})$ by quotienting by an element of exponent sum 6, we see that exponent sum descends to a well defined homomorphism $\varepsilon:PSL(2,\mathbb{Z})\rightarrow\mathbb{Z}_6 $, and by abuse of terminology we will still call it exponent sum) not equal to $\pm 1$, then the exponent sum is in $\{2,3,4\}$, which would imply the exponent sum of the entire image of $\psi$ is a proper subset of $\mathbb{Z}_6$, which contradicts the fact that $\psi$ is surjective. Changing orientation of $K$ if necessary, let us assume that the exponent sum of $\psi(\mu)$ is $1\in \mathbb{Z}_6$.

Let us choose the lift $\tau\in B_3$ of $\psi(\mu)$,  with exponent sum of $\tau$ being 1 (the various choices for $\tau$ differ up to a central element, $(\sigma_1\sigma_2\sigma_1)^{2n}$ where $n$ is some integer). As we saw earlier, there is a unique lift $\phi:\pi_1(S^3\setminus K)\rightarrow B_3$ of $\psi$ sending $\mu$ to $\tau$.  Moreover we see by using the surjectivity of $\psi$, that for $i=1,2$ there are elements  $\alpha_i\in\pi_1(S^3\setminus K)$ so that $\psi(\alpha_i)=p(\sigma_i)$ and $\alpha_1$ is conjugate to $\alpha_2$. We must have $\phi(\alpha_i)=\sigma_i(\Delta^2)^n$ where $n$ is some integer.

If we write $\tau$ as a word in $\sigma_1$ and $\sigma_2$; and we write out the same word in $\alpha_1$ and $\alpha_2$, we see that element will map under $\phi$ to $\tau (\Delta^2)^n$. Thus $(\Delta^2)^n$ is in the image $\Im\phi$, and hence so are $\sigma_1$ and $\sigma_2$ and consequently $\phi$ is surjective. Hence, we have
\begin{thm}\label{thmepi}
If $K$ is a knot which has a simple $B_3$ coloring, then there is an epimorphism from the knot group $\pi_1(S^3\setminus K)$ to $B_3$. 
\end{thm}

The advantage of promoting the existence of a homomorphism to the existence an epimorphism is that there are known obstructions to such epimorphism. For example, Fox \cite{CF} showed that if there is any epimorphism between knot groups (or groups like knot groups, where Alexander polynomials are defined) then the Alexander polynomial of the target space has to divide the Alexander polynomial of the domain space. This obstruction was upgraded to obstruction coming from twisted Alexander polynomials \cite{KSW}, and these tools have been used to study partial orders on the set of knots \cite{KS}.
As a consequence,  there are lots of tricolorable knots, for which the tricoloring does not lift to a simple $B_3$-coloring.

\begin{thm}\label{alpolydiv}
Suppose $K$ is any tricolorable knot so that $1-t+t^2$ does not divide the Alexander polynomial of $K$ (or an analogous statement with the twisted Alexander polynomials), then no tricoloring of $K$ lifts to a simple $B_3$-coloring.
\end{thm}

This theorem let's us answer for each knot in Rolfsen's knot table, if a knot admits a simple $B_3$-coloring. In Rolfsen's table the bridge index of each knot is at most three, and we already know the answer for two-bridge knots, so it remains to answer it for the three-bridge knots in the table.
The three-bridge tricolorable in Rolfsen's knot table which have a simple braid coloring (see \cite{KS} for explicit homomorphisms) are:
$$8_{5},8_{10},8_{15},8_{18},8_{19}, 8_{20},8_{21}, 9_{16}, 9_{24}, 9_{28}, 9_{40}, 10_{61},10_{62},10_{63},10_{64},10_{65},10_{66}, 
10_{76},10_{77},10_{78},10_{82},$$ 
$$10_{84},10_{85},10_{87},
10_{98},10_{99},10_{103},10_{106},10_{112},10_{114},10_{139},10_{140},10_{141},10_{142},10_{143},10_{144},10_{159},10_{164};$$
and those knots which does not have a simple braid coloring are:
$$9_{29},9_{34}, 9_{35},9_{37},9_{38},9_{46},9_{47},9_{48}, 10_{59},10_{67},10_{68},10_{69},10_{74},10_{75},10_{89},10_{96},10_{97},10_{107},10_{108},10_{113},$$
$$10_{120},10_{122},10_{136},10_{145},10_{146},10_{147},10_{158},10_{160},10_{163},10_{165}. $$

The above statement is about existence of a simple $B_3$-coloring on a knot, not about whether a given tricoloring lifts. More precisely, we can ask:
\begin{question}\label {q1}
Is there a prime knot $K$ with two tricolorings, one of which lifts to a simple $B_3$-coloring, and the other does not?
\end{question}

If we did not include the hypothesis that $K$ is prime, then the answer is easily seen to be "Yes". For we can take a connect sum of a knot $K_1$ which has a non-trivial simple $B_3$-coloring (for example the trefoil, $3_1$), and a tricolorable knot which does not have a non-trivial simple $B_3$-coloring (for example the $6_1$ knot) and consider two tricolorings of $K_1\#K_2$, one which is non-trivial on $K_1$ and trivial on $K_2$, and another which is non-trivial on both $K_1$ and $K_2$.

Let us digress for a moment and discuss of colorings for a connect sum of knots, and we see that the lifting problem for a connect sum reduces to a lifting problem for each of the components.
\begin{prop}\label{conn}
 Suppose $K$ is a connect sum of two knots $K_1, K_2$ in $S^3$. Then for any $G$-coloring on $K$, the colors on the two strands on the boundary of the band corresponding to the connect sum are the same (where we choose consistently oriented meridians for each of the strands). Consequently, by cutting the two strands and joining them to each component, we get $G$- colorings on $K_1$ and $K_2$ respectively.
\end{prop}

\begin{proof} Consider a splitting sphere $S$ or for the connect sum which intersects $K_1\# K_2$ in exactly two points. Let us choose the basepoint for the knot complement on $S$, and by choosing an orientation on $K$, we get oriented meridians for the two strands of $K_1\# K_2$ intersecting $S$. We can homotope the meridians to lie on the sphere $S$ punctured at two points (we remove from $S$ the two points of intersection). We know see the meridians are the homotopic in the twice punctured sphere, and hence in the knot complement. The result follows since any homomorphism must preserve equalities in the domain.
\end{proof}

\begin{remark} We note that the analogue of the above result also holds for  analogues of connect sum of links (by which we mean we can pick any two components of either link and perform a band attachment). While the resulting link will depend (in general) on the choice of the components and the band, any $G$-coloring on it will canonically give rise to $G$-colorings on each of the original links. We also remark that similar result holds for codimension two links in higher dimensional spheres, with exactly the same proof.
\end{remark}

It follows from the above proposition that a tricoloring on a connected sum of knots lift if and only if the corresponding tricolorings on each of the individual knots lift.

Returning to our example, we can conclude the first tricoloring on $K_1\#K_2$ lifts to a braid coloring and the second does not.


Note that for two-bridge knots there is only one tricoloring up to conjugation, and so in this case whether a given tricoloring lifts is the same question as whether the underlying knot has a simple $B_3$ coloring. The answer to the above question is still "Yes", and this means the question about lifting tricolorings is something pertaining to a branched covering map (corresponding to permutation coloring), not just about what the underlying branch set is (or the branched covering manifold upstairs).

To answer Question~\ref{q1}, we need to gain a better understanding of types of relations about more than two Dehn twists. From \cite{K2}, we know that subgroups in $B_3$ generated by three half twists is either free (of rank at most 3), or the entire braid group. This implies any simple braid coloring of a three bridge knot must actually be surjective.  However this does not immediately answer the question about whether a tricoloring on a 3-bridge knot lifts to a braid coloring.

However, by building upon ideas from \cite{K2}, we will answer the question for a family of 3-strand pretzel knots. Let us first make the observation that the epimorphism $\pi:B_3\rightarrow S_3$ factors through $SL_2(\mathbb{Z})$, the mapping class group of the torus. The mapping class group of the once holed torus is isomorphic to $B_3 \cong\langle a,b|aba=bab \rangle$, and from there we can get the mapping class group of the torus by capping off the boundary component, which corresponds to adding the relation $(ab)^6=1$. Recall that we get the presentation of the symmetric group $S_3$ by adding the relations $a^2=1$ and $b^2=1$. Note the relation $(ab)^6=1$ holds in the symmetric group since:
$$(ab)^6=(ababab)^2=(abaaba)^2=(ab^2a)^2=(a^2)^2=1.$$

Thus, we see that $\pi$ factors through\footnote{In fact $\pi$ factors also through $PSL_2(\mathbb{Z})$ as we actually showed $(ab)^3=1$ in $S_3$.} $SL_2(\mathbb{Z})$, i.e. there are homomorphisms $\tau:B_3\rightarrow SL_2(\mathbb{Z})$ and $\rho:SL_2(\mathbb{Z})\rightarrow S_3$ so that $\pi=\rho\circ \tau$.  We observe that $\pi_1$ sends half twists in the braid group to Dehn twists in the mapping class group of the torus (this is exactly the Dehn twist about the curve in the double branched cover which is a  lift of the arc corresponding to the half twist, except that we have capped off the boundary component). Recall simple closed curves (and Dehn twists about them) in the torus are in one to one correspondence of projectivised\footnote{up to a sign (for orientation)} primitive vectors in the first homology $H_1(\mathbb{T}^2)\cong\mathbb{Z}^2$. So given a primitive vector $\vec x=\spalignmat{p;q}$, let us first discuss how to find the image of a Dehn twist under $\rho$.

\begin{claim}\label{colorofDT}
The image under $\rho$ of any Dehn twist $T_x$ only depends on the congruence class of $\vec x$ (with either orientation) modulo 2. More specifically,
\begin{itemize}
    \item if $\vec x\equiv \spalignmat{1;0}\pmod {2}$, then $\rho(T_x)=(12)$,
    \item if $\vec x\equiv \spalignmat{0;1}\pmod {2}$, then $\rho(T_x)=(23)$,
    \item if $\vec x\equiv \spalignmat{1;1}\pmod {2}$, then $\rho(T_x)=(13)$.
\end{itemize}

\end{claim}

\begin{proof} By definition (or by convention) we have that 
$\rho(T_x)=(12), \rho(T_y)=(23), \rho(T_z)=(13) $, where 
$\vec x=\spalignmat{1;0}, \vec y= \spalignmat{0;1}, \vec z=\spalignmat{1;1}$ respectively. By the proof of \cite[Remark 7.4]{K2}, we see that
given any primitive vector $\vec u$, there is another primitive vector $\vec v\in \{\pm \vec x, \pm \vec y,\pm \vec z \}$, and a word $w$ in $T_x^2$, $T_y^2$ and $T_z^2$ so that $w\vec v=\vec u$. It follows that $T_u=wT_vw^{-1}$. Since $\rho (T_x^2)=\rho (T_y^2)=\rho (T_z^2)=1$, it follows that $\rho(T_u)=\rho(T_v)$, and hence the claim holds.

\end{proof}

Let us record the following consequence of the above claim for future use.
\begin{claim} \label{odd}
If $\vec u$ and $\vec v$ are primitive vectors in $\mathbb{Z}^2$, then $\rho(T_u)=\rho(T_v)$ if and only if the algebraic intersection number $\langle \vec u, \vec v \rangle$ is even.
\end{claim}

\begin{proof}
Recall the algebraic intersection number $\langle \vec u, \vec v \rangle$ is equal to the determinant of the matrix with columns $\vec u$ and $\vec v$ (in that order). To find the parity of the determinant (i.e. reduce the determinant modulo 2), we may equivalently reduce $\vec u$ and $\vec v$ modulo 2, and then find the determinant. For a primitive vector modulo 2, there are exactly three choices, and the result now follows from the previous claim.
\end{proof}





\subsection {Labellings in a twist region}

Suppose we have a twist region with $m$ half twists ($m$ can be positive or negative) both strands oriented bottom to top\footnote{Even if the natural orientation of the strands coming from the knot or link goes the other way, we can take the inverse of the meridian and carry out the computations assuming the strands are oriented from bottom to top in the twist region}.
 Let the bottom left and bottom right meridians in the fundamental group be $a$ and $b$, and we denote by $A$ and $B$ the inverses of $a$ and $b$ respectively.

If $m=2n$ is even, by Wirtinger relations the top left and top right meridians are $a^{(ba)^n}$ and $b^{(ba)^n}$ (which  equals $b^{(ab)^{n-1}a}$) respectively. If $m=2n+1$  is odd, by Wirtinger relations the top left and top right meridians are $b^{a(ba)^n}$ and $a^{a(ba)^n}$ (which  equals $a^{(ba)^n}$) respectively.

Figure \ref{wtwist} illustrates these formulas with twist regions with four positive and negative crossings respectively.

\begin{figure}[H]
    \centering
    \includegraphics[width=9 cm]{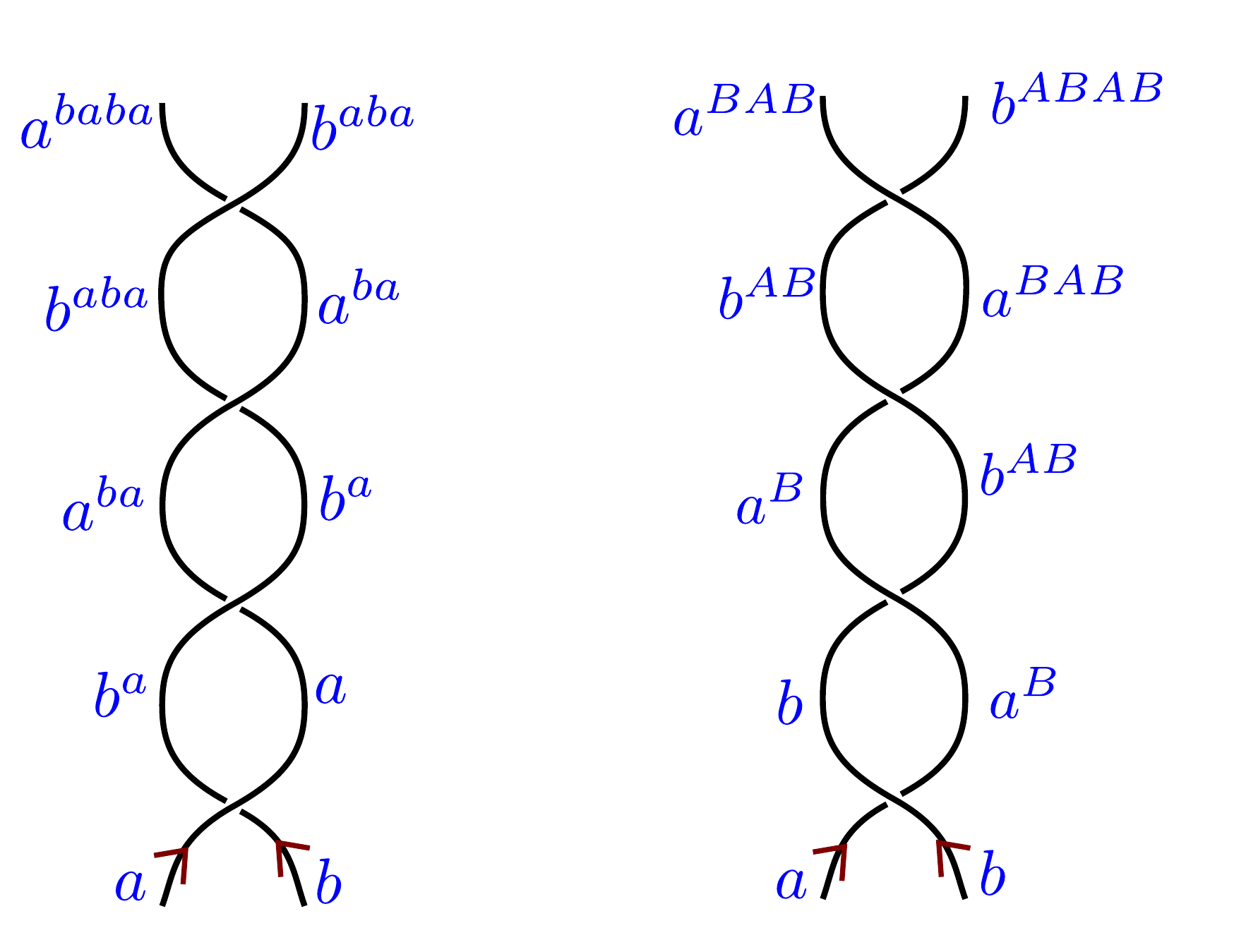}
    \caption{Meridians in twist regions using the Wirtinger presentation, if we start with meridians $a$, $b$ in the bottom; and $A$ and $B$ denote their inverses.}
    \label{wtwist}
\end{figure}

Consider a braid coloring of the twist region with the meridians mapping to half twists, and let us further send it to the corresponding Dehn twist in the double branched cover,  and call this coloring $\phi$.

Note that for Dehn twists we have $T_{f(x)}=f\circ T_x \circ f^{-1}$, so 
if $\phi(a)=T_x$ and $\phi(b)=T_y$, then by Wirtinger relation $c=a^{-1}ba$ (respectively $c=aba^{-1}$) we have $\phi(c)=T_x \circ T_y$ (respectively $\phi(c)=T_x^{-1} \circ T_y$).

We see that $$\phi(a^{(ba)^n})=T_{(T_xT_y)^n(x)} \quad \phi(b^{(ba)^n})=T_{(T_xT_y)^n(y)}$$
$$\phi(b^{a(ba)^n})=T_{(T_xT_y)^n(T_x(y))} \quad \phi(a^{(ba)^n})=T_{(T_xT_y)^n(x)}$$

Instead of labeling the strands with a left or right handed Dehn twist, we could just label by the vector in homology corresponding to the simple closed curve (there is a sign ambiguity when we make initial choices for the bottom left and bottom right) about which the Dehn twists are taking place.
We make the following observation:
\begin{claim}
A simple $B_3$ coloring on a link in $S^3$ is equivalent to having a labelling on the strands by projectivised primitive vectors in $\mathbb{Z}^2$ , such that at each crossing the associated Dehn twists satisfy the Wirtinger relations.
\end{claim}

Let us now rewrite the labellings on the strands in a twist region with $m$ half twists with this convention. Suppose the bottom left strand is labelled with $\vec x$ and the bottom right strand is labelled with $\vec y$. If $m=2n$ is even, the labelling on the top left and top right strands are $(T_xT_y)^n(\vec x)$ and $(T_xT_y)^n(\vec y)$ respectively. If $m=2n+1$  is odd, the labelling on the top left and top right strands are $(T_xT_y)^n(T_x(\vec y))$ and $(T_xT_y)^n(\vec x)$ respectively. Let us try to understand how these expressions look as a linear combination of $\vec x$ and $\vec y$.


Let $k=\langle \vec x,\vec y \rangle$, the by definition we have
$T_y(\vec x)=\vec x-k\vec y$, and hence we obtain
$$ T_x\circ T_y(\vec x)=T_x(T_y(\vec x))=\vec x- k\vec y+\langle \vec x, \vec x-k\vec  y\rangle\vec x= \vec x- k\vec y-k \langle \vec x, \vec y \rangle \vec x=(1-k^2)\vec x-k\vec y$$

More generally, we inductively have for any natural number $n$: $$(T_x\circ T_y)^n(\vec x)=[1-k^2 A_n(k)]\vec x-kB_n(k)\vec y,$$ as justified below for some integer valued functions $A_n,B_n$ of $k$.

We note that $$T_y\circ(T_x\circ T_y)^{n}(\vec x)= T_y((T_x\circ T_y)^n(\vec x))=T_x\circ T_y([1-k^2 A_n(k)]\vec x-kB_n(k)\vec y)$$
$$= [1-k^2 A_n(k)][\vec x-k \vec y]-kB_n(k)\vec y=[1-k^2 A_n(k)]\vec x-k[B_n(k)+1-k^2 A_n(k)]\vec y$$

Hence we have $$(T_x\circ T_y)^{n+1}(\vec x)=T_x([1-k^2 A_n(k)]\vec x-k[B_n(k)+1-k^2 A_n(k)]\vec y)$$
$$=[1-k^2 A_n(k)]\vec x-k[B_n(k)+1-k^2 A_n(k)][\vec y+k\vec x]$$
$$=[1-k^2 [(1-k^2) A_n(k)+B_n(k)+1]]\vec x-k[B_n(k)+1-k^2 A_n(k)]\vec y$$

Hence we have the recursive formulas \begin{align}\label{recA}
  A_{n+1}(k)=(1-k^2) A_n(k)+B_n(k)+1 \text {, and }  
\end{align}
\begin{align} \label{recB}
B_{n+1}(k)=B_n(k)+1-k^2 A_n(k),
\end{align}
 with the initial conditions $A_1(k)=1=B_1(k)$.

Since in the recursive formulas for $A_n(k)$ and $B_n(k)$  we only see quadratic terms in $k$, it follows that $A_n(-k)=A_n(k)$ and $B_n(-k)=B_n(k)$.
By interchanging the roles of $x$ and $y$ (and remembering that $\langle \vec y,\vec x \rangle=-\langle \vec x,\vec y \rangle=-k$) we see that $$(T_y\circ T_x)^n(\vec y)=kB_n(k)\vec x+(1-k^2 A_n(k))\vec y$$

Along the way we also have found formulas for the other labels appearing at the end of the twist regions:

$$T_y\circ(T_x\circ T_y)^{n}(\vec x)=[1-k^2 A_n(k)]\vec x-kB_{n+1}(k)\vec y$$
and (once again by interchanging $x$ and $y$)
$$T_x\circ(T_y\circ T_x)^{n}(\vec y)= [1-k^2 A_n(k)]\vec y +kB_{n+1}(k)\vec x$$

In a similar fashion we may calculate 
$$(T_x^{-1}\circ T_y^{-1})^n(\vec x)=[1-k^2 A_n(k)]\vec x-kB_n(k)\vec y,$$
$$T_y^{-1}(T_x^{-1}\circ T_y^{-1})^n(\vec x)=[1-k^2 A_n(k)]\vec x-kB_{n+1}(k)\vec y,$$


One similarly has functions $C_n,D_n$ of $k$ so that we have the following formulas for composition of left and right handed Dehn twists:
$$(T_x\circ T_y^{-1})^n(\vec x)=[1+k^2 C_n(k)]\vec x+kD_n(k)\vec y,$$
$$T_y^{-1}(T_x\circ T_y^{-1})^n(\vec x)=[1+k^2 C_n(k)]\vec x+kD_{n+1}(k)\vec y,$$
$$(T_x^{-1}\circ T_y)^n(\vec x)=[1+k^2 C_n(k)]\vec x-kD_n(k)\vec y,$$
$$T_y(T_x^{-1}\circ T_y)^n(\vec x)=[1+k^2 C_n(k)]\vec x-kD_{n+1}(k)\vec y.$$

We have similar recursive formulas for $C_n$ and $D_n$, just like the ones for $A_n$ and $B_n$
\begin{align}\label{recC}
  C_{n+1}(k)=(1+k^2) C_n(k)+D_n(k)+1 \text {, and }   
\end{align}
\begin{align} \label{recD}
D_{n+1}(k)=D_n(k)+1+k^2 C_n(k),
\end{align}
 with the initial conditions $C_1(k)=1=D_1(k)$.

In fact, we can directly relate the functions $C_n$ and $D_n$ to the functions $A_n$ and $B_n$ as follows:
$$C_n(k)=A_n(ik) \text{, and } D_n(k)=B_n(ik),$$
where $i$ is a square root of $-1$. An alternate (and probably less mysterious) way of making the same statement is as follows. Recall that $A_n,B_n,C_n$ and $D_n$ are polynomial functions of $k$, where only even degree terms appear. Consequently, we can think of them as polynomials in $k^2$. We get the polynomial $C_n$ (respectively $D_n$) by substituting all occurrences of $k^2$ in $A_n$ (respectively $B_n$) by $-k^2$, and vice versa.

We illustrate using recurrence formulas by finding the labellings in a twist regions with three positive half twists, with the bottom strands
colored by the Dehn twists $T_x^{\pm 1}$ and $T_y^{\pm 1}$, and we pick some orientation of $x$ and $y$, and let $k$ denote the algebraic intersection number between $x$ and $y$. Depending on the sign of the exponents, we have four cases to deal with, in Figure \ref{lab1} (see Remark \ref{same handed}) we consider the two cases that the two Dehn twists have the same handedness (i.e. both right handed or both left handed); while in Figure \ref{lab2} (see Remark \ref{opp handed})
 we consider the two cases that the two Dehn twists have different handedness.
 
 We can also do a similar calculation and see what the labellings would turn out to be if we had a negative twist region. With the same choices as above, if the labellings after $n$ positive half twists are $a\vec x+b\vec y$ and $c\vec x+d\vec y$ (read from left to right), then the labellings after $n$ positive half twists are $(-1)^n(d\vec x-b\vec y)$ and $(-1)^n(-c\vec x+a\vec y)$. One way to think about this is think of the input colors as a basis, and the output colors as the image of a linear map. The linear maps corresponding to positive and negative twist regions (with the same number of half twists) should be inverses of each other, and everything boils down to the formula of the inverse of a $2\times 2$ matrix: $$\spalignmat{a b; c d}^{-1}=\frac{1}{ad-bc}\spalignmat{d -b; -c a}.$$
\begin{figure}[!ht]
    \centering
    \includegraphics[width=14 cm]{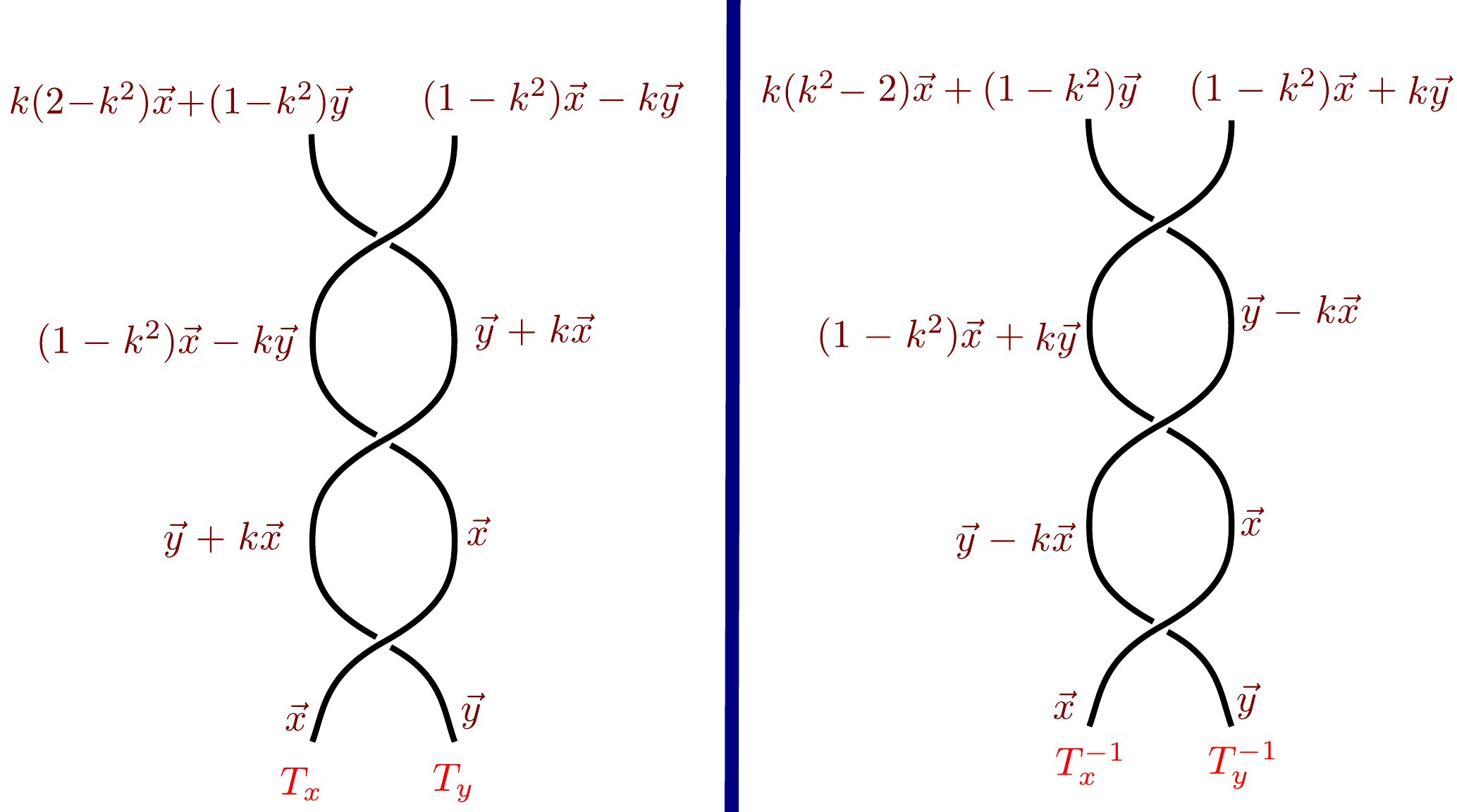}
    \caption{Labellings of twist region with homology vectors corresponding to similar handed Dehn twists}
    \label{lab1}
\end{figure}

\begin{remark} \label{same handed} We discuss the case the two  Dehn twists which have same handedness, and this case is illustrated by Figure \ref{lab1}.

If there are an even number of half twists, the labelling on the top left and top right strands are of the form $(1-k^2a)\vec x+ kb\vec y$ and $-kb \vec x+ (1-k^2c)\vec y$ respectively for some integers $a,b,c$.
Moreover we have $a+c=\mp b$, where the sign is same as the sign of the twist region, times the exponent sum of the bottom left coloring.

If there are an odd number of half twists, the labelling on the top left and top right strands are of the form $kb\vec x+ (1-k^2a)\vec y$ and $(1-k^2a) \vec x+ kc\vec y$ respectively for some integers $a,b,c$.
Moreover we have $b+c=\pm (1-k^2a)$, where the sign is same as the sign of the twist region, times the exponent sum of the bottom left coloring.
\end{remark}

\begin{figure}[H]
    \centering
    \includegraphics[width=14 cm]{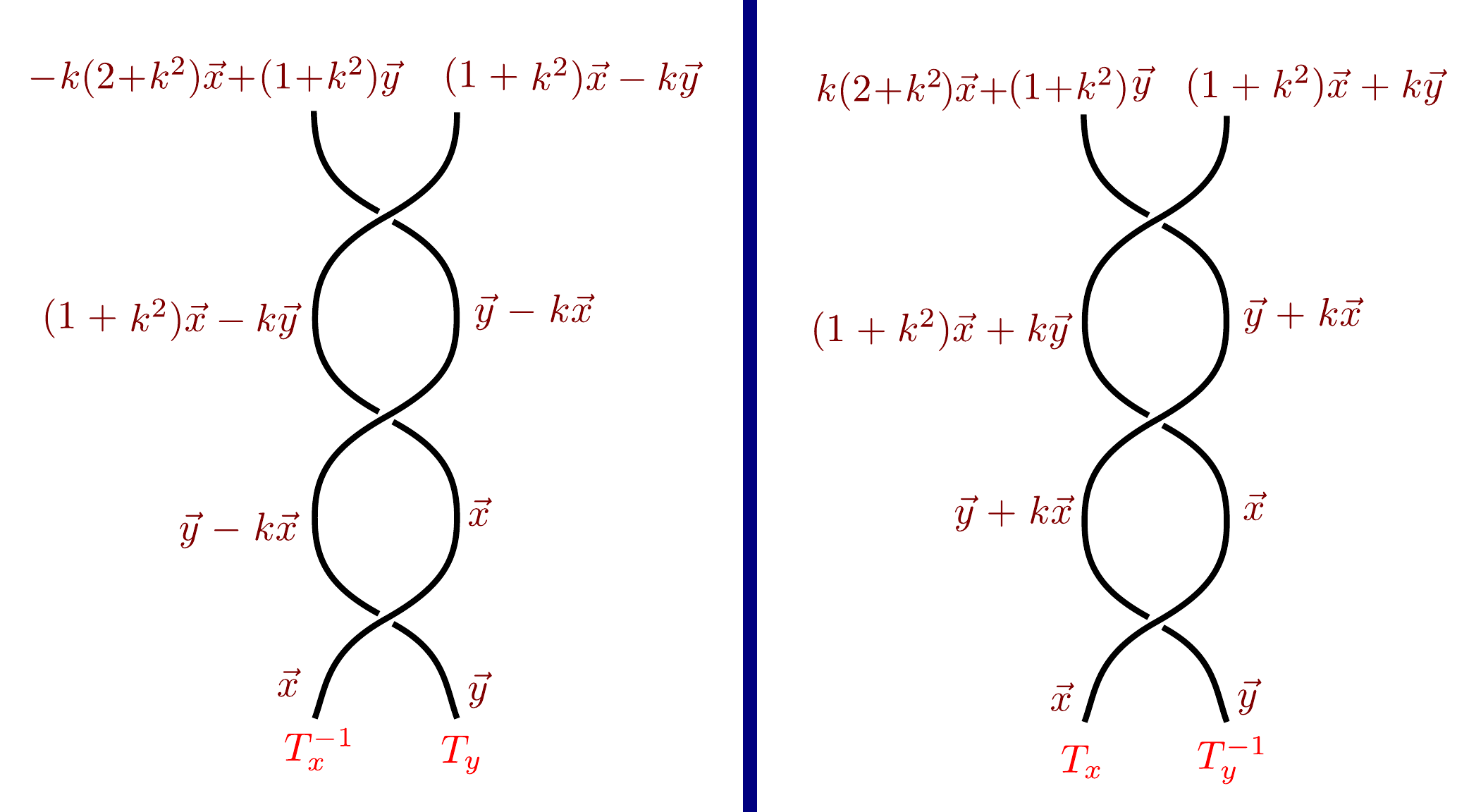}
    \caption{Labellings of twist region with homology vectors corresponding to different handed Dehn twists}
    \label{lab2}
\end{figure}
\begin{remark}\label{opp handed} We now discuss the case of two Dehn twists which have opposite handedness, see Figure \ref{lab2}.

If there are an even number of half twists, the labelling on the top left and top right strands are of the form $(1+k^2a)\vec x+ kb\vec y$ and $kb \vec x+ (1+k^2c)\vec y$ respectively for some integers $a,b,c$.
Moreover we have $a-c=\pm b$, where the sign is same as the sign of the twist region, times the exponent sum of the bottom left coloring.

If there are an odd number of half twists, the labelling on the top left and top right strands are of the form $kb\vec x+ (1+k^2a)\vec y$ and $(1+k^2a) \vec x+ kc\vec y$ respectively for some integers $a,b,c$.
Moreover we have $b-c=\pm (1+k^2a)$, where the sign is same as the sign of the twist region, times the exponent sum of the bottom left coloring.

Note that in this case if the number of half twists increase, then so do these coefficients $a,b$ and $c$. This is because the functions $C_n(k)$ and $D_n(k)$ are strictly increasing for $k\neq 0$, by the recurrence relations.

\end{remark} 

\begin{remark}
Consider $\vec x$ and $\vec y$ to have algebraic intersection number 1, and consider the two situations in Figure \ref{lab1}. We see the labellings on the top strands are $\vec x$ and $-\vec y$ for the sub figure on the left (corresponding to two right handed Dehn twists), and $-\vec x$ and $\vec y$ for the sub-figure on the right (corresponding to two left handed Dehn twists). We see the labellings at the top and bottom match up only if we allow the sign ambiguity. However, if we close up the twist region, we get a trefoil, with the standard non-trivial braid coloring. This example is to illustrate that it is necessary that we allow the sign ambiguity when we are dealing with labellings.
\end{remark}





\subsection{Pretzel Knots: Introduction} 
A pretzel link $P(q_1, q_2,...,q_m)$ has a diagram with $m$ twist regions joined up as illustrated in Figure \ref{pret}, where there are $q_i$ (which can be both positive and negative and zero) is the number of half-twists in the $i$-th region. 
\begin{figure}[H]
    \centering
    \includegraphics[width=9 cm]{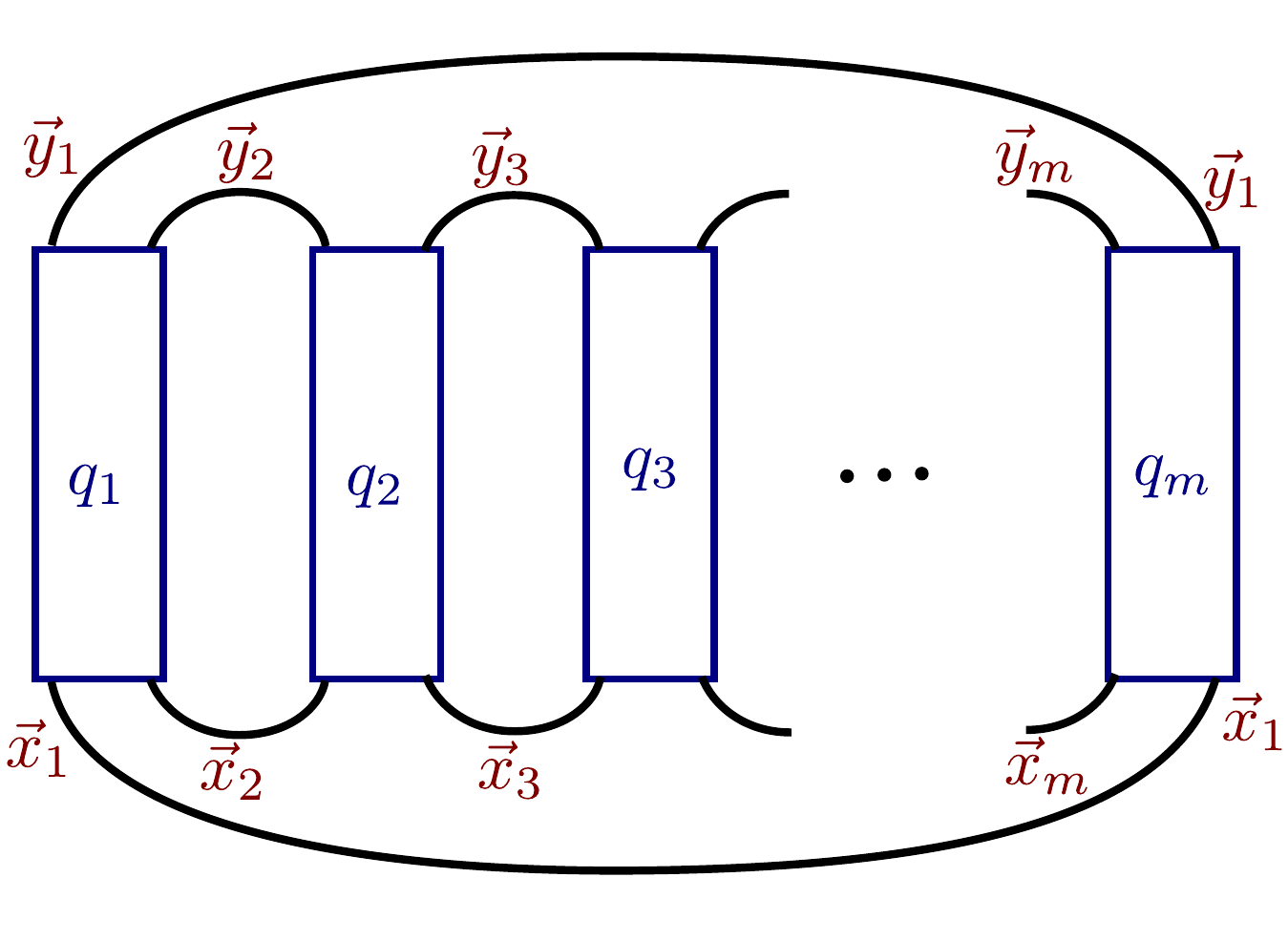}
    \caption{A pretzel link $P(q_1, q_2,...,q_m)$.}
    \label{pret}
\end{figure}
\begin{fact}

 We recall some basic facts about pretzel links.
\begin{itemize}
    \item $P(q_1, q_2,...,q_m)$ is isotopic to $P(q_2,...,q_m, q_1)$.
    \item $P(q_1, q_2,...,q_m)$ is a knot if and only if 
    either exactly one of the $q_i$ is even, or all the $q_i$'s and $m$ are odd.
    \item For a pretzel knot, if we pick any orientation on the knot, then one of the following possibilities occur:
    \begin{enumerate}
     \item in every twist region the strands go in opposite directions (this happens if and only if all the $q_i$'s and $m$ are odd);
     \item in every twist region the strands go in same directions (this happens if and only if exactly one of the $q_i$ is even and $m$ is even);
     \item in exactly one twist region the strands go in opposite direction (this happens if and only if exactly one of the $q_i$ is even and $m$ is odd).
     \end{enumerate} 
\end{itemize}
\end{fact}

\begin{claim} If we have any $G$-coloring on a pretzel link, and suppose
that for some twist region the color on the bottom left  (respectively right) strand is the same as the color on the top left (respectively right) strand, where we look at the color individually on each twist region orienting the strands from bottom to top. Then same holds in any twist region, i.e. the colors on the bottom strands is the same as the colors on the top strands (in the same order).
\end{claim}
\begin{proof}
In any twist region, the products of the meridians (in the fundamental group of the complement) we have around the top two strands is the same as that around the bottom two strands, since a loop enclosing the bottom two strands is homotopic to the loop bounding the top two strands. 
Alternately, this can be deduced by repeated application of the Wirtinger relations, notice that the Wirtinger relation is precisely the statement in the previous sentence if the twist region has one crossing. If we are given that the $G$-coloring fixes the colors on the left strand (which is equivalent to both left and right strand), we see the same has to be true for the rightmost strand of the twist region to the left, and we get the result by repeating this observation.
\end{proof}

\begin{prop} \label{equallab}
Suppose we have a non-trivial simple $B_3$-coloring, so that the associated labelling by primitive vectors in $\mathbb{Z}^2$ on two strands on the bottom of a twist region in any pretzel knot $P(q_1,...,q_m)$ are the linearly dependent (i.e. same or negatives of each other). Then in every twist region where the coloring is non-constant must have the strands gong the same way, and the number of half twists must be a multiple of 3, and we have the pattern from Remark 
\ref{2braid} repeated in each such a twist region.
\end{prop}
 \begin{proof}
 As we have a non-trivial braid coloring, so there must be two twist regions, say $T_1$ and $T_2$ next to each other with say $T_1$ on the left having the same labellings, and $T_2$ on the right different labellings.
 We note that this twist region with different labellings must go in the same direction, as otherwise we cannot get the same label (since the functions $C_n(k)$ and $D_n(k)$ are monotonically increasing by Remark \ref{opp handed}). 
 
 
 Since we have the strands going the same way in the twist region on the right, we know that the exponent sums of the coloring top left and bottom left strands on $T_2$ must be the same, and hence it follows that the same should be true for the exponent sums of the coloring top right and bottom left strands on $T_2$. In particular, for the twist region $T_2$  the colorings on the top left strand is the same as the colorings of the bottom left strand, and now the result follows from the previous claim and Remark \ref{same handed}.
\end{proof}

As a consequence of the above proposition, when we are analyzing simple $B_3$- colorings on pretzel knots, we may assume the associated labellings in each twist regions are distinct, since otherwise we understand the coloring extremely well by the above proposition.
\begin{remark}\label{basis}
We introduce some notation for future use. Suppose we label the strands at the bottom and top with the labels $\vec x_1$,...,$\vec x_m$ and $\vec y_1$,...,$\vec y_m$ as illustrated in Figure \ref{podd}.

Let us denote by $k_i$ the algebraic intersection number
$\langle \vec x_i, \vec x_{i+1} \rangle$.

Since we assume for all $i$, $\vec x_{i-1}$ and $\vec x_{i}$ are linearly independent we can express each $x_{i+1}$ as a rational linear combination of $x_{i-1}$ and $x_{i}$, $$\vec x_{i+1}=\alpha_{i} \vec x_{i-1}+\beta_{i} \vec x_i \text { for some } \alpha_{i}, \beta_{i} \in \mathbb{Q}.$$

Notice that we have:

\begin{align}\label{eqa}
    k_i=\langle \vec x_i, \vec x_{i+1} \rangle =\langle \vec x_i, \alpha_{i} \vec x_{i-1}+\beta_{i} \vec x_i \rangle=\alpha_{i}\langle \vec x_i, \vec x_{i-1} \rangle=-\alpha_{i}k_{i-1} \text{, and}
\end{align}
\begin{align}\label{eqb}
    \langle \vec x_{i-1}, \vec x_{i+1} \rangle =\langle \vec x_{i-1}, \alpha_{i} \vec x_{i-1}+\beta_{i} \vec x_i \rangle=\beta_{i}\langle \vec x_{i-1}, \vec x_{i} \rangle=\beta_{i}k_{i-1}.
\end{align}

\end{remark}

\subsection{Pretzel knots I: three twist regions and all odd}
In this subsection, we show the first part of Theorem~\ref{thm3pret3}.
Let us consider the pretzel knot $P(p,q,r)$ with $p,q,r$ all odd. The Alexander polynomial of this knot is \cite[Example 6.9]{L}
$$\Delta_{P(p,q,r)}(t)=\frac{1}{4}((pq+qr+rp)(t^2-2t+1)+t^2+2t+1)).$$
In particular the determinant of such a pretzel knot is $\Delta_{P(p,q,r)}(-1)=pq+qr+rp$ (this formula for the determinant is true for more general pretzel knots, see \cite[Section 8]{K1}) and so pretzel knots $P(p,q,r)$ is tricolorable if and only if 3 divides $pq+qr+rp$.

It is easy to see that $\Delta_{P(p,q,r)}(t)$ is divisible by $1-t+t^2$, the Alexander polynomial of the trefoil knot (which is a pretzel knot $P(1,1,1)$ or $P(-1,-1,-1)$ depending on the handedness) if and only if $pq+qr+rp=3$. It follows from Theorem \ref{alpolydiv} that if the determinant $pq+qr+rp$ is any multiple of three different from three, then any associated tricoloring does not lift to a simple $B_3$-coloring.
In this subsection we will show that the only such pretzel knots which admit a simple $B_3$-coloring are the left and right handed trefoils.
We remark that the knot group of the trefoil (recall the knot groups of any knot in $S^3$ and it's mirror are isomorphic) is isomorphic to the braid group
$B_3$ on three strands, and this isomorphism sends meridians to half-twist, and gives us the desired braid coloring.

\begin{figure}[H]
    \centering
    \includegraphics[width=9 cm]{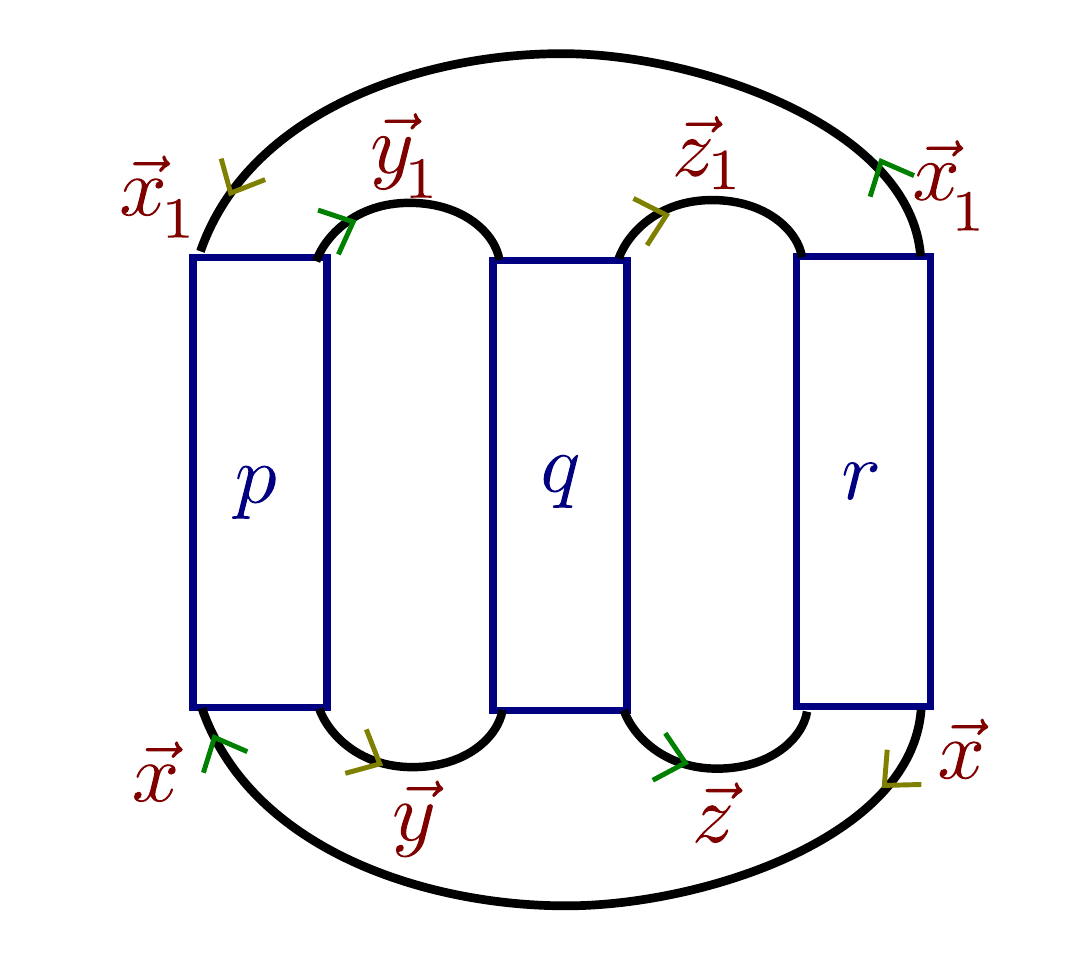}
    \caption{A three strand pretzel knot with $p,q$ and $r$ are all odd. The arrows (in different colors) indicate an orientation of the knot in each twist region.}
    \label{3odd}
\end{figure}

Suppose we have a pretzel knot $P(p,q,r)$, with the strands in the bottom labelled by $\vec x,\vec y,\vec z$ and the three strands on top labelled by $\vec x_1,\vec y_1,\vec z_1$, as illustrated in Figure \ref{3odd} , and let us suppose $\langle \vec x, \vec y \rangle =k$ , $\langle \vec y, \vec z \rangle =l$  and $\langle \vec z, \vec x \rangle =m$.

By our discussion in the last subsection, from the leftmost twist region we get:
$$\pm\vec x_1=(1+k^2a)\vec y+kb \vec x \text { and } \pm\vec y_1=(1+k^2a)\vec x+kc \vec y.$$
From the middle twist region we get:
$$\pm\vec y_1=(1+l^2d)\vec z+le \vec y \text { and } \pm\vec z_1=(1+l^2d)\vec y+lf \vec z.$$
From the rightmost twist region we get:

$$\pm\vec z_1=(1+m^2g)\vec x+mh \vec z \text { and }  \pm\vec x_1=(1+m^2g)\vec z+mi \vec x.$$

In case any two of $\vec x$, $\vec y$ and $\vec z$ are linearly dependent, any braid coloring has to be trivial by Proposition \ref{equallab}. So we may assume that we are in the non trivial case of $\vec x$ and $\vec y$ are linearly independent in $H_1(S_{1,1})\cong\mathbb{Z}^2$, and thus forms a basis for $\mathbb{Q}^2$ over the rationals, and thus we can express $\vec z$ as a rational linear combination of $\vec x$ and $\vec y$, $\vec z=\alpha\vec x+\beta\vec y$.

Note that $l=\langle \vec y, \vec z \rangle= \alpha \langle \vec y, \vec x \rangle =-\alpha k$ and
$m=\langle \vec z, \vec x \rangle= \beta \langle \vec y, \vec x \rangle =-\beta k.$\\
Comparing coefficients in the expansion of $\vec x_1$, for some $\eta_1\in\{\pm1\}$ we obtain:

\begin{align}\label{3odd1}
    \eta_1 (1+k^2a)=\beta (1+m^2g);\quad  \eta_1 kb= \alpha (1+m^2g)+mi
\end{align}
Similarly, comparing coefficients in the expansion of $\vec y_1$, for some $\eta_2\in\{\pm1\}$ we get:
\begin{align}\label{3odd2}
    \eta_2 (1+k^2a)=\alpha(1+l^2d);\quad  \eta_2 kc= \beta(1+l^2d)+le
\end{align}
Finally, comparing coefficients in the expansion of $\vec z_1$, for some $\eta_3\in\{\pm1\}$ we have:
$$\eta_3lf\alpha =(1+m^2g)+mh\alpha;\quad  \eta_3 (lf\beta+(1+l^2d))= mh\beta $$
Equivalently, we obtain:
\begin{align}\label{3odd3}
    1+m^2g=\alpha (\eta_3lf-mh) ;\quad  1+l^2d= \beta (\eta_3mh -lf)= -\eta_3\beta (\eta_3lf-mh)
\end{align}
From the equalities in the left of Equations \eqref{3odd1} and \eqref{3odd2}  we obtain:
\begin{align}\label{3oddA}
    |k|(1+k^2a)=|l|(1+l^2d)=|m|(1+m^2g)
\end{align}

\begin{claim} \label{cl1} The integers $k,l$ and $m$ are pairwise coprime.

\end{claim}
\begin{proof}
Suppose not, say $\gcd(k,l)>1$ (a similar argument works for other pairs).
From the second equality in Equation \eqref{3odd2} we get
$\eta_2k^2c=-m(1+l^2d)+kle$. It follows that $\gcd(k,l)^2$ divides $m(1+l^2d)$, and hence $\gcd(k,l)^2$ divides $m$. From Equation \eqref{3oddA}, it follows that $\gcd(k,l)^2$ divides $k(1+k^2a)$ and $l(1+l^2d)$, and consequently it divides $k$ and $l$.
Thus,  $\gcd(k,l)^2$ must divide $\gcd(k,l)$, a contradiction.
\end{proof}


Thus $k$ and $l$ are coprime integers both dividing $1+m^2g$, and we must have for some integer $\theta$,
$1+m^2g=kl\theta$. Note that by Equation \eqref{3oddA}, $\theta$ must be coprime to $k,l$ and $m$. Suppose we set $\Delta=\eta_3lf-mh$, then we note that $\Delta=-k^2\theta$. We see by using Equations \eqref{3odd1}, \eqref{3odd2} and \eqref{3odd3} that $$\frac{\eta_1\alpha}{\eta_2\beta}=\frac{1+m^2g}{1+l^2d}=-\frac {\alpha\Delta}{\eta_3 \beta\Delta}=-\frac {\alpha}{\eta_3 \beta},$$ and hence we have $\eta_1\eta_3=-\eta_2$. Now we see
$$kb+kc=\eta_1\alpha (1+m^2g)+\eta_1mi+ \eta_2\beta(1+l^2d)+ \eta_2le$$
Since we know that $\eta_1\eta_3=-\eta_2$ we have $$kb+kc=\eta_1(\alpha (1+m^2g)+mi-\eta_3\beta(1+l^2d)- \eta_3le))$$
Now taking $\eta_1$ to the other side and substituting $$\eta_1 (kb+kc)=\alpha (1+m^2g)+mi-\eta_3\beta(1+l^2d)- \eta_3le= (\alpha^2 + \beta ^2)\Delta +mi -\eta_3le$$
$$= (\alpha^2 + \beta ^2)\Delta -\Delta +m(i-h)-\eta_3 l(e-f)$$
$$= (\alpha^2 + \beta ^2)\Delta -\Delta \pm m(1+m^2g)\pm l(1+l^2d)$$
$$=\Delta (\alpha^2 + \beta ^2-1\pm \alpha m \pm \beta l)=-\theta (l^2+m^2-k^2 \pm \gamma klm) $$

\noindent Thus $\theta$ divides $kb+kc$
and $kb-kc=\pm k(1+k^2a)$. Consequently $\theta$ divides $2kb$ and $1+k^2a$, and since $\vec x_1$ is primitive, it must be the case that $\theta\in\{\pm 1, \pm 2\}.$ Thus we have:


\begin{align}\label{3oddB}
    1+k^2a=|\theta lm|,\quad 1+l^2d=|\theta km|, \quad 1+m^2g=|\theta kl|
\end{align}

\noindent Note that one of $a,d$ or $g$ is 0 if and only if $p,q$ or $r$ equals $\pm 1$, respectively. If that happens, then we claim that must have $P(p,q,r)$ is a (left or right handed) trefoil.

For instance, if $a=0$ (the argument holds similarly for $d$ or $g$ being 0), then $1=|\theta lm|$, which means $|\theta|=|l|=|m|=1$, and so $|k|=1+d=1+g$. Since $a=0$, one of $b$ or $c$ must be 0 as well. Thus by Equations \eqref{3odd1} or \eqref{3odd2}, we see that one of $e$ or $i$ has to be 1,
which in turn implies one of $d$ or $g$  has to be 0 or 1. Since we know $|k|=1+d=1+g$, $|k|$ must be either equal to 1 or 2. We now have the case $|p|=1$, and since $d=g$, we must have  $|q|=|r|$, which must lie in $\{1,3\}$, with $|k|\in\{1,2\},$ and $|l|=|m|=1$. One can check\footnote{This can be done either directly, by trying to find solution to the system of equations, or by using the Alexander polynomial criterion, or by using Theorem \ref{2bridge}. For the last method, observe that $\pm P(1,3,-3)$ is isotopic to the $\pm 6_1$ knot, and $\pm P(1,3,3)$ is isotopic to the $\pm 7_4$ knot, and clearly $\pm P(1,-1,1)$ is the unknot.} that we get non-trivial colorings only on the trefoils, $\pm P(1,1,1)$ and $\pm P(3,3,-1)$.

Without loss of generality assume $|k|\leq |l|\leq m$. 
Let us now suppose all of $a,d,g$ are at least 1. It cannot be the case that $|\theta|=1$, since otherwise $1+m^2g$ will not strictly bigger than $|kl|$.

Thus $\theta=\pm 2$, and it follows that $k,l,m$ and $a,d,g$ are odd. If $g>1$, we again see that we cannot have $1+m^2g=2|kl|$, hence $g=1$. 

If $|m|=1$, it implies $|k|=|l|=1$. Thus by Equation \eqref{3oddB} we have $a=d=g=1$, and so $|p|=|q|=|r|=3$. In this case 9 divides the determinant of $P(p,q,r)$, and consequently there cannot be a non-trivial simple $B_3$ coloring. If $|m|=3,$ then $1+m^2g=10=2|kl|$, which implies either $|k|$ or $|l|$ has to be at least $5$, contradicting the maximality of $|m|$. Thus $|m|>4$, and so $2|kl|=1+m^2g=1+m^2>4|m|$, i.e $|m|\leq \frac{|kl|}{2}$. In particular this implies $|k|\geq 2$, and thus we also have the inequalities $|l|\leq \frac{|km|}{2}$ and $|k|\leq \frac{|lm|}{2}$. Thus $\vec x, \vec y$ and $\vec z$ are 1-proportional, and hence by \cite[Theorem 1.1]{K2}, there cannot be any relations between $T_x, T_y$ and $T_z$.
Hence, we have shown:
\begin{prop} Consider the three strand pretzel knot $P(p,q,r)$ with $p,q,r$ odd. Then a non-trivial tricoloring on $P(p,q,r)$ lifts
if and only if $\pm (p,q,r)\in\{(1,1,1),(3,3,-1),(3,-1,3),(-1,3,3)\}$.

\end{prop}

Recall $P(p,q,r)$ is tricolorable if and only if 3 divides $pq+qr+rp$, and so the above proposition gives infinitely many tricolrings of three strand odd pretzel knots which do not lift to braid colorings.

\begin{remark}
It follows from the discussion above that the only odd three stranded pretzel knots which is isotopic to the trefoil are the four in the statement of the proposition above.
\end{remark}

We are yet to see an example of a knot with two distinct tricolorings, one of which lift to $B_3$ and the other do not. In the next subsection we will see such examples using similar techniques as our discussion above.


\subsection{Pretzel knots II: three twist regions and one even}

In this subsection, we prove the second part of Theorem~\ref{thm3pret3}.
Let us consider the pretzel knot $P(p,q,r)$ with $p$ even and $q,r$  odd. Suppose we have a diagram of $P(p,q,r)$ with the strands in the bottom labelled by $\vec x,\vec y,\vec z$ and the three strands on top labelled by $\vec x_1,\vec y_1,\vec z_1$, as illustrated in Figure \ref{2odd1e},

\begin{figure}[H]
    \centering
    \includegraphics[width=9 cm]{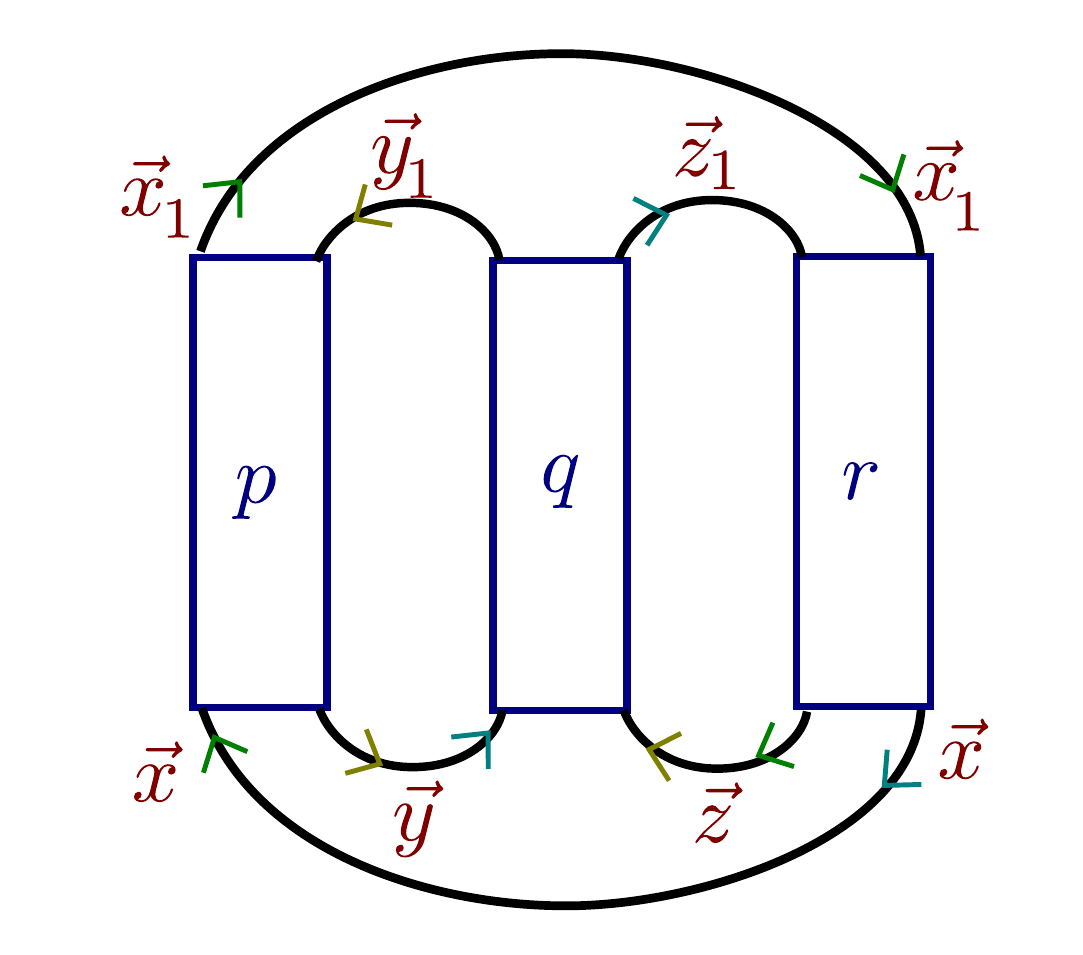}
    \caption{A three strand pretzel knot with $p$ even and $q$ and $r$ are odd, together with labellings on strands.}
    \label{2odd1e}
\end{figure}

\noindent and let us suppose $\langle \vec x, \vec y \rangle =k$ , $\langle \vec y, \vec z \rangle =l$  and $\langle \vec z, \vec x \rangle =m$. 

\noindent In the leftmost twist region we get:
$$\pm\vec x_1=(1+k^2a)\vec x+kb \vec y \text { and } \pm\vec y_1=kb\vec x+ (1+k^2c)\vec y.$$

\noindent From the middle twist region we get:
$$\pm\vec y_1=(1-l^2d)\vec z+le \vec y \text { and } \pm\vec z_1=(1-l^2d)\vec y+lf \vec z.$$

\noindent From the rightmost twist region we get:

$$\pm\vec z_1=(1-m^2g)\vec x+mh \vec z \text { and }  \pm\vec x_1=(1-m^2g)\vec z+mi \vec x.$$

In case any two of $\vec x$, $\vec y$ and $\vec z$ are linearly dependent, any braid coloring has to be trivial by Proposition \ref{equallab} or constant on the leftmost twist region. So we may assume that we are in the non trivial case of $\vec x$ and $\vec y$ are linearly independent in $H_1(S_{1,1})\cong\mathbb{Z}^2$, and thus forms a basis for $\mathbb{Q}^2$ over the rationals, and thus we can express $\vec z$ as a rational linear combination of $\vec x$ and $\vec y$, $\vec z=\alpha\vec x+\beta\vec y$.

Note that $l=\langle \vec y, \vec z \rangle= \alpha \langle \vec y, \vec x \rangle =-\alpha k$ and
$m=\langle \vec z, \vec x \rangle= \beta \langle \vec y, \vec x \rangle =-\beta k.$ Comparing coefficients in the expansion of $\vec x_1$, for some $\eta_1\in\{\pm1\}$ we obtain:

\begin{align}\label{3ev1}
    \eta_1 (1+k^2a)=\alpha (1-m^2g)+mi;\quad  \eta_1 kb= \beta (1-m^2g)
\end{align}

\noindent Similarly, comparing coefficients in the expansion of $\vec y_1$, for some $\eta_2\in\{\pm1\}$ we get:

\begin{align}\label{3ev2}
    \eta_2kb =\alpha(1-l^2d);\quad  \eta_2 (1+k^2c)= \beta(1-l^2d)+le
\end{align}

\noindent Finally, comparing coefficients in the expansion of $\vec z_1$, for some $\eta_3\in\{\pm1\}$ we have:
$$\eta_3lf\alpha =(1-m^2g)+mh\alpha;\quad  \eta_3 (lf\beta+(1-l^2d))= mh\beta $$

\noindent Equivalently, we obtain:

\begin{align}\label{3ev3}
    1-m^2g=\alpha (\eta_3lf-mh) ;\quad  1-l^2d= \beta (\eta_3mh -lf)= -\eta_3\beta (\eta_3lf-mh)
\end{align}

\noindent From the equalities in the right of Equation \eqref{3ev1} and left of Equation \eqref{3ev2}  we obtain:
\begin{align}\label{3evA}
    k^2b=|l(1-l^2d)|=|m(1-m^2g)|
\end{align}

\begin{claim} \label{cl2} The integers $k,l$ and $m$ are pairwise coprime.

\end{claim}
\begin{proof} 
From Equation \eqref{3evA} we see that $\gcd(k,l)^2$ divides $l$ and $\gcd(k,m)^2$ divides $m$. Now, we observe that from Equation \eqref{3ev3} that  $k(1-m^2g)=-l(\eta_3lf-mh)$, and consequently $\gcd(l,m)^2$ divides $k$.

Suppose now $\gcd(l,m)>1$, and suppose $x$ is a prime power exactly dividing $\gcd(l,m)$, and so $x^2$ divides $k$. Without loss of generality, assume $x$ exactly divides $l$, and consequently $x$ divides $\gcd(k,l)$, and hence $x^2$ divides $l$, a contradiction. Thus $\gcd(l,m)=1$.

From Equation \eqref{3ev1}, we see that $\alpha(1-m^2g)$ and $\beta(1-m^2g)$, and consequently $k$ divides $l(1-m^2g)$ and $m(1-m^2g)$. Since $\gcd(l,m)=1$, it follows that $k$ divides $1-m^2g$, and thus $\gcd(k,m)=1$. A similar argument shows $\gcd(k,l)=1$. 
\end{proof}

Thus $k^2$ and $l$ are coprime integers both dividing $1-m^2g$, thus we must have for some integer $\theta$,
$1-m^2g=k^2l\theta$. Hence, from Equation \eqref{3ev3}, we see that $l(1-l^2d)=-\eta_3 m(1-m^2g)=-\eta_3 k^2lm\theta$, and so $1-l^2d=-\eta_3 k^2m\theta$. Similarly, from Equation \eqref{3ev1}, we see $k^2b=-\eta_1k^2lm\theta$.

Note that by Equation \eqref{3evA}, $\theta$ must be coprime to $l$ and $m$.
From Equations \eqref{3ev1}, \eqref{3ev2} and \eqref{3ev3} we see that:
$$\frac{\eta_1}{\eta_2}=\frac{\beta(1-m^2g)}{\alpha(1-l^2d)}=\frac{m(1-m^2g)}{l(1-l^2d)}=-\frac{1}{\eta_3}.$$
and hence $\eta_1\eta_2\eta_3=-1$. From Equations \eqref{3ev1} and \eqref{3ev2} we get:

$$1+k^2a+ 1+k^2c=\eta_1\alpha(1-m^2g)+\eta_1 mi + \eta_2\beta(1-l^2d)+\eta_2le$$
$$=\eta_1\alpha(1-m^2g)+\eta_1(mi-mh)+\eta_1mh + \eta_2\beta(1-l^2d)+\eta_2l(e-f)+\eta_2lf$$
$$=\eta_1\alpha(1-m^2g)-\eta_1\sgn(r)m(1-m^2g)+\eta_2\beta(1-l^2d)+\eta_2\sgn(q)l(1-l^2d)+\eta_2(lf-\eta_3mh)$$

\noindent We note that each summand on the right hand side is an integer divisible by $k\theta$, by using the formulas
$1-m^2g=k^2l\theta$, $1-l^2d=\eta_3k^2m\theta$, and $k(1-m^2g)=-l(\eta_3lf-mh)$.

Since $k\theta$ divides the sum and difference of $1+k^2a$ and $1+k^2c$, and $\vec x_1$ is primitive; it follows that $|k\theta|$ has to be either 0, 1 or 2.

Let us consider various cases:
\begin{enumerate}

    \item $m(1-m^2g)=0$ (Note that by Equation \eqref{3evA}, this contains the case $|k\theta|=0$): Since $m(1-m^2g)=0$, we see that either $m=0$, or $1-m^2g=0$, which implies $|m|=1$. In either case, the coloring by Dehn twists on the right twist region starts and ends with the same colors, hence by Proposition \ref{equallab}, the same must be true for all the twist regions.
    This means $k=0$, and $|m|,|l|$ can be either 0 or 1. The braid coloring is non-trivial if one of $m$ and $l$ is non-zero, which means the other has to be as well.
    We see that   all non-trivial braid colorings we get in this case which are constant in the leftmost twist region, and non-constant in the middle and right twist regions (but the coloring matches at the ends of either twist regions). In particular this can happen only when $q$ and $r$ are odd multiples of 3. 

    \item $|k|=1$ and $|\theta|=1$: we see  then $|1-m^2g|=|l|$ and $|1-l^2d|=|m|$. We consider the subcases: \\
    \begin{itemize}
        \item  $dg=0$: Without loss of generality let us assume $d=0$, which implies then $|m|=1$, and hence $|l|=|1-g|$.
Since $|m|=1$, and the strands are going in the same direction in the rightmost twist region, we know the pattern repeats every three half twists, we can assume $r\in\{\pm 1, \pm 3\}$. Consequently the possible values of $g$ are $0$ and $1$. Hence $|l|$ has to be 1 or 0 (which we already dealt with). So $b=1$, and hence $p=\pm 2$. 
let us assume $p=2$, by mirroring the pretzel knot if necessary. Note that $P(2,-1,-1)$, is the trefoil, and has a tricoloring which lifts to a simple $B_3$ coloring. Hence, by adding multiples of 6 to the right and middle twist regions, we obtain simple $B_3$-colorings\footnote{Note that  for these knots the determinant is $\pm 3$, and so there is only one tricoloring (up to conjugation), which lifts to a unique (up to conjugation) simple $B_3$-coloring.} on $\pm P(2, 6q_0-1,6r_0-1)$, for any integers $q_0,r_0$; and we see that no other pretzel knot has a non-trivial simple $B_3$ coloring of this type.


\item $dg\neq 0$: If $d\neq 1$, then we see $|m|=|1-l^2d|>|l|$, and similarly if $g\neq 1$, then $|l|=|1-m^2g|>|m|$. Since we cannot simultaneously have $|l|>|m|$ and $|m|>|l|$, it must be that one of $d$ or $g$ is 1. If $d=1$, we must have $|l|>1$ (otherwise $m=0$, a case we saw earlier), so $|m|=|l^2-1|>|l|$. Similarly if $g=1$ , then either $|m|=1$ (which implies $l=0$, which we saw earlier), or $|l|=|1-m^2g|>|m|$. Hence, we do not come across any new braid colorings in this subcase.

    \end{itemize}
    
    \item $|k|=1$ and $|\theta|=2$: we see then $|1-m^2g|=2|l|$ and $|1-l^2d|=2|m|$. Thus, $m,g,l$ and $d$ have to be odd.  If $g=1$, then $1-m^2g=(1-m)(1+m)$ is divisible by 4, since $m$ is odd.
    But this means $l$ is even, since $|1-m^2g|=2|l|$, which contradicts that $l$ is odd. We get a similar contradiction if $d=1$. If $g\neq 1$, then $2|l|=|1-m^2g|>m^2$, and so either $|l|>|m|$ or $|l|=|m|=1$. Similarly, if $d\neq 1$, then either $|m|>|l|$, or $|l|=|m|=1$. Thus, we have to have $|l|=|m|=1$, since we cannot simultaneously have $|l|>|m|$ and $|m|>|l|$. In case $|l|=|m|=1$, by Figure \ref{lab1} (and the fact that the labellings repeat after six half twists) the possible values of $d$ and $g$ are 0, 1 and $2$, which leads to a contradiction, as we saw above.
    Thus, we do not come across any new braid colorings in this case.
   
   \item $|k|=2$ and $|\theta|=1$:  we see then $|1-m^2g|=4|l|$ and $|1-l^2d|=4|m|$. Hence $d,g,l$ and $m$ are odd. Without loss of generality, let us assume $|l|\leq |m|$. 
   First suppose $|l|=1$, we know by Remark \ref{2braid} (since the braid coloring repeats) and Figure \ref{lab1}; the only possible values of $d$ are 0, 1 and 2. Since $d$ has to be odd, we have $d=1$, and so $4|m|=|1-l^2d|=0$, a contradiction. If $|l|\neq 1$, then $|l|\geq 3$, and so $|m|\geq |l|+2$ since $|m|$ has to be an odd integer coprime to $|l|$. Consequently, $4|l|=|m^2g-1|\geq |m^2-1|\geq (|l|+2)^2-1>4|l|$, a contradiction. Thus, we do not obtain any new braid colorings in this case.

\end{enumerate}

By combining our discussion of the various cases (and using that $P(p,q,r)$ is isotopic to $P(q,r,p)$), we have proved:
\begin{prop} Suppose we have a three strand pretzel knot $P(p,q,r)$  with one of $p,q$ or $r$ even. Then, any tricoloring on $P(p,q,r)$ lifts to a simple $B_3$-coloring if and only if the tricoloring is constant on the twist region with an even number of half twists. Moreover, there is unique (up to conjugation) lift to $B_3$, when the tricoloring does lift.
 
\end{prop}

We know \cite[Section 8]{K1} that if $p,q,r$ are all multiples of 3, then there are four inequivalent tricolorings, depending on how many twist regions have non-constant tricolorings. We illustrate this with the four colorings on the pretzel knot $P(3,3,6)$, see Figure \ref{p336}.
\begin{figure}[H]
    \centering
    \includegraphics[width=11 cm]{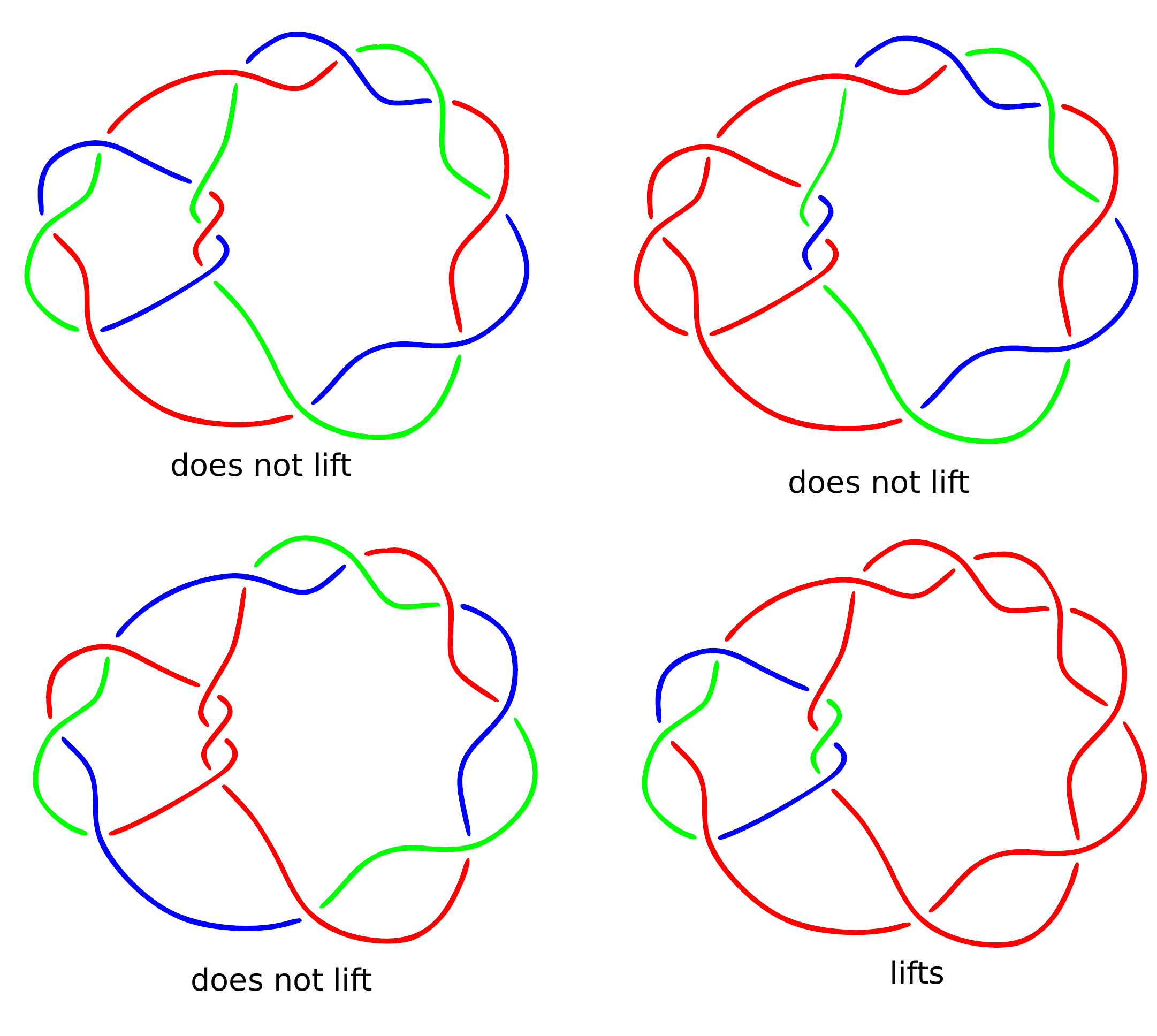}
    \caption{Four tricolorings of $P(3,3,6)$, of which only the last one lifts to a simple $B_3$-coloring.}
    \label{p336}
\end{figure}

By the Proposition above, we see that three of these tricolorings do not lift, while the fourth does.
More generally, we have:
\begin{prop} Consider the pretzel knot $P(p,q,r)$ with two of $p,q,r$ odd multiples of 3, and the third a multiple of 6. Then exactly one among the four inequivalent tricolorings of $P(p,q,r)$ lift to a simple $B_3$-coloring.
\end{prop} 

Thus we see that the answer to Question \ref{q1} is "Yes", and in fact there are infinitely many such examples.

In the next two subsections, we will study the lifting problem for more general pretzel knots, however we cannot give a complete characterization for the lifting problem 
because we cannot show analouge's of Claims \ref{cl1} and \ref{cl2}.

 



%

\subsection{Pretzel Knots III: more than three twist regions and all odd}
In this subsection, we show the part of Theorem~\ref{thm3pretm}.
Suppose we have a pretzel knot $P(q_1,...,q_m)$ with each $q_i$ and $m$ odd. Note that in this case if we pick any orientation on the knot, the two strands in each twist region are oppositely oriented.
Suppose we label the strands at the bottom and top with the labels $\vec x_1$,...,$\vec x_m$ and $\vec y_1$,...,$\vec y_m$, as illustrated in Figure \ref{podd}.

\begin{figure}[H]
    \centering
    \includegraphics[width=9 cm]{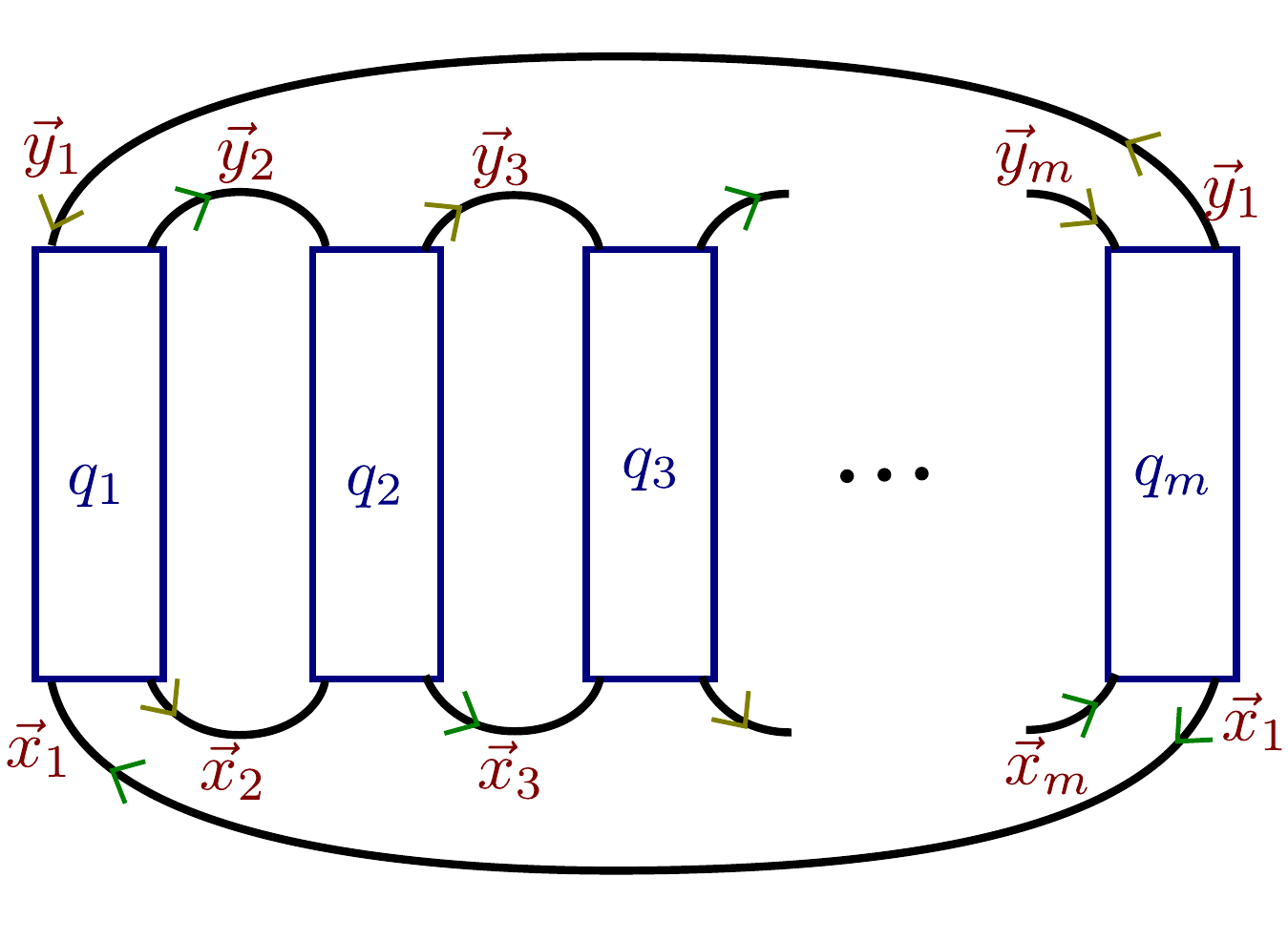}
    \caption{}
    \label{podd}
\end{figure}

Let us denote by $k_i$ the algebraic intersection number
$\langle \vec x_i, \vec x_{i+1} \rangle$. In the $i$-th twist region, we see by Remark \ref{opp handed} that $$\pm \vec y_{i}= k_ib_i\vec x_i +(1+k_i^2a_i)\vec x_{i+1} $$
$$\pm \vec y_{i+1}=(1+k_i^2a_i)\vec x_{i}+ k_ic_i\vec x_{i+1} $$

\noindent Using notation as introduced in Remark \ref{basis}, we have $$\vec x_{i+1}=\alpha_{i} \vec x_{i-1}+\beta_{i} \vec x_i \text { for some } \alpha_{i}, \beta_{i} \in \mathbb{Q}.$$

\noindent Consequently, $$ k_ib_i\vec x_i +(1+k_i^2a_i)\vec x_{i+1} = \alpha_{i}(1+k_i^2 a_i)\vec x_{i-1}+(\beta_{i}(1+k_i^2 a_i)+ k_ib_i)\vec x_i$$


\noindent Equating the labels $\vec y_{i}$ coming from the $(i-1)$-th and $i$-th twist regions we see that for some sign $\eta_i \in \{\pm 1\}$

$$\eta_i [(1+k_{i-1}^2a_{i-1})\vec x_{i-1}+ k_{i-1}c_{i-1}\vec x_{i}]= \alpha_{i}(1+k_i^2 a_i)\vec x_{i-1}+(\beta_{i}(1+k_i^2 a_i)+ k_ib_i)\vec x_i$$

\noindent Comparing the coefficients of $\vec x_{i-1}$ and $\vec x_{i}$ respectively we obtain
\begin{align}\label{eq1}
    \eta_i (1+k_{i-1}^2a_{i-1})=\alpha_{i}(1+k_i^2 a_i) \text {  and  } 
    \eta_ik_{i-1}c_{i-1}=\beta_{i}(1+k_i^2 a_i)+ k_ib_i
\end{align}

\noindent From the equality on the left we see that $\eta_i k_{i-1}(1+k_{i-1}^2a_{i-1})=-k_i(1+k_i^2 a_i)$. Since the $a_i$'s are all positive (we have Type III twist regions), we get $$|k_{i-1}|(1+k_{i-1}^2a_{i-1})=|k_i|(1+k_i^2 a_i) \text { for all } 1\leq i \leq m, \text{ and hence we have}$$
\begin{align}\label{eq2}
|k_{j}|(1+k_{j}^2a_{j})=|k_i|(1+k_i^2 a_i)\text {  for all } 1\leq i,j \leq m; 
\end{align}

\noindent  and let us denote this quantity by $K$.  Also, we have $\eta_ik_{i-1}$ has the same sign as $-k_i$. Rewriting the second equality from Equation \eqref{eq1}, we have
 

\begin{align}\label{eq3}\eta_ik_{i-1}c_{i-1}-k_ib_i=\beta_{i}(1+k_i^2 a_i) \end{align}
Since $a_i,b_i,c_{i-1}$ are all positive, we have

\begin{align}\label{eq4}
|k_{i-1}|c_{i-1}+|k_i|b_i=|\beta_{i}|(1+k_i^2 a_i)=K\frac{|\beta_i|}{|k_i|}.
\end{align}

\begin{prop} For the pretzel knot $P(q_1,...,q_m)$ with each $q_i$ and $m$ odd, if $q_j=\pm q_{j+1}$ for some $j$ with $|q_j|\geq 3$ and $|q_p|>3$ for some $p$, then any simple $B_3$-coloring on $P(q_1,...,q_m)$ must be trivial. 

\end{prop}

\begin{proof} Suppose not, let us assume there is a non-constant coloring on such a pretzel knot.
From Equation \eqref{eq2}, we have $|k_{j}|(1+k_{j}^2a_{j})=|k_{j+1}|(1+k_{j+1}^2 a_{j+1})$, or in other words
$|k_j|(1+C_{q_j}(k_j))=|k_{j+1}|(1+C_{q_j}(k_{j+1}))$. Since $C_n(k)$ is a single variable polynomial with only even degrees of $k$ appearing and all coefficients are positive, it follows that if $|k_j|<|k_{j+1}|$, then $|k_j|(1+C_{q_j}(k_j))=|k_{j+1}|(1+C_{q_j}(k_{j+1}))$, and a similar result for $|k_j|>|k_{j+1}|.$ Thus it has to be the case that $|k_j|=|k_{j+1}|$.
If $k_j=\pm k_{j+1}$, then we have 
$b_j=b_{j+1}$ and $c_j=c_{j+1}$.
Thus $b_j-c_j=\pm(1+k_{j}^2a_{j}).$


We see from Equation \eqref{eqb} that $k_j\beta_{j+1}$ is an integer, and from Equation \eqref{eq2} that $(1+k_{j+1}^2a_{j+1})\beta_{j+1}$ is an integer. Since $k_j=\pm k_{j+1}$, we have 
$k_j$ and $(1+k_{j+1}^2a_{j+1})$ are coprime, and thus $\beta_{j+1}$ is an integer. 

We see from Equation \eqref{eq3} that $c_j+b_{j+1}$ is divisible by $(1+k_{j+1}^2a_{j+1})=(1+k_{j}^2a_{j})$. So we see that both $b_j+c_j$ and $b_j-c_j$ are divisible by $1+k_{j}^2a_{j}$, and thus so is $2b_j$. This means $1+k_{j}^2a_{j}$ must equal 1 or 2, since $k_jb_j\vec x_j +(1+k_j^2a_j)\vec x_{j+1}$ must be primitive. 

If $1+k_{j}^2a_{j}=1$, then either $k_j=0$ (this implies that we have a constant braid coloring, since for any $i$, $k_i(1+k_i^2a_i)=k_j(1+k_j^2a_j)=0$) or $a_j=0$ (this implies $q_j=\pm 1$, which is ruled out by the hypothesis).

If $1+k_{j}^2a_{j}=2$, then we must have $|k_j|=1$, $a_j=1$  and $|q_j|= 3$, and thus $K=2$.
It follows that for all $i$ we must have 
$|k_i|=1$ and $|q_i|= 3$, or $|k_i|=2$ and $|q_i|=1$.

Since we know that $|q_p|>3$ for some $p$, we have a contradiction.

\end{proof}

\begin{prop}
For the pretzel knot $P(q_1,...,q_m)$ with each $q_i$ and $m$ odd, if all the $q_i$  are of the same sign, then any simple $B_3$-coloring on $P(q_1,...,q_m)$ has to be the trivial coloring.
\end{prop}

\begin{proof}

Suppose not, then $P(q_1,...,q_m)$ has a non constant coloring, which would mean some $k_i$ is non zero. It follows that $K$, and hence every $k_i, a_i,b_i$ and $c_i$ are non-zero.


\noindent Suppose all the $q_i$ are positive, then we have 
$b_i-c_i=1+k_i^2a_i$, and so from Equation \eqref{eq3}:
$$|k_{i-1}|c_{i-1}+|k_i|c_i =K \left(\frac{|\beta_i|}{|k_i|}-1 \right),\text{ and thus }|\beta_i|\geq |k_i|.$$
\noindent Suppose all the $q_i$ are negative, then we have 
$c_i-b_i=1+k_i^2a_i$, and thus from Equation \eqref{eq3}:
$$|k_{i-1}|b_{i-1}+|k_i|b_i =K \left(\frac{|\beta_i|}{|k_i|}-1 \right),\text{ and thus }|\beta_i|\geq |k_i|.$$
\noindent Thus, if all twist regions have the same sign, we have $|\beta_i|\geq |k_i|$ for all $1\leq i \leq m$.

\begin{claim} With notation as above, we have $(x_1,x_i)\leq (x_1,x_{i+1})$ for all $1\leq i \leq m$.
\end{claim}
\begin{proof} We will prove the claim by using induction on $i$. The base case follows from our assumption that $(x_1,x_2)=k_1>0=(x_1,x_1)$.
Let us inductively assume $(x_1,x_{i-1})\leq (x_1,x_{i})$. We have 
$\vec x_{i+1}=\alpha_{i} \vec x_{i-1}+\beta_{i} \vec x_i$, and hence
$$\langle  \vec x_1,\vec x_{i+1}\rangle= \alpha_{i} \langle  \vec x_1,\vec x_{i-1} \rangle+\beta_{i} \langle  \vec x_1,\vec x_i \rangle $$
\noindent and so we obtain using the triangle inequality
$$(x_1,x_{i+1})\geq -|\alpha_i|(x_1,x_{i-1})+|\beta_i|(x_1,x_{i}),$$
\noindent and by using the induction hypothesis we get
$$(x_1,x_{i+1})\geq (|\beta_i|-|\alpha_i| )(x_1,x_{i}).$$

Recall that $|\alpha_i||k_{i-1}|=(x_{i},x_{i+1})=k_i$ and $|\beta_i||k_{i-1}|=(x_{i-1},x_{i+1})$ are integers.
Now we observe that, $(|\beta_i|-|\alpha_i| )|k_{i-1}|=|\beta_i||k_{i-1}|-|k_i|>|k_i| |k_{i-1}|-|k_i|$.
In case $|k_{i-1}|=1$, then $|\beta_i|-|\alpha_i|$ is a positive integer, so has to be at least 1.
In case $|k_{i}|=1$, then $(|\beta_i|-|\alpha_i|)|k_{i-1}|>|k_{i-1}|-1$, and since the expression on the left is an integer, we get: $(|\beta_i|-|\alpha_i|)|k_{i-1}|\geq |k_{i-1}|$.
In case $|k_{i-1}|>1$ and $|k_{i}|>1$, $(|\beta_i|-|\alpha_i| )|k_{i-1}$, then $|k_i| |k_{i-1}|\geq |k_i|+ |k_{i-1}|$, and hence $(|\beta_i|-|\alpha_i|)|k_{i-1}|>|k_{i-1}|$. In any case we see $(|\beta_i|-|\alpha_i| )|k_{i-1}|\geq|k_{i-1}|$, and so $(|\beta_i|-|\alpha_i| )\geq 1$. Consequently,
$$(x_1,x_{i+1})>(|\beta_i|-|\alpha_i| )(x_1,x_{i})\geq(x_1,x_{i})$$
\text{and the claim holds}.
\end{proof}

Since $\vec x_{m+1}=\pm \vec x_1$, we have $(x_1,x_{m+1})=0$ and $(x_1,x_2)>0$, and this is a contradiction to our assumption that there is a non-trivial coloring. Thus the proposition follows  from the claim.
\end{proof}





\subsection{Pretzel Knots IV:  more than three twist regions one even} In this subsection, we prove the remaining part of Theorem~\ref{thm3pretm}, not covered in the previous subsection.
Suppose we have a pretzel knot $P(q_1,...,q_m)$ with $q_1$ even and every other $q_i$ odd, and one of the $q_i$ is a multiple of 3. 
Let us first consider the case that $m$ is even.

\begin{prop}\label{EvenA} Consider the pretzel knot $P(q_1,...,q_m)$ with $m$, $q_1$ even and every other $q_i$ odd, and one of the $q_i$ is a multiple of 3. Any tricoloring on the pretzel knot $P(q_1,...,q_m)$  lifts to a simple $B_3$- coloring.
\end{prop}

\begin{proof}

Note that in this case if we pick any orientation on the knot, the two strands in each twist region are similarly oriented. Thus for any given tricoloring, the colorings on each of the twist regions have the same transpositions, and we know by Remark \ref{2braid} that we can lift them to simple braid colorings. The result follows.
\end{proof}

Let us now consider the case $m$ is odd. In this case, if we pick any orientation on the knot, the two strands in the first (leftmost) twist region are oppositely oriented, and are similarly oriented in every other twist region.
Similar to the above proposition (we skip the proof as it is essentially the same as above), we obtain:
\begin{prop}\label{EvenB} Consider the pretzel knot $P(q_1,...,q_m)$ with $q_1$ even and every other $q_i$ and $m$  odd, and one of the $q_i$ is a multiple of 3. Any tricoloring on the pretzel knot $P(q_1,...,q_m)$ lifts to a simple $B_3$- coloring if the tricoloring is constant on the leftmost (corresponding to $q_1$) twist region (this includes the case $q_1=0$ and the two strands have different colors).
\end{prop}

However in this case ($m$ odd), there are examples of non-liftable tricolorings. Let us first set up equations corresponding to simple $B_3$-colorings, as we did in the last subsection.

Suppose we label the strands at the bottom and top with the labels $\vec x_1$, ..., $\vec x_m$ and $\vec y_1$, ..., $\vec y_m$ as illustrated in Figure \ref{pev}, as in Remark \ref{basis}.

\begin{figure}[H]
    \centering
    \includegraphics[width=9 cm]{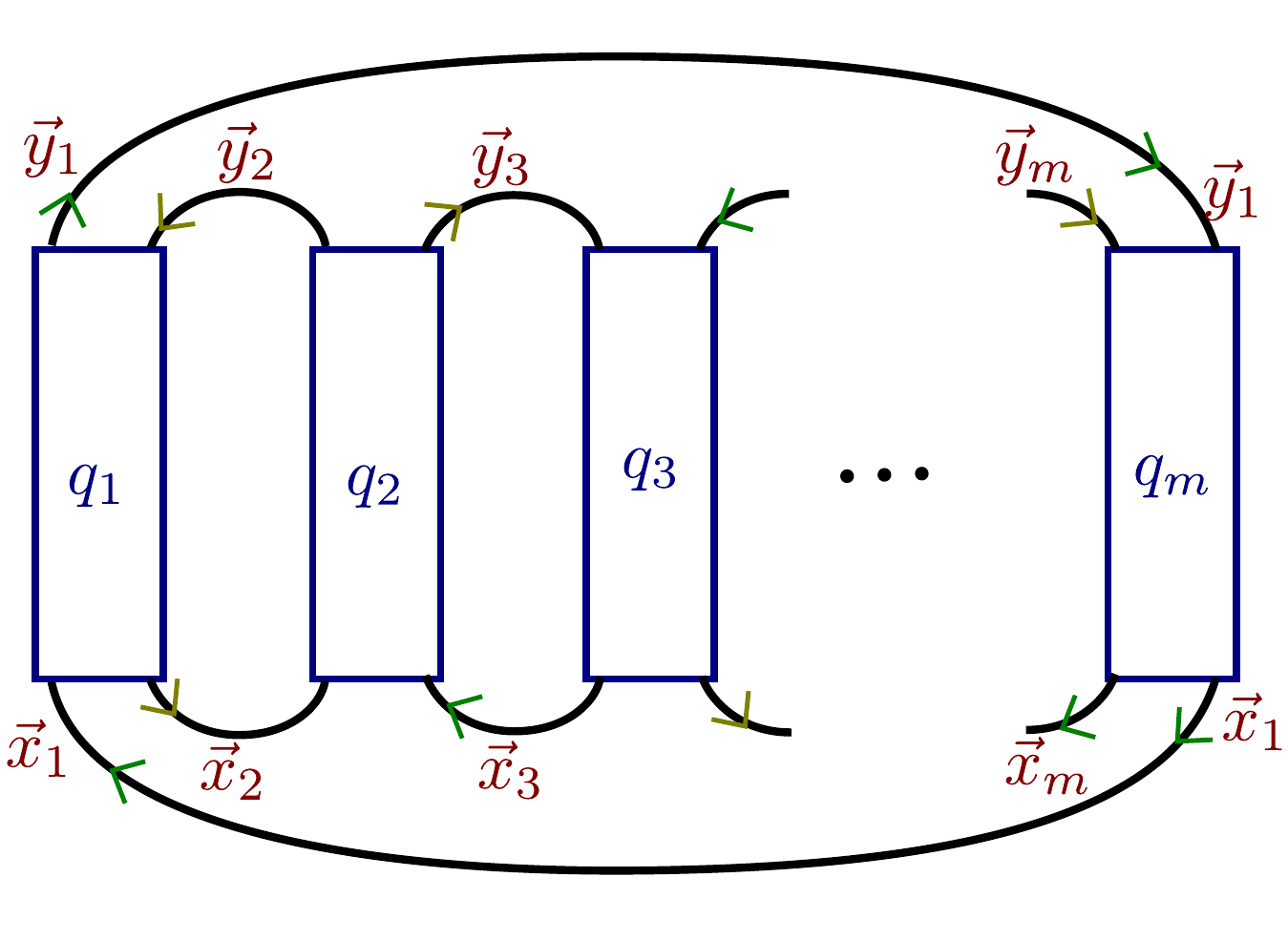}
    \caption{}
    \label{pev}
\end{figure}

\noindent In the first twist region, by Remark \ref{opp handed} we have:
$$\pm \vec y_{1}= (1+k_1^2a_1)\vec x_1 +k_1b_1\vec x_{2}, \text {and } $$
$$\pm \vec y_{2}=k_1b_1\vec x_{1}+ (1+k_1^2c_1)\vec x_{2}.$$
In the $i$-th (with $1<i\leq m$) twist region, similarly Remark \ref{same handed} we see that $$\pm \vec y_{i}= k_ib_i\vec x_i +(1-k_i^2a_i)\vec x_{i+1} $$
$$\pm \vec y_{i+1}=(1-k_i^2a_i)\vec x_{i}+ k_ic_i\vec x_{i+1} $$
Since we assume $$\vec x_{i+1}=\alpha_{i} \vec x_{i-1}+\beta_{i} \vec x_i \text { for some } \alpha_{i}, \beta_{i} \in \mathbb{Q}, \text{ we have }$$
$$(1+k_1^2a_1)\vec x_1 +k_1b_1\vec x_{2}= (1+k_1^2a_1)\vec x_1 +k_1b_1 (\alpha_1 \vec x_m +\beta_1 \vec x_1)=\alpha_1 k_1b_1\vec x_m +(1+k_1^2a_1+\beta_1k_1b_1)\vec x_1, $$
and
$$ k_ib_i\vec x_i +(1-k_i^2a_i)\vec x_{i+1} = k_ib_i\vec x_i +(1-k_i^2a_i)(\alpha_{i} \vec x_{i-1}+\beta_{i} \vec x_i)= \alpha_{i}(1-k_i^2 a_i)\vec x_{i-1}+(\beta_{i}(1-k_i^2 a_i)+ k_ib_i)\vec x_i.$$
Equating the labels $\vec y_{1}$ coming from the $m$-th and first twist regions we see that for some sign $\eta_1 \in \{\pm 1\}$
$$\eta_1 [(1-k_m^2a_m)\vec x_{m}+ k_mc_m\vec x_{1}]=\alpha_1 k_1b_1\vec x_m +(1+k_1^2a_1+\beta_1k_1b_1)\vec x_1. $$
Comparing the coefficients of $\vec x_{m}$ and $\vec x_{1}$ respectively we obtain
\begin{align}\label{pev1}
    \eta_1 (1-k_m^2a_m)=\alpha_1 k_1b_1 \text {  and  } 
    \eta_1k_mc_m=1+k_1^2a_1+\beta_1k_1b_1
\end{align}
Equating the labels $\vec y_{2}$ coming from the first and second twist regions we see that for some sign $\eta_2 \in \{\pm 1\}$
$$\eta_2 [k_1b_1\vec x_{1}+ (1+k_1^2c_1)\vec x_{2}]=\alpha_{2}(1-k_2^2 a_2)\vec x_{1}+(\beta_{2}(1-k_2^2 a_2)+ k_2b_2)\vec x_2. $$
Comparing the coefficients of $\vec x_{1}$ and $\vec x_{2}$ respectively we obtain
\begin{align}\label{pev2}
    \eta_2 k_1b_1=\alpha_{2}(1-k_2^2 a_2) \text {  and  } 
    \eta_2(1+k_1^2c_1)=\beta_{2}(1-k_2^2 a_2)+ k_2b_2.
\end{align}
For $3\leq i \leq m$, by equating the labels $\vec y_{i}$ coming from the $(i-1)$-th and $i$-th twist regions we see that for some sign $\eta_i \in \{\pm 1\}$
$$\eta_i [(1-k_{i-1}^2a_{i-1})\vec x_{i-1}+ k_{i-1}c_{i-1}\vec x_{i}]= \alpha_{i}(1-k_i^2 a_i)\vec x_{i-1}+(\beta_{i}(1-k_i^2 a_i)+ k_ib_i)\vec x_i$$
Comparing the coefficients of $\vec x_{i-1}$ and $\vec x_{i}$ respectively we obtain
\begin{align}\label{pev3}
    \eta_i (1-k_{i-1}^2a_{i-1})=\alpha_{i}(1-k_i^2 a_i) \text {  and  } 
    \eta_ik_{i-1}c_{i-1}=\beta_{i}(1-k_i^2 a_i)+ k_ib_i
\end{align}
Since $\alpha_i=-\frac{k_i}{k_{i-1}}$ for any $ 1\leq i \leq m$, we obtain
\begin{align}\label{pev4}
k_1^2b_1=|k_i(1-k_i^2 a_i)|\text {  for all } 2\leq i \leq m. 
\end{align}

\begin{prop}\label{EvenC} Consider the pretzel knot $P(q_1,...,q_m)$ with $q_1$ non-zero even and every other $q_i$ and $m$  odd, and every $q_i$ is a multiple of 3. None of the tricolorings which are non-constant on each twist region lift to a simple braid coloring.
\end{prop}

\begin{proof} Suppose not, let us assume there is a simple $B_3$-coloring lifting a tricoloring which is non-constant on each twist region.
Since $q_1$ is non-zero even multiple of 3, we see that $b_1$ is strictly positive and in fact a multiple of 4. Since we assumed that in each twist region the tricoloring is non-constant (and $q_i$ is odd for $i\neq 1$), by Claim \ref{colorofDT} we see that each $k_i$ is odd. By Equation \eqref{pev4}, we see for $1<i\leq m$, we must have $4$ divides $1-k_i^2a_i$, since $k_i$ is odd. Hence by Equation $\eqref{pev3}$, we see that for $1<i\leq m$, $\beta_i(1-k_i^2a_i)$ is an integer divisible by 4 (since $k_{i-1}\beta_i$ is an integer with $k_{i-1}$ odd). Hence for $2<i\leq m$, we have  $\eta_ik_{i-1}c_{i-1}-k_ib_i$ is a multiple of 4. Recall by Remark \ref{same handed} we know that
$k_{i}(b_i+c_i)=k_i(1-k_i^2a_i)$, and the right hand side is a multiple of $4$. Consequently, we can add these quantities (a signed telescoping sum) as follows:
$$k_2c_2-\eta_3k_3b_3=\eta_3\beta_3(1-k_3^2a_3)$$
$$-\eta_3 (k_3c_3-\eta_4k_4b_4)=-\eta_3\eta_4\beta_4(1-k_4^2a_4)$$
$$\eta_3\eta_4 (k_4c_4-\eta_5k_5b_5)=\eta_3\eta_4\eta_5\beta_5(1-k_5^2a_5)$$
$$ \cdots$$
$$(-1)^m(\eta_3\eta_4...\eta_m) (k_{m-1}c_{m-1}-\eta_mk_mb_m)=(-1)^m(\eta_3\eta_4...\eta_m)\beta_m(1-k_m^2a_m)$$

We conclude $k_2c_2 -(\eta_3...\eta_{m})k_mb_m$ is a multiple of 4, say  $4S$, for some integer $S$. We can now use Equations \eqref{pev1} and \eqref{pev2}, to relate this to $1+k_1^2a_1$ and $1+k_1^2c_1$.

$$k_2b_2 -\eta_2(1+k_1^2c_1)=\eta_2(1-k_2^2a_2)$$
$$k_2c_2 -(\eta_3...\eta_{m})k_mb_m=4S$$
$$(\eta_3...\eta_{m})(\eta_1(1+k_1^2a_1)- k_mc_m)=-(\eta_3...\eta_{m})\beta_1\eta_1k_1b_1 $$
Adding these three equations, we see that :
$$(k_2b_2+k_2c_2)-(\eta_3...\eta_{m})(k_mb_m+k_mc_m)+\eta_2[(\eta_1\eta_2...\eta_{m})(1+k_1^2a_1)-1-k_1^2c_1]$$
$$=\eta_2(1-k_2^2a_2)+4S+-(\eta_3...\eta_{m})\beta_1\eta_1k_1b_1$$
Since $4$ divides every other expression it follows that it also divides $$(\eta_1\eta_2...\eta_{m})(1+k_1^2a_1)-1-k_1^2c_1.$$ Now we claim that $\eta_1\eta_2...\eta_{m}=-1$. This follows from multiplying the equality in the left from Equations \eqref{pev1}, \eqref{pev2} and \eqref{pev3} (for all $3\leq i\leq m$), and recalling $\alpha_i=-\frac{k_i}{k_{i-1}}$. Thus it follows that 4 divides $1+k_1^2a_1+1+k_1^2c_1$, and it also divides their difference $1+k_1^2a_1-(1+k_1^2c_1)=\pm k_1^2b_1$. This means $1+k_1^2a_1$ is even, which contradicts the fact that $\vec y_1$ is primitive, and so we have proved the proposition.
\end{proof}

\begin{prop}\label{EvenD} Consider the pretzel knot $P(q_1,...,q_m)$ with $q_1$ even and every other $q_i$ and $m$  odd, and one of the $q_i$ is a multiple of 3. Suppose further  for some $j$ we have $q_j=q_{j+1}=\pm 3$. Then a tricoloring on the pretzel knot $P(q_1,...,q_m)$ lifts to a simple $B_3$- coloring if the tricoloring is constant on the leftmost twist region.
\end{prop}

\begin{proof}
Since $q_j=q_{j+1}=\pm 3$, by Figure \ref{lab1}, we have $a_j=1=a_{j+1}$. Thus the first equality of Equation \eqref{pev3} for the pair $i,i+1$ of indices we have $k_j(1-k_j^2)=\eta_j k_{j+1}(1-k_{j+1}^2)$.  Since the real valued function $x\mapsto x^3-x$ is monotonically increasing on $[1,\infty)$, it follows that we must have $k_j=\eta_{j+1}k_{j+1}$. We now see that $\beta_{j+1}$ has to be an integer, as $k_j\beta_{j+1}$ and $(1-k_j^2)\beta_{j+1}$ are (since $k_j$ is coprime to $1-k_j^2$). We see that we have by the right equality of Equation \eqref{pev3} that, $k_{j+1}(c_j-b_{j+1})=\beta_{j+1}(1-k_j^2)$, hence $1-k_j^2$ divides $c_j-b_{j+1}$.
Since $|q_j|=|q_{j+1}|=3$, we have by Figure \ref{lab1}, $b_{j+1}$ must equal $b_j$. Thus $c_j-b_j$ is divisible by $1-k_j^2$, and we know $c_j-b_j=\pm(1-k_j^2)$ from our discussion on the recurrence relations of the functions $A_n$ and $B_n$. Thus $1-k_j^2$ must divide $2c_j$, and since $\vec x_j$ has to be primitive, the only possible options are $k_j\in\{-1,0,1\}$. In any of these cases we see that $k_j(1-k_j^2)=0$, which implies $k_1^2b_1=0$ by Equation \eqref{pev4}. Note that $b_1\neq 0$ if and only if $q_1\neq 0$, so we must have $k_1=0$ if $q_1\neq 0$. Consequently, we see that we must have any simple $B_3$-coloring on such a pretzel knot must be constant in the leftmost twist region.
\end{proof}

\subsection{Lifting Simple $S_4$-Colorings}\label{fourfold}

Consider the homomorphism\footnote{We leave it to the reader to verify that this is indeed a homomorphism.} $\varrho:B_4\rightarrow B_3$ defined on the standard generators by $$\sigma_1\mapsto \sigma_1,\quad\sigma_2\mapsto \sigma_2,\quad \sigma_3\mapsto \sigma_1.$$  Note that if we have a simple $B_4$-coloring $\varphi:\pi_1(S^3\setminus L)\rightarrow B_4$, by composing with $\varrho$, we obtain a simple $B_3$-coloring $\varrho\circ\varphi:\pi_1(S^3\setminus L)\rightarrow B_3$.

We also get a similar homomorphism $\rho:S_4\rightarrow S_3$ defined by the same  formulas (except we now think of $\sigma_i$ as the transposition $(i,i+1)$) which is covered by $\varrho$. The map $\rho$ corresponds to quotienting by the Klein four subgroup $\{Id,(12)(34),(13)(24),(14)(23)\}$. Similar to above, given a simple $S_4$-coloring $\phi:\pi_1(S^3\setminus L)\rightarrow S_4$, by composing with $\rho$, we obtain a simple $S_3$-coloring $\rho\circ\phi:\pi_1(S^3\setminus L)\rightarrow S_3$. The resulting tricoloring is easy to describe: $(12),(13)$ and $(23)$ remains unchanged, and we replace the colors  $(14),(24)$ and $(34)$ by $(23),(13)$ and $(12)$ respectively.

\begin{remark}
We observe that if we have a simple $S_4$-coloring $\phi:\pi_1(S^3\setminus L)\rightarrow S_4$ which lifts to a simple $B_4$-coloring $\varphi:\pi_1(S^3\setminus L)\rightarrow B_4$, it is necessarily the case that 
 the tricoloring $\rho\circ\phi:\pi_1(S^3\setminus L)\rightarrow S_3$ lifts to the simple $B_3$-coloring $\varrho\circ\varphi:\pi_1(S^3\setminus L)\rightarrow B_3$. Since we know certain tricolorings do not lift, we can obstruct liftings of simple $S_4$-colorings.
\end{remark}

 We will need the following facts about simple $S_4$-colorings on pretzel links, and braid colorings in a twist region.

 \begin{fact}[\cite{CK0}] A pretzel link $P(q_1,...,q_m)$ is simple $S_4$-colorable if and only if two of the $q_i$'s are multiples of 3, and another $q_i$ is congruent to one of 0, 2, 3, or 4 modulo 6.
 
 \end{fact}
 
 Finally, we note that Theorem~\ref{thm4pretB} follows from the above discussion, Theorems ~\ref{thm3pret3} and ~\ref{thm3pretm}, and analogues of Remark \ref{2braid} for simple $B_4$-colorings.
 
 
 




\section{Lifting branched coverings in higher dimensions}
\subsection{Lifting coverings of manifolds in dimension bigger than four}
It turns out that liftings of coverings of manifolds of dimension bigger than four is essentially a question about algebra. 
This is because any finitely presented group is the fundamental group of a compact manifold of any dimension greater than or equal to four. For, given any finite presentation, we can attach one handles for each of the generators and one two handle for each of the relations, that winds around the one handles according to the relation. If dimension is bigger than four, we can make sure these one handles are disjoint, and hence we obtain a manifold. This stands in marked contrast with fundamental groups of two and three dimensional manifolds, where the dimension restricts the type of fundamental groups we can come across.

\subsection{Lifting branched coverings over higher dimensional spheres}
\subsubsection{Colorings of Fox's Example 12}
In Example 12 in \cite{F2}, Ralph Fox proved a conjecture of Morton Curtis, by showing that there is a 2-knot (i.e. smoothly knotted 2-sphere in $S^4$)
for which the knot group has torsion.

\begin{figure}[H]
    \centering
    \includegraphics[width=14 cm]{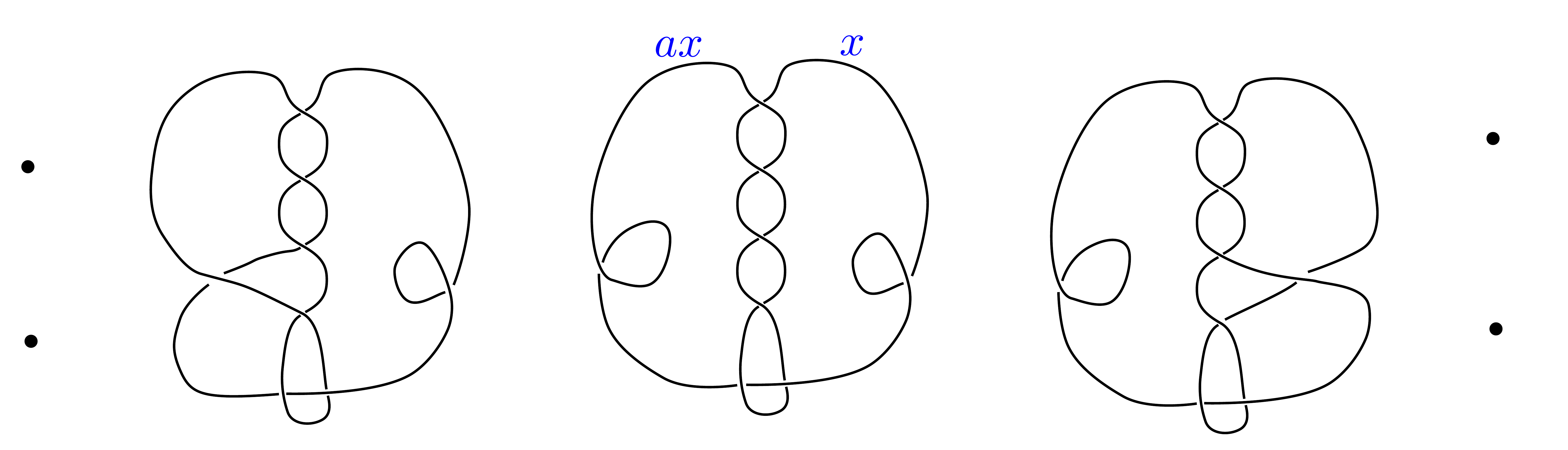}
    \caption{Fox's Example 12: A motion picture description of a 2-knot. Starting with the equatorial cross section (the middle picture), note that the result of attaching a band gives rise to two component unlinks, which can be then capped off.}
    \label{fox12}
\end{figure}

See Figure \ref{fox12} for a motion picture description of the 2-knot. It has knot group 
$$G=\langle x, a \enspace| \enspace xa^2=ax, a^2x=xa\rangle.$$
Note that in $G$ we have $a^2(xa^2)=a^2 (ax)$, and so $a^2(xa)=a^3x$, thus $a^3=1$, whence there is torsion in $G$. Given ths relation, the two relations in the presentation of $G$ become equivalent to $axa=x$. Thus we  have $$G=\langle x, a \enspace| \enspace a^3=1, axa=x\rangle.$$
Consider the tricoloring in Figure \ref{fox12tc},  of Fox's Example 12 defined by (observe that all meridians are conjugate to $x$)
$$a\mapsto (123), \enspace x\mapsto(23).$$

\begin{figure}[H]
    \centering
    \includegraphics[width=14 cm]{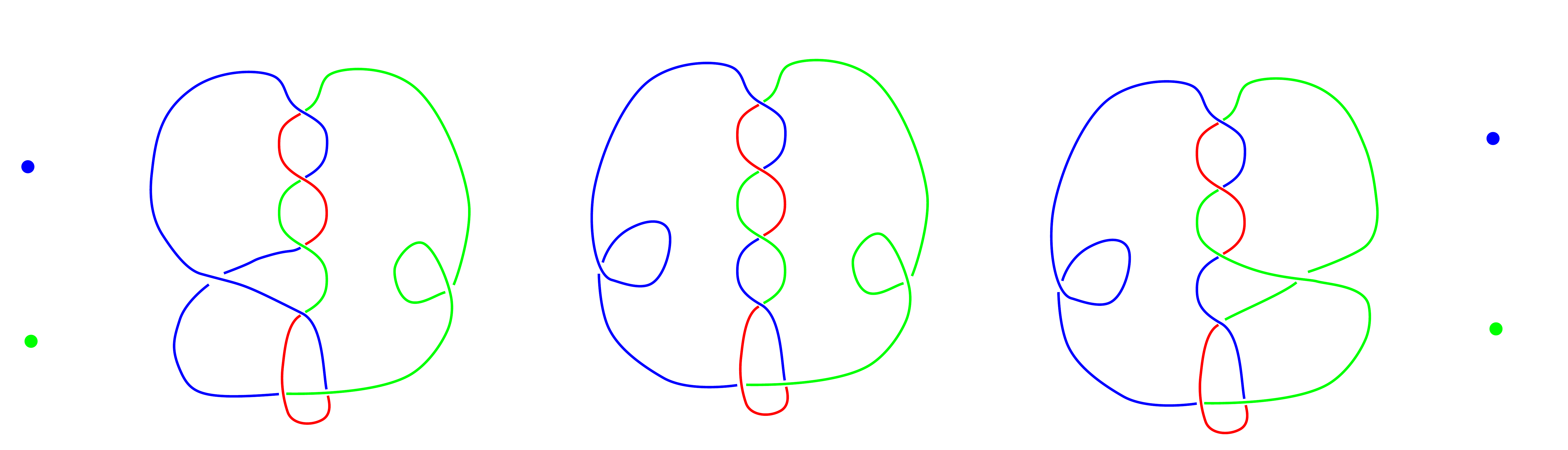}
    \caption{A tricoloring on Fox's Example 12.}
    \label{fox12tc}
\end{figure}
Since $a$ is torsion in $G$, we see that this tricoloring does not lift to a $B_3$-coloring, by Claim \ref{torsionobs}. Note that we do not need to add the constraint that the braid coloring be simple, meaning the branched cover associated to this coloring does not lift to any (even possibly non-locally flat topological) braided embedding.

\subsubsection{Non-liftable branched covers in dimension bigger than 4}
As there are other 2-knot groups with torsion \cite[Section 15.4]{Hil}
, and $n$-knot groups with $n\geq 3$ (Kervaire characterization) with torsion, we have lots of branched covers over $S^n$, with $n\geq 4$, which do not lift to braided embeddings. Once we find a non-liftable branched cover, we can get such families in all higher dimension, using Artin's spinning construction \cite{Ar2}, as explained below.

Recall that given any codimension 2 embedded connected submanifold $K$ in $\mathbb{R}^n$ (equivalently $S^n$), we can delete a standard ball $(B^n,B^{n-2})$ pair from $(\mathbb{R}^n,K)$  and we end up with a properly embedded submanifold $K_1$ with boundary in the right half plane $\mathbb{R}^n_+$. By spinning this half plane about an axis we get $\mathbb{R}^{n+1}$, and $K_1$ sweeps out a submanifold $S(K)$, called the spun knot of $K$. $S(K)$ is a knot (in the same category as the original knot $K$)
with the same knot group as $K$. 

Since the knot groups are isomorphic, there is a canonical bijection between colorings on these knots, and thus a branched coring on a knot lifts to a braided embedding if and only if the corresponding branched cover on the spun knot lifts to a braided embedding. Hence if we start with any non-liftable branched cover over a surface knot $K$ in $S^4$ (for instance the one coming from the tricoloring in Fox's example 12 mentioned in the previous subsection), we see that for any $m\in \mathbb{N}$, there is a branched cover (coming from the corresponding coloring) on the spun knot $S_m(K)$ in $S^{4+m}$ which does not lift to a braided embedding.

\end{document}